\documentclass[reqno]{conm-p-l}
\usepackage{amsmath,amssymb,amsthm,amsfonts,mathdots,amscd,bbm}
\usepackage[numeric]{amsrefs}
\input colordvi

\swapnumbers

\newtheorem{thm}[equation]{Theorem}
\newtheorem{cor}[equation]{Corollary}
\newtheorem{lem}[equation]{Lemma}

\newtheorem{prop}[equation]{Proposition}

\newcommand{\secref}[1]{section~\ref{#1}}

\DeclareMathOperator{\sg}{sgn}

\DeclareMathOperator{\tr}{tr}

\DeclareMathOperator{\covol}{covol}

\numberwithin{equation}{section}

\renewcommand\b{\beta}
\newcommand\g{\gamma}
\renewcommand\d{\delta}
\newcommand\e{\varepsilon}
\renewcommand\l{\lambda}
\renewcommand\L{\mathcal L}

\newcommand\G{\Gamma}
\newcommand\f{\frac}
\newcommand\smallf[2]{{\textstyle{\frac{#1}{#2}}}}
\newcommand\srel[2]{\begin{smallmatrix} {#1} \\ {#2} \end{smallmatrix}}

\newcommand{\ndiv}{\nmid}

\newcommand{\Z}{{\mathbb{Z}}}
\newcommand{\R}{{\mathbb{R}}}
\newcommand{\RP}{{\mathbb{RP}}}
\newcommand{\C}{{\mathbb{C}}}
\newcommand{\A}{{\mathbb{A}}}
\newcommand{\Q}{{\mathbb{Q}}}

\newcommand\re{\text{Re~}}

\renewcommand\Re{\text{Re~}}

\renewcommand\O{{\mathcal O}}

\renewcommand\i{^{-1}}
\renewcommand\({\left(}
\renewcommand\){\right)}
\newcommand{\ttwo}[4]{
\(\begin{smallmatrix}{#1} & {#2}
\\ {#3} & {#4} \end{smallmatrix}\)}
\newcommand{\bx}{\hfill$\square$\vspace{.6cm}}
\newcommand{\sgn}{\operatorname{sgn}}
\newcommand{\isobar}{\boxplus}
\newcommand{\gobble}[1]{}
  \newcommand{\rangeref}[2]{%
    \ref{#1}--\afterassignment\gobble\fam 0\ref{#2}%
  }
\makeatletter
\def\imod#1{\allowbreak\mkern5mu({\operator@font mod}\,#1)}
\makeatother

\begin{document}

\title[Adelization of Automorphic Distributions and Eisenstein Series]{Adelization of Automorphic Distributions and Mirabolic Eisenstein Series}

\author{Stephen D. Miller}
\address{Department of Mathematics, Hill Center -- Busch Campus, Rutgers, The State University of New Jersey, 110 Frelinghuysen Rd,  Piscataway, NJ 08854-8019}
\email{miller@math.rutgers.edu}
\thanks{Partially supported
by NSF grant DMS-0901594 and an Alfred P. Sloan Foundation Fellowship.}

\author{Wilfried Schmid}
\address{Department of Mathematics, Harvard University, Cambridge, MA 02138}
\email{schmid@math.harvard.edu}
\thanks{Partially supported by DARPA grant HR0011-04-1-0031 and
NSF grant DMS-0500922}

\subjclass{ }
\date{April 6, 2011}

\dedicatory{Dedicated to Gregg Zuckerman on his 60th birthday}


\begin{abstract}

Automorphic representations can be studied in terms of the embeddings of abstract models of  representations  into
spaces of functions on  Lie groups that are invariant under   discrete subgroups.  In this paper we describe an
adelic framework to describe them for the group $GL(n,\R)$, and provide a detailed analysis of the automorphic
distributions associated to the mirabolic Eisenstein series.  We give an explicit functional equation for some
distributional pairings involving this mirabolic Eisenstein distribution, and the action of intertwining operators.
\end{abstract}

\maketitle

\section{Introduction}

Ever since the Poisson integral formula, the principal of recovering an eigenfunction from its   ``boundary
values''   (which are in general distributions) has been a useful tool in analysis.  For automorphic
forms, which are eigenfunctions of a ring of invariant differential operators, the boundary values can
alternatively be described in terms of embeddings of models of representations into spaces of functions, embeddings
which share the invariance of the automorphic forms.  These automorphic distributions then control an entire
automorphic representation in terms of a single object.

In previous papers we have applied automorphic distributions to studying summation formulas and the analytic
continuation of $L$-functions \cite{voronoi,glnvoronoi,korea}, mainly for the full level congruence subgroup
$GL(n,\Z)\subset GL(n,\R)$. In this paper we present automorphic distributions in an adelic setting, in order to
use them for general congruence subgroups.  We also provide a thorough treatment of the automorphic distributions
for a special but prominent type of Eisenstein series, the mirabolic Eisenstein series for the congruence subgroup
$\G_0(N)\subset GL(n,\Z)$. We derive a precise form of their Fourier expansions, which also gives the analytic
continuation of this mirabolic series, and prove an intertwining relation that is analogous to a functional
equation. We also show that these properties extend to a relevant automorphic pairing established in
\cite{pairingpaper} that involves these mirabolic Eisenstein distributions. In our forthcoming paper
\cite{extsqpaper} this pairing will be calculated as the exterior square $L$-function times a precise ratio of
Gamma factors, thereby giving a new construction of this $L$-function that   leads to   a stronger
analytic continuation than previously known, as well as a functional equation.

The notion of adelic automorphic distribution is designed so that the action of the $p$-adic groups $GL(n,\Q_p)$
matches its usual action on adelic automorphic forms.  This has the advantage of being able to quote certain
calculations, such as local integrals, that have already been performed in related problems.  One could also
attempt stronger generalizations, which more generally treat the boundary values on a finite number of $p$-adic
groups simultaneously with those on the real group, or which extend to number fields and different groups.

Sections 2, 3, and 4 contain, respectively, some properties of cuspidal automorphic distributions, mirabolic
Eisenstein distributions, and the pairings of automorphic distributions.  These topics are then reconsidered in
section 5 using adelic terminology, which   re-expresses   them in a different notation that is
useful in many applications.  We also include an appendix   recalling the known description of
the generic unitary dual of $GL(n,\R)$,   as well as   Langlands' recipe for defining the Gamma
factors of the tensor product, symmetric square, and exterior square $L$-functions. Both   are
useful in analytic number theory, where one inputs the structure of a functional equation, and uses constraints on
the shifts in the Gamma factors to obtain estimates.

It is a pleasure to dedicate this paper to Gregg Zuckerman on his 60th birthday, as his early work on Whittaker
functions is essential to clarity with which we now understand the generic unitary dual.  The first author in
particular extends his appreciation to Zuckerman for his friendliness and helpfulness as a colleague at an early
stage in his career.  We also wish to thank Bill Casselman, Erez Lapid, and Freydoon Shahidi for helpful
discussions, and the referee for a careful reading of the paper.

\section{Automorphic Distributions}
\label{autdistsec}

 In this section we recall the
notion of automorphic distribution.  We let $G$ denote the group of real points of a reductive matrix group defined
over $\Q$, and $\G\subset G$ an arithmetic subgroup. The particular examples that will matter to us are
$G=GL(n,\R)$, and a rational conjugate of a congruence subgroup\footnote{The principal congruence subgroup
$\G(m)\subset GL(n,\Z)$ is the kernel of the reduction map from $GL(n,\Z)$ to $GL(n,\Z/m\Z)$.  A congruence
subgroup is one which contains $\G(m)$ for some $m$.  For $n>2$, they are precisely the finite index subgroups.}
$\G \subset GL(n,\Z)$.  We let $Z_G=$ denote the center of $G$, and fix a unitary central character
\begin{equation}
\label{centchar1} \omega \, : \, Z_G\ \longrightarrow \
\{\,z\in\C^*\ \mid\ |z|=1\,\}\,.
\end{equation}
Then $G$ acts unitarily, by right translation, on the Hilbert space
\begin{equation}
\label{centchar2}
\gathered L^2_\omega(\G\backslash G) \  \  =
\qquad\qquad\qquad\qquad\qquad\qquad\qquad\qquad\qquad\qquad\qquad\qquad\qquad\ \ \
 \\ \ \{\, f\in
L^2_{\text{loc}}(\G\backslash G)\ \mid \ \int_{\G\backslash G/Z_G}
|f|^2\,dg < \infty\ \,\text{and} \, \ f(gz)=\omega(z)f(g),\, z\in Z_G\}\,.
\endgathered
\end{equation}
Automorphic distributions are associated to classical\footnote{As distinguished from adelic automorphic
representations.} automorphic representations, i.e., to $G$-invariant unitary embeddings
\begin{equation}
\label{autodist1} j\, :\, V\ \hookrightarrow \
L^2_\omega(\G\backslash G)
\end{equation}
of an irreducible unitary representation $(\pi,V)$ of $G$. The space of $C^\infty$ vectors $V^\infty\subset V$ is
dense in $V$, and carries a canonical Frech\'et topology. The linear map\begin{equation} \label{autodist2} \tau\ =
\ \tau_j\, :\, V^\infty \ \longrightarrow\ \C\,,\qquad \tau(v)\ = \ j(v)(e)\,,
\end{equation}
is well defined and $\G$-invariant because $j$ maps $V^\infty$ to $C^\infty(\G\backslash G)$. It is also continuous
with respect to the topology of $V^\infty$, and thus may be regarded as a $\G$-invariant {\it distribution vector}
for the dual unitary representation $(\pi',V')$,
\begin{equation}
\label{autodist3} \tau\ \in\  \left((V')^{-\infty}\right)^\G.
\end{equation}
This is the {\it automorphic distribution} corresponding to the automorphic representation (\ref{autodist1}). The
former determines the latter completely:~for $v\in V^\infty$ and $g\in G$,
\begin{equation}
\label{autodist4} j(v)(g)\ = \ j(\pi(g)v)(e)\ = \ \langle \tau ,
\pi(g)v\rangle\ = \ \langle \pi'(g^{-1})\tau, v\rangle\,,
\end{equation}
so one can reconstruct the functions $j(v)$, $v\in V^\infty$, in terms of $\tau$\,; because of the density of
$V^\infty$ in $V$, $\tau$ determines $j(v)\in L^2_\omega(\G\backslash G)$ for all vectors $v\in V$.

  In the following, we shall also consider automorphic distributions that do not correspond to irreducible
summands of $L^2_\omega(\G\backslash G)$, as in (\ref{autodist1}). These are $\G$-invariant distribution vectors
for admissible representations of finite length which need not be unitary, in particular the distribution analogues
of Eisenstein series.

Most traditional approaches to automorphic forms work with finite dimensional $K$-invariant spaces of {\it
automorphic functions}, meaning collections of functions $\{j(v)\}$ with $v$ ranging over a basis of a finite
dimensional, $K$-invariant subspace of $V$; here $K\subset G$ denotes a maximal compact subgroup. Finite
dimensional, $K$-invariant subspaces necessarily consist of $C^\infty$ vectors, so these automorphic functions are
smooth. When $(\pi,V)$ happens to be a spherical representation, it is natural to consider the single automorphic
function $j(v_0)$ determined by the~-- unique, up to sca\-ling~-- $K$-fixed vector $v_0\in V$, $v_0\neq 0$. In that
case $j(v_0)$ can be interpreted as a $\G$-invariant function on the symmetric space $G/K$. For non-spherical
representations, typically no such canonical choice exists, and   making a definite choice   may in
fact be delicate. In the theory of integral representations of $L$-functions, for example, a wrong choice may
result in an integral being identically zero instead of the $L$-function one is interested in, or it may result in
an archimedean integral that is more difficult to compute, possibly even not computable at all \cite[\S
2.6]{bumprssurvey}.  By working directly with the automorphic distribution $\tau$, our approach avoids these
issues; in particular it does not matter whether $(\pi,V)$ is spherical or not.

Results of Casselman \cite{Casselman:1980} and Casselman-Wallach \cite{Casselman:1989,Wallach:1983} imply that
$(V')^{-\infty}$ can be realized as a closed subspace of the space of distribution vectors for a
not-necessarily-unitary principal series representation,
\begin{equation}
\label{autodist6} (V')^{-\infty}\ \hookrightarrow\
V_{\l,\d}^{-\infty}\,;
\end{equation}
the subscripts $\l,\d$ refer to the parameters of the principal series and will be explained shortly. Thus
\begin{equation}
\label{autodist7} \tau\ \in\  \left(V_{\l,\d}^{-\infty}\right)^\G
\end{equation}
becomes a $\G$-invariant distribution vector for a principal series representation\footnote{This convention differs
slightly from our earlier papers \cite{voronoi,korea}, where we had  switched the role of $(\pi,V)$ and $(\pi',V')$
at this stage for notational convenience.  However, that switch causes a notational inconsistency for our adelic
automorphic distributions  in  section~\ref{adelize} that we have elected to avoid.}   with parameters
$(\l,\d)$.   The embedding (\ref{autodist6}) is equivalent to the representation $V$ being a quotient of
the dual principal series representation $V_{-\l,\d}$.

In describing the principal series, we specialize the choice of $G$ to keep the discussion concrete,
\begin{equation}
\label{ps1} G\ = \ GL(n,\R).
\end{equation}
Its two subgroups
\begin{equation}
\label{ps2}
\begin{aligned}
&B\ = \  \left\{ \left. \left(\begin{smallmatrix}
b_1 & 0 & {\textstyle\dots} & 0 \\
{ }_{\scriptstyle *} & b_2 & {\textstyle\dots}   & 0\\
\vdots & \vdots &  \ddots    & \vdots\\
{ }_{\scriptstyle *} & { }_{\scriptstyle *} &{\textstyle\dots} &
b_n
\end{smallmatrix}\right) \ \ \right| \ \ b_j\in \R^*,\ \ 1\leq j\leq n\ \right\}\ ,
\\
{\  }
\\
&\qquad\qquad N\ = \  \left\{\ \left(\begin{smallmatrix}
1 & {\scriptstyle *} & {\textstyle\dots} & {\scriptstyle *} \\
0 & 1 & {\textstyle\dots}   & {\scriptstyle *} \\
\vdots & \vdots &  \ddots    & \vdots\\
0 & 0 &{\textstyle\dots} &   1
\end{smallmatrix}\right) \  \right\}
\end{aligned}
\end{equation}
are, respectively, maximal solvable and maximal unipotent. The quotient
\begin{equation}
\label{ps3} X\ = \ G/B
\end{equation}
is compact, and is called the {\it flag variety} of $G$. Since $N$ acts freely on its orbit through the identity
coset in $X=G/B$ and has the same dimension as $X$,  one can identify $N$ with an dense open subset of the flag
variety,
\begin{equation}
\label{ps4} N \ \ \simeq \ \ N\cdot eB\ \ \hookrightarrow \ \ X\,.
\end{equation}
This is the {\it open Schubert cell\/} in $X$.

The principal series is parameterized by pairs $(\l,\d)\in \C^n \times (\Z/2\Z)^n$. For any such pair, we define
the character
\begin{equation}
\label{ps5}
\begin{aligned}
&\chi_{\l,\d} \, :\, B\  \longrightarrow\  \C^*\,,
\\
&{\ }
\\
&\chi_{\l,\d} \left(\begin{smallmatrix}
b_1 & 0 & {\textstyle\dots} & 0 \\
{ }_{\scriptstyle *} & b_2 & {\textstyle\dots}   & 0\\
\vdots & \vdots &  \ddots    & \vdots\\
{ }_{\scriptstyle *} & { }_{\scriptstyle *} &{\textstyle\dots} &
b_n
\end{smallmatrix}\right)\ \ = \ \ {\prod}_{j=1}^n \left( (\sgn b_j)^{\d_j} |b_j|^{\l_j}\right)\,.
\end{aligned}
\end{equation}
The parametrization also involves the quantity
\begin{equation}
\label{ps6} \rho\ \ = \ \ \left(\textstyle
\frac{n-1}2,\,\textstyle \frac{n-3}2,\, \dots,\, \textstyle
\frac{1-n}2 \right)\ \in \ \C^n\,.
\end{equation}
Each pair $(\l,\d)$ determines a $G$-equivariant $C^\infty$ line bundle $\L_{\l,\d}\to X$, on whose fiber at the
identity coset the isotropy group $B$ acts via $\chi_{\l,\d}$. By pullback from $X=G/B$ to $G$, the space of
$C^\infty$ sections becomes naturally isomorphic to a space of $C^\infty$ functions on $G$,
\begin{equation}
\label{ps7} C^\infty(X, \L_{\l,\d}) \ \simeq  \ \{\,f\!\in\! C^\infty(G)
\mid f(gb)=\chi_{\l,\d}(b^{-1})f(g)\ \,\text{for}\,\ g\in G,\,
b\in B\,\}.
\end{equation}
This isomorphism relates the translation action of $G$ on sections of $\L_{\l,\d}$ to left translation of
functions. By definition,
\begin{equation}
\label{ps8} V_{\l,\d}^\infty \ \ = \  \ C^\infty(X, \L_{\l-\rho,\d})
\end{equation}
is the space of $C^\infty$ vectors of the principal series representation $V_{\l,\d}\,$; the shift by $\rho$ serves
the purpose of making the labeling compatible with Harish-Chandra's parametrization of infinitesimal characters.
Analogously
\begin{equation}
\label{ps9}
\begin{aligned}
V_{\l,\d}^{-\infty} \ &= \ C^{-\infty}(X, \L_{\l-\rho,\d})
\\
&\simeq \ \{f\!\in\! C^{-\infty}(G) \mid
f(gb)=\chi_{\l-\rho,\d}(b^{-1})f(g)\ \text{for}\ g\!\in\! G,
b\!\in\! B\}
\end{aligned}
\end{equation}
is the space of distribution vectors. The isomorphism in the second line is entirely analogous to (\ref{ps7}).

The group $N$, which we had identified with the open Schubert cell, intersects $B$ only in the identity. Thus, when
the equivariant line bundle $\L_{\l-\rho,\d}\to X$ is restricted to the open Schubert cell, it becomes canonically
trivial, and distribution sections of the restricted line bundle become scalar-valued distributions,
\begin{equation}
\label{ps10} C^{-\infty}(N, \L_{\l-\rho,\d}) \   = \
C^{-\infty}(N)\,.
\end{equation}
This identification is $N$-invariant, of course. In particular any automorphic distribution
\begin{equation}
\label{ps11} \tau  \ \in \  (V_{\l,\d}^{-\infty})^\G \ = \
C^{-\infty}(X, \L_{\l-\rho,\d})^\G
\end{equation}
restricts to a $\G\cap N$-invariant distribution on the open Schubert cell:
\begin{equation}
\label{ps12} \tau  \ \in \   C^{-\infty}\bigl(\G \cap N\backslash
N\bigr).
\end{equation}
Two comments are in order. Ordinarily, a distribution on a manifold is not completely determined by its restriction
to a dense open subset. Since the $\G$-translates of the open Schubert cell cover $X$, any automorphic distribution
{\it is\/} determined by its restriction to $N$. The containment (\ref{ps12}) should be interpreted in this sense.
Secondly, when one views $\tau$ this way, the invariance under $\G\cap N$ is directly visible. The invariance under
any $\g\in \G$ that does not lie in $N$ can be described in terms of an appropriate factor of automorphy.

The abelianization $N/[N,N]$ -- i.e., the quotient of $N$ by the derived subgroup $[N,N]$ -- is isomorphic to the
additive group $\R^{n-1}$. Concretely, let
\begin{equation}
\label{ab1} n(x)\ = \  \left(\begin{smallmatrix}
1 & \ x_1 & 0 & {\textstyle\dots} & 0 \\
{ }_{\scriptstyle 0} & {\ }_{\scriptstyle 1} & {\ }_{\scriptstyle x_2} & {\textstyle\dots}   & { }_{\scriptstyle 0}\\
\vdots & { }^{\scriptstyle 0} & { }^{\scriptstyle 1} & \ddots    & \vdots\\
\vdots & \vdots & \vdots & \ \ \ddots & {\ }^{\scriptstyle x_{n-1}} \\
{ }_{\scriptstyle 0} & {\ }_{\scriptstyle 0} & { }_{\scriptstyle
0} &{\textstyle\dots} &   { }_{\scriptstyle 1}
\end{smallmatrix}\right) \ \ \   (\, x= (x_1,x_2,\dots ,x_{n-1})\in \R^{n-1} \,)\,;
\end{equation}
then $\,\R^{n-1} \simeq N/[N,N]$ via
\begin{equation}
\label{ab2} \R^{n-1} \ni x \ \longmapsto \  \text{image of} \ n(x)
\in N/[N,N]\,.
\end{equation}
A congruence subgroup $\G \subset G$ intersects $N$ in a cocompact subgroup of $N$, and similarly $[N,N]$ in a
cocompact subgroup of itself. This allows us to define
\begin{equation}
\label{ab3}
\tau_{\text{abelian}}\ = \   \frac{1}{\covol(\G\cap [N,N])}    \int_{  (\G\cap [N,N])   \backslash [N,N]} \ell(n)\tau \,dn \,,
\end{equation}
  the sum of   the {\it abelian Fourier component\/} of the automorphic distribution $\tau$, as in
(\rangeref{ps11}{ps12}); $\ell(n)$ denotes left translation by $n$. Equivalently
\begin{equation}
\label{ab4} \tau\ \ = \ \ \tau_{\text{abelian}}\  + \ \cdots\ \ ,
\end{equation}
where $\ \cdots\ $ refers to the sum of Fourier components of $\,\tau$ on which $[N,N]$ acts non-trivially. By
construction,   $\, \tau_{\text{abelian}}\in V_{\l,\d}^{-\infty}\,$,\,\ and the restriction of $\,
\tau_{\text{abelian}}\,$ to $N$ lies in $C^{-\infty}\bigl(\bigl([N,N]\cdot (\G\cap N)\bigr)\backslash N
\bigr)$.

The quotient $\bigl([N,N]\cdot (\G\cap N)\bigr)\backslash N$ is compact, connected, abelian, hence a torus. Like
any distribution on a torus, $\, \tau_{\text{abelian}}$ can be expressed as an infinite linear combination of
characters.  We may write
\begin{equation}
\label{ab6a} \tau_{\text{abelian}}(n(x))\ = \ {\sum}_{k\in
\Q^{n-1}}\ c_k\,e(k_1\,x_1 + k_2\,x_2\, + \dots +
k_{n-1}\,x_{n-1})\,
\end{equation}
in which  the coefficients $c_k$   are tacitly assumed to   vanish unless $k$ lies inside $M\i\Z$,
for some appropriate integer $M$ (which takes into account the size of the torus).  Here, as from now on, we use
the notational convention
\begin{equation}
\label{ab7} e(z)\ =_{\text{def}}\ e^{2\pi i z}\,.
\end{equation}
In the case that $\G$ equals the full level congruence group $GL(n,\Z)$, $\G\cap N=N(\Z)$ and $k$ lies in
$\Z^{n-1}$, because  the isomorphism (\ref{ab2}) induces $\bigl([N,N]\cdot (\G\cap N)\bigr)\backslash N\simeq
\Z^{n-1} \backslash \R^{n-1}$.

Recall the notion of a cuspidal automorphic representation:~an automorphic representation in the same sense as
(2.3), such that
\begin{equation}
\label{cuspidal1o} \int_{(\G\cap N)\backslash N}\ j(v)(ng)\,dn\ = \
0\ \ \text{for every $v\in V^\infty$, $g\in G$,}
\end{equation}
whenever $N\subset G$ is the unipotent radical of a proper parabolic subgroup, defined over $\Q$. We call an
automorphic distribution $\tau\in (V^{-\infty})^\G$ cuspidal if the corresponding automorphic representation has
that property; this is equivalent to
\begin{equation}
\label{cuspidal2o} \int_{N/(\G\cap N)}\ \ell(n)\tau\,dn\,\ =\, 0
\end{equation}
for every $N$ as in (\ref{cuspidal1o}) \cite[Lemma~2.16]{decay}. In our particular setting of $GL(n)$ the
cuspidality of $\tau$   implies
\begin{equation}
\label{cuspidal1}
k\in \Q^{n-1}\,, \ k_j = 0 \,\ \text{for at least one $j$, $1\leq j\leq n-1$}\ \ \Longrightarrow\ \ c_k=0\,,
\end{equation}
as can be seen by averaging the $u$-translates of $\tau$ over $U_{j,n-j}(\Z)\backslash U_{j,n-j}$, the quotient of
the unipotent radical of the $(j,n-j)$ parabolic modulo its group of integral points.   However, the
cuspidality of $\tau$ cannot be characterized solely in terms of the vanishing of certain Fourier coefficients at
each cusp; it also involves conditions ``at infinity"~-- see, for example, \cite[\S5]{voronoi}.

The Casselman embedding (\ref{autodist6}) does not necessarily determine the parameters $(\l,\d)$ uniquely. For
example, when $V_{\l,\d}$ is an irreducible principal series representation, $(\l,\d)$ is determined only up to the
action of the Weyl group. The abelian Fourier coefficients $c_k$, $k\in \Q^{n-1}$, do depend on the choice of
Casselman embedding. When $\tau$ is cuspidal, one can introduce its {\it renormalized Fourier coefficients}
\begin{equation}
\label{ab8} a_{(k_1,k_2,\dots,k_{n-1})}\, = \,
\prod_{j=1}^{n-1}\left( (\sgn k_j)^{\d_1 + \d_2 + \dots + \d_j}
\,|k_j|^{\l_1 + \l_2 + \dots + \l_j}\right)
c_{(k_1,k_2,\dots,k_{n-1})}\,,
\end{equation}
which have canonical meaning. The $L$-functions of $\,\tau$ can be most naturally expressed in terms of the $a_k$.
For $k$ coprime to a finite set of primes depending on $\tau$, the $a_k$ are actually the eigenvalues of certain
Hecke operators $T_k$ acting on the automorphic representation,   provided the Hecke action preserves
the automorphic representation. This applies to   all $k$ when $\G=GL(n,\Z)$, demonstrating that the $a_k$
are independent of the particular Casselman embedding. This independence can also be shown directly, without
reference to Hecke operators -- meaning that this independence holds for congruence subgroups $\G$ as well.  We
shall see this from a different point of view later in \secref{adelize}, in terms of adelic Whittaker functions.

The terms in (\ref{ab6a}) have a canonical extension from the big Schubert cell $N$ to $G/B$ (i.e., the opposite of
the restriction in (\ref{ps10}-\ref{ps12})); see \cite{chm}, where this issue is considered and resolved in greater
generality.  Let us consider the canonical extension of the additive character $n(x)\mapsto e(x_1+\cdots +
x_{n-1})$ in (\ref{ab6a}), which we will call the ``Whittaker distribution'' $w_{\l,\d}\in V_{\l,\d}^{-\infty}$ to
emphasize its dependence on the principal series parameters. Its restriction to the big Bruhat cell $NB\subset G$
is determined by the transformation formula
\begin{equation}\label{btransform}
\gathered
    w_{\l,\d}\(\(\begin{smallmatrix}
      1 & x_{1} & \star & \star  & \star          \\
      & 1     & x_{2}& \star & \star \\
       & & \ddots  & \star & \star \\
       & & & 1 & x_{n-1} \\
       & & & & 1
    \end{smallmatrix}\)\(\begin{smallmatrix}
      b_1 &  & & &          \\
      \star & b_2      &  & &  \\
      \star  & \star & \ddots  & & \\
        \star &  \star &  \star & b_{n-1} & \\
       \star &  \star&  \star&  \star& b_n
    \end{smallmatrix}\)\) \qquad\qquad\qquad\qquad\qquad\qquad
    \\ \qquad\qquad\qquad\qquad = \ \ e(x_1+\cdots +
    x_{n-1})\,{\prod}_{j=1}^n
    |b_j|^{(n+1)/2-j-\l_j}\sgn(b_j)^{\d_j}\,.
\endgathered
\end{equation}
  For $k=(k_1,k_2,\dots,k_n)$, let $D(k)=\operatorname{diag}(k_1\cdots k_{n-1},k_2\cdots k_{n-1},\ldots,
k_{n-1},1)$ denote the diagonal matrix with diagonal entries $k_1\cdots k_{n-1},k_2\cdots k_{n-1},\ldots,
k_{n-1},1$. If each $k_j\neq 0$ (as is automatically true for the indices corresponding to a cuspidal $\tau$)
con\-ju\-gation by $D(k)$ transforms the character $n(x) \mapsto e(x_1 + \dots + x_{n-1})$ into the character
$n(x)\mapsto e(k_1x_1+\cdots + k_{n-1}x_{n-1})$. The canonical extension of the latter is therefore given by
\begin{equation}
\begin{aligned}
\label{canonexttauabel-1}
    &w_{\l,\d}\left(D(k) g D(k)^{-1}\right) \ \ =
    \\
&\qquad\qquad = \ \ w_{\l,\d}\left(D(k) g\right) {\prod}_{j=1}^{n-1}\left( {\prod}_{i=j}^{n-1} |k_i|^{-(n+1)/2+ j + \l_j} \,{\prod}_{i=1}^j k_i^{\d_j} \right) \,.
\end{aligned}
\end{equation}
In view of (\ref{ab6a}) and (\ref{ab8}), the canonical extension of $\tau_{\text{abelian}}$ to $G$ can be written
as
\begin{equation}
\label{canonexttauabel}
    \tau_{\text{abelian}}(g) \ \ = \ \
{\sum}_{k\,\in\,\Q^{n-1}}  \f{a_k}{\left| \prod_{j=1}^{n-1}
k_j^{j(n-j)/2} \right|}\,w_{\l,\d}(D(k) g)\,.
\end{equation}
One then also has the following  equality between distributions on $G$:
\begin{equation}\label{taucoeffbform}\gathered
 \f{1}{\covol(\Gamma\cap N)} \,  \int_{\Gamma\cap N\backslash N} \tau(u g) \, e(-k_1 \,u_{1,2}- \cdots  -
    k_{n-1}\,u_{n-1,n}) \,du  \qquad \ \ \ \  \\
    \qquad \qquad \qquad\qquad\qquad\qquad \qquad
     = \ \ \f{a_k}{\left| \prod_{j=1}^{n-1}
k_j^{j(n-j)/2} \right|}\,w_{\l,\d}(D(k) g)\,,
\endgathered
\end{equation}
  where $u_{i,j}$ denote the entries of $u\in N$ and $\operatorname{covol}(\Gamma\cap N)$ denotes the volume of the quotient $\Gamma\cap N\backslash N$ under the Haar measure $du$, normalized so that $\operatorname{covol}(N(\Z))=1$.

A number of relations involving automorphic distributions, such as the functional equations of their $L$-functions,
involve not only a particular automorphic distribution -- or equivalently, the corresponding automorphic
representation -- but also its contragredient. The map
\begin{equation}
\label{ab9} g\ \mapsto \ \widetilde g\,,\ \ \widetilde g\, = \,
w_{\text{long}}(g^t)^{-1}w_{\text{long}}^{-1}\,, \ \ \text{with}\
\ w_{\text{long}}\, = \, \left(
\begin{smallmatrix}
 & & & & 1 \\
 & & & \cdot & \\
 & & \cdot & & \\
 & \cdot &  & & \\
 1 & & & &
\end{smallmatrix}
\right),
\end{equation}
defines an outer automorphism of $G= GL(n,\R)$, which preserves the subgroups $GL(n,\Z)$, $B$ and $N$. One easily
checks that
\begin{equation}
\label{ab10} \widetilde\tau (g)\ =_{\text{def}} \tau\!\(\widetilde
g\)\ \in \ (V_{\tilde\l,\tilde\d}^{-\infty})^{\widetilde\G}\   , \ \ \text{with} \, \ \widetilde{\G} \,=\,\{\widetilde\g \,|\,\g\,\in\,\G\}\,,
\end{equation}
the contragredient of $\,\tau$, has abelian Fourier coefficients
\begin{equation}
\label{ab11} \widetilde c_{(k_1,k_2,\dots,k_{n-1})}\ = \
c_{(-k_{n-1},-k_{n-2},\dots,-k_1)}
\end{equation}
and principal series parameters
\begin{equation}
\label{ab12} \widetilde \l\ = \ (-\l_{n}, -\l_{n-1}, \dots ,
-\l_1)\,, \ \ \ \widetilde \d\ = \ (\d_{n}, \d_{n-1}, \dots ,
\d_1)\,.
\end{equation}

\section{Mirabolic Eisenstein series for $GL(n)$}
\label{eisenstein}

The Epstein   zeta functions   on $GL(n,\R)$, which are sums of powers of the norms of lattice
vectors in $\R^n$, were an early example of higher rank Eisenstein series.  They have a functional equation and
analytic continuation coming from Poisson summation, in complete analogy with the Riemann   zeta function.
  Langlands, and later Jacquet and Shalika \cite{jseuler}, studied mirabolic Eisenstein series, which are
an adelic generalization involving homogenous functions other than the norm.  They play a crucial role in the
functional equation and analytic continuation of a number of integral representations of $L$-functions,
e.g.~\cite{pspatterson,bumpginzburg,bumpfriedberg,ginzburge6}.  In this section  we describe their distributional
counterparts.   Proposition~\ref{eisonu} gives the analytic continuation  and an explicit formula for their Fourier
coefficients in terms of $L$-functions and arithmetic sums.  These  have direct applications elsewhere, most
recently to string theory where they describe fine details of graviton scattering amplitudes (see, for example,
\cite{green,pioline}).   A functional equation is given in proposition~\ref{eisfe}. The analytic properties  later
transfer to  the pairings in \secref{pairingsec}.  They are understood most easily in classical terminology; in
section~\ref{adelize} we shall convert them into adelic expressions  whose analytic properties rest on what is
proven here. It is possible to recover the results here from \cite{jseuler}, using sophisticated machinery of
Casselman and Wallach.  However,  the translation between the two is somewhat lengthy and unenlightening, and so we
have chosen to rederive them from basic principles instead, highlighting the role of degenerate principal series
and intertwining operators.

Mirabolic Eisenstein series are induced from one dimensional representations of the so-called mirabolic subgroup of
$GL(n)$, colloquially dubbed the ``miraculous parabolic''\footnote{The terminology in the literature is not
entirely consistent:~some reserve the term ``mirabolic" for the stabilizer of a line in $\R^n$, e.g. $\widetilde
P$, but not $P$.}. In fact, the functional equation involves not just one, but two different mirabolic subgroups
and Eisenstein series. The mirabolic subgroups and the ``opposites" of their unipotent radicals are
\begin{equation}
\label{mira1}
\begin{aligned}
&P = \left\{ \left. \left(
\begin{smallmatrix}
\textstyle{a} & 0 &\textstyle{\dots} & 0 \\
{ }_{\textstyle{*}} &    &  & \\
\vdots & & { }^{\textstyle{C}} &   \\
\textstyle{*} &  & &
\end{smallmatrix}
\right) \  \right|  \ C \in GL(n-1,\R),\ \, a \in \R^* \right\} \,,\\
 { } \\
&\qquad \widetilde P = \left\{ \left. \left(
\begin{smallmatrix}
 & & & 0 \\
 & { }^{\textstyle{C}} & & \vdots \\
 &  & & { }_{\scriptstyle{0}} \\
 { }_{\textstyle{*}} & \textstyle{\dots} & { }_{\textstyle{*}} &
 { }_{\textstyle{a}}
  \end{smallmatrix}
\right)\   \right| \  C \in GL(n-1,\R),\ \, a \in \R^* \right\} \,,\\
 { } \\
&\qquad\qquad U =  \left\{  \left(
\begin{smallmatrix}1& {\textstyle{*}} & \textstyle{\dots}  & {\textstyle{*}} \\
  & { }_{\scriptstyle{1}} & \ \  { }_{{ }_{\textstyle{0}}} &  \\
{ }^{\textstyle{0}} & & \ \ \ddots &  \\
 &  &  & 1
 \end{smallmatrix}
\right)  \right\}\,,\ \ \ \widetilde U =  \left\{  \left(
\begin{smallmatrix}
{ }_{\scriptstyle{1}} &  &  & {}_{\textstyle{*} } \\
 & \ddots  & {}^{\textstyle{0}} & \vdots  \\
 { }^{\textstyle{0}}  &   & 1 & {\textstyle{*}}  \\
 &  &   &   1
  \end{smallmatrix}\right)  \right\}\,;
\end{aligned}
\end{equation}
note that the outer automorphism (\ref{ab9}) relates $P$ to $\widetilde P$ and $U$ to $\widetilde U$. In analogy to
the flag variety $X=G/B$,
\begin{equation}
\label{mira2} Y\ = \ G/P\ \ \text{and}\ \ \ \widetilde Y\ = \
G/\widetilde P
\end{equation}
  are {\it generalized flag varieties}.   The former can be naturally identified with the
projective space of hyperplanes
 in $\R^n$, the latter with the projective space of lines. Since
$U\cap P = \widetilde U \cap \widetilde P = \{e\}$, we can identify $U$ and $\widetilde U$ with the open Schubert
cells in these two spaces,
\begin{equation}
\label{mira3} U\ \simeq \ U\cdot eP\ \hookrightarrow \ Y\,,\qquad
\widetilde U\ \simeq \ \widetilde U \cdot e\widetilde P\
\hookrightarrow \widetilde Y\,.
\end{equation}
This is again entirely analogous to (\ref{ps4}).

 For $\nu\in \C$  and $\e\in \Z/2\Z$, we define
\begin{equation}
\label{mira4}
\begin{aligned}
&\chi_{\nu, \e }\, :\, P\ \to \ \C^*\,, \ \ \chi_{\nu, \e } \left(
\begin{smallmatrix}
\textstyle{a} & 0 &\textstyle{\dots} & 0 \\
{ }_{\textstyle{*}} &    &  & \\
\vdots & & { }^{\textstyle{B}} &   \\
\textstyle{*} &  & &
\end{smallmatrix}\right)\ = \ |a|^{\frac {(n-1)\nu}n}  (\sgn a )^\e  |\det B|^{-\frac \nu
n}\,,
 \\
{  }
\\
&\qquad \widetilde \chi_{\nu,\e }\, :\, \widetilde P\ \to \ \C^*\,, \ \
\ \widetilde \chi_{\nu,  \e  } \left(
\begin{smallmatrix}
 & & & 0 \\
 & { }^{\textstyle{B}} & & \vdots \\
 &  & & { }_{\scriptstyle{0}} \\
 { }_{\textstyle{*}} & \textstyle{\dots} & { }_{\textstyle{*}} & { }_{\textstyle{a}}
\end{smallmatrix}\right)\ = \ |\det B|^{\frac{\nu}{n}}\,|a|^{-\frac
{(n-1)\nu}n} (\sgn a)^\e  \,.
\end{aligned}
\end{equation}
We study these two characters without any loss of generality, because they  account for all characters of $P$ and
$\widetilde{P}$, up to tensoring by central characters.   Taking these other choices amounts to multiplying our
eventual Eisenstein distributions   by $\sgn(\det{g})$, and has no analytic impact. The quantity
\begin{equation}
\label{mira5} \rho_{\text{mir}}\ =  \  \frac n2
\end{equation}
plays the role of $\rho$ in the present context.

There exist unique $G$-equivariant $C^\infty$ line bundles $\L_{\nu,\e } \to Y$, $\widetilde\L_{\nu,\e } \to
\widetilde Y$, on whose fibers at the identity cosets the isotropy groups act by, respectively, $\chi_{\nu,\e }$
and $\widetilde\chi_{\nu, \e }$. The group $G$ acts via left translation on
\begin{equation}
\label{mira6}
\begin{aligned}
W_{\nu, \e  }^\infty \ &= \ C^\infty(Y,\L_{\nu-\rho_{\text{mir}},  \e  })
\\
&\simeq \  \{f\!\in\! C^\infty(G) \mid
f(gp)=\chi_{\nu-\rho_{\text{mir}}, \e  }(p^{-1})f(g)\ \text{for}\ g\!\in\! G,
p\!\in\! P\} \,,
\\
\widetilde W_{\nu,  \e  }^\infty \ &= \ C^\infty(\widetilde Y,\widetilde
\L_{\nu-\rho_{\text{mir}},  \e  })
\\
&\simeq \  \{f\!\in\! C^\infty(G) \mid f(g \tilde
p)=\widetilde\chi_{\nu-\rho_{\text{mir}},  \e  }(\tilde p^{-1})f(g)\ \text{for}\
g\!\in\! G,\tilde p\!\in\! \widetilde P\} \,.
\end{aligned}
\end{equation}
In particular, functions $f\in W_{\nu,  \e  }^\infty$ and $\tilde{f} \in \widetilde{W}_{\nu, \e  }^\infty$ obey the
respective transformation laws
\begin{equation}
\label{mira6explicated}
\begin{aligned}
    f&\(g\ttwo{a}{}{\star}{B}\) \ = \
    |a|^{n/2-\nu}\,  (\sgn a)^\e   f(g) \  \ \text{and} \\ & \tilde{f}\(g\ttwo{B}{}{\star}{a}\) \ = \
    |a|^{\nu-n/2}\,  (\sgn a)^\e  \tilde{f}(g)\,,  \ \ \text{provided} \ \ |a||\det B|=1\,.
\end{aligned}
\end{equation}
 These are the spaces of $C^\infty$ vectors for
degenerate principal series representations $W_{\nu, \e  }$, $\widetilde W_{\nu, \e  }$.

As in the case of the principal series, the line bundle $\L_{\nu-\rho_{\text{mir}},  \e }$ is equivariantly trivial over the
open Schubert cell $U\subset Y$. Since $\d_e\in C^{-\infty}(U)$, the Dirac delta function at $e\in U$, evidently
has compact support in $U$, we may regard it as a distribution section of $\L_{\nu-\rho_{\text{mir}},  \e  }$, or in other
words, as a vector in $W_{\nu, \e  }^{-\infty}$. This makes\begin{equation} \label{mira10} \d_\infty\
=_{\text{def}}\ \ell(w_{\text{long}})\d_e \ \in \ W_{\nu,  \e }^{-\infty}
\end{equation}
well defined. By construction, $\d_\infty$ is supported at $w_{\text{long}}P \in Y$, the unique fixed point of $U$,
also known as the closed Schubert cell in $Y$. Similarly there exists a delta function
$\d_{\widetilde\infty}\in\widetilde W_{\nu,\e  }^{-\infty}$ supported on the closed Schubert cell
$w_{\text{long}}\widetilde{P}\in \widetilde Y$.

Mirabolic Eisenstein series are globally induced from a character of $P$ or $\widetilde{P}$. As for their analytic
properties, it suffices to study them for the congruence subgroups
\begin{equation}\label{gamma0n}
    \G_0(N) \ \ = \ \ \left\{\g \,\in\,GL(n,\Z)\,\mid\, \g\,\equiv \,
   \( \begin{smallmatrix}
      \star & \cdots & \star & \star \\
      \vdots & \ddots & \vdots & \vdots \\
      \star & \cdots & \star & \star \\
      0 & \cdots & 0 & \star \\
    \end{smallmatrix}\) \pmod N \right\}\ \
\end{equation}
or
\begin{equation}\label{gammatild0n}
    \widetilde{\G}_0(N) \ \ = \ \ \left\{\g \,\in\,GL(n,\Z)\,\mid\, \g\,\equiv \,
\(    \begin{smallmatrix}
      \star & \star & \cdots & \star \\
      0 & \star & \cdots & \star \\
      \vdots & \vdots & \ddots & \vdots \\
      0 & \star & \cdots & \star \\
    \end{smallmatrix}\) \pmod N \right\}\,,
\end{equation}
by means of a reduction we will discuss in \secref{adelize}. Of course $\G_0(N)$ and $\widetilde{\G}_0(N)$ are
related by the outer automorphism (\ref{ab9}).  Any Dirichlet character $\psi$ modulo $N$ lifts to characters
$\alpha$ of $\G_0(N)$ and $\widetilde{\alpha}$ of $\widetilde{\G}_0(N)$ defined   through the formulas
\begin{equation}\label{omegachar}
\alpha(\g) \ = \ \psi(\g_{nn})\i \ \ \ \text{and} \ \ \ \widetilde{\alpha}(\g) \ = \ \psi(\g_{11})\, , \ \ \ \ \ \g \, = \, (\g_{ij})\,.
\end{equation}
The reason for the inverse is to ensure $\widetilde{\alpha}(\widetilde{\g})=\alpha(\g)$, a property used below in
(\ref{twoeis}).  These characters are respectively trivial on the subgroups $\G_1(N)\subset \G_0(N)$ and
$\widetilde \G_1(N)\subset \widetilde \G_0(N)$, which are defined by the congruence $\g_{nn}\equiv 1\imod N$ in the former
case, and $\g_{11}\equiv 1 \imod N$ in the latter case.

We let $\G=\G_0(N)$  and   $\G_\infty = \G \cap w_{\text{long}} P w_{\text{long}}$ denote its  isotropy subgroup at
$w_{\text{long}}P \in Y$.  Because $-e\in \G_\infty$, we insist that $\psi(-1)=(-1)^\e$ so that  $\G_\infty$  acts
trivially on $\d_e$.  (Otherwise the Eisenstein series we presently define would be identically zero.)   With this
choice of parity parameter
 define
\begin{equation}
\label{eisen1} E_{\nu,\psi} \ \  = \  \   L(\nu+\textstyle{\f{n}{2}},\psi) \  {\sum}_{\g\in
\G/\G_\infty} \, \alpha(\g)  \, \ell(\gamma)\d_\infty  \  \in \   W_{\nu,\e }^{-\infty}\,.
\end{equation}
 For $\Re \nu > \rho_{\text{mir}}=n/2$
this sum converges in the strong distribution topology.   In the region $\{\,\Re\nu> \rho_{\text{mir}}\,\}$,\,\ the
resulting distribution vector depends holomorphically on $\nu$ and satisfies the condition $\ell(\g)E_{\nu,\psi} =
\alpha(\g)\i E_{\nu,\psi}$ for all $\g\in \G$.   Entirely analogously, with
$\widetilde{\G}=\widetilde{\G}_0(N)$,
\begin{equation}
\label{eisen2} \widetilde E_{\nu, \psi } \ \  =  \ \  L(\nu+\textstyle{\f{n}{2}},\psi) \
{\sum}_{\g\in \widetilde{\G}/ \widetilde{\G}_{\widetilde\infty}} \, \widetilde{\alpha}(\g) \,
\ell(\gamma)\d_{\widetilde\infty}\  \in\   \widetilde
W_{\nu, \e }^{-\infty}
\end{equation}
converges and depends holomorphically on $\nu$ in $\{\,\Re\nu> \rho_{\text{mir}}=n/2\,\}$. The two Eisenstein series are
related by the involution (\ref{ab9}):
\begin{equation}\label{twoeis}
E_{\nu, \psi }(g) \ \ = \ \ \widetilde{E}_{\nu, \psi}(\widetilde{g})\, .
\end{equation}
The following proposition gives a simpler formula for these Eisenstein distributions when restricted to the open,
dense Bruhat cells $U, w_{\text{long}}   U \subset Y$, and $\widetilde{U},  w_{\text{long}}   \widetilde{U}\subset \widetilde{Y}$,
respectively.  Since both Eisenstein series are  invariant under a congruence group, and the
 translates   of any of  these cells by that invariance group cover $Y$ and $\widetilde{Y}$,
respectively,  restriction to either determines   them completely.  The statement involves the finite Fourier transform
\begin{equation}\label{fft}
    \widehat\psi(m)  \ \ =  \ \ \sum_{a\imod N}\psi(a)\,e(\smallf{a m}{N})
\end{equation}
of a Dirichlet character of modulus $N$.  (Note that $\psi(a)=0$ when $a$ is not relatively prime to $N$.)

\begin{prop}\label{eisonu}   {\rm (Analytic continuation and Fourier expansion of mirabolic
Eisenstein distributions.)}     Let $\Re\nu>n/2$.  The restriction of the distribution $E_{\nu,  \psi  }$ to $U$ as well as the restriction of the distribution $\widetilde{E}_{\nu,\psi}$ to $\widetilde U$ are determined by the common formulas
\[
\aligned
&    E_{\nu,  \psi   }\left(
\begin{smallmatrix}1& -u_{n-1} & \cdots & -u_{1} \\
  & { }_{\scriptstyle{1}} & \ \  { }_{{ }_{\textstyle{0}}} &  \\
{ }^{\textstyle{0}} & & \ \ \ddots &  \\
 &  &  & 1
 \end{smallmatrix}
\right) \ \  = \ \ \widetilde{E}_{\nu,  \psi } \left(
    \begin{smallmatrix}{ }_{\scriptstyle{1}} &  &  &
    {}_{u_1} \\
 & \ddots  & {}^{\textstyle{0}} & \vdots  \\
 { }^{\textstyle{0}}  &   & 1 & {u_{n-1}}  \\
 &  &   &   1
\end{smallmatrix}
\right)
 \qquad\qquad\qquad\qquad\qquad\qquad\qquad\qquad\vspace{4pt}
\\& \  =\ \
{\sum}_{\stackrel{\scriptstyle v \, \in \, \Z^n\!,\ {\scriptstyle v_1
\, >  \, 0}}{N\mid v_1,\ldots,v_{n-1}\ }}\ \ \psi(v_n)\, v_1^{\ -\nu-n/2}\, \d_{v_n/v_1}(u_1)\cdots
\d_{v_{2}/v_1}(u_{n-1})
\\
& \ = \ \ N^{-\nu-n/2}\sum_{v\in \Z^n, \ v_1>0} v_1^{-\nu-n/2+1} e(v_1v_n u_1) \d_{v_{n-1}/v_1}(u_2)\cdots
\d_{v_2/v_1}(u_{n-1}) \, \widehat\psi(-v_n)
\\
&\  = \ \
\sum_{r\,\in\,\Z^{n-1}} a_r \, e(r_1u_1+\cdots +r_{n-1}u_{n-1})\,,
\endaligned
\]
where
\[ \ a_r \ = \
N^{-\nu-n/2} \sum_{\stackrel{\scriptstyle{d\,>\,0}}{d
\, | \, r_1,\ldots,r_{n-1}}}
      \, d^{-\nu+n/2-1}\,\widehat{\psi}(-r_1/d)\,.
\]
Their restrictions to
$w_{\text{long}}U$ and  $w_{\text{long}}\widetilde{U}$      are determined by the common formula
\[
\gathered
\big(\ell(w_{\text{long}})\,
E_{\nu,  \psi   }\big)\!\left(
\begin{smallmatrix}1& u_{n-1} & \cdots & u_{1} \\
  & { }_{\scriptstyle{1}} & \ \  { }_{{ }_{\textstyle{0}}} &  \\
{ }^{\textstyle{0}} & & \ \ \ddots &  \\
 &  &  & 1
 \end{smallmatrix}
\right) \ \  =  \ \ \big(\ell(w_{\text{long}})\,
\widetilde{E}_{\nu,  \psi }\big)\! \left(
    \begin{smallmatrix}{ }_{\scriptstyle{1}} &  &  &
    {}_{u_1} \\
 & \ddots  & {}^{\textstyle{0}} & \vdots  \\
 { }^{\textstyle{0}}  &   & 1 & {u_{n-1}}  \\
 &  &   &   1
\end{smallmatrix}
\right)  \\
 \qquad\qquad\qquad\qquad\qquad\qquad\qquad\qquad\vspace{4pt}
\\ =\ \ {\sum}_{\stackrel{\scriptstyle v \, \in \, \Z^n\!,\ {\scriptstyle v_n
\, >  \, 0}}{N\mid v_1,\ldots,v_{n-1}\ }}\ \ \psi(v_n)\, v_n^{\ -\nu-n/2}\, \d_{v_1/v_n}(u_1)\cdots
\d_{v_{n-1}/v_n}(u_{n-1}) \\
= \ \
  \sum_{r\in\Z^{n-1}} c_r \,   e\(\textstyle{\f{r_1u_1 \,+\, \cdots  \, + \,
r_{n-1}u_{n-1}}{N}}\)  \,,    \endgathered
 \]
 where \[  \  c_r  \, =  \,  \f{1}{N^{n-1}} \sum_{\stackrel{\scriptstyle{d\,>\,0}}{d
\, | \, r_1,\ldots,r_{n-1}}}
    \psi(d)   \, d^{-\nu+n/2-1}\,.
\]
These sums, and hence   also both $E_{\nu,  \psi  }$ and $\widetilde{E}_{\nu, \psi }$, can be holomorphically
continued   to $\C - \{n/2\}$.  They are entire if $\psi$ is nontrivial, and have a simple pole at
$\nu=n/2$ otherwise.
\end{prop}

{\noindent \bf Proof:} Because of the relation (\ref{twoeis}) and the visible transformation properties of the asserted formulas, the formulas for $E_{\nu,\psi}$ and $\widetilde{E}_{\nu,\psi}$ are equivalent.  We shall thus work with  $\widetilde{E}_{\nu,\psi}$, first deriving the formulas as sums of $\d$-functions,  then the alternative expressions in terms of Fourier series, and finally deduce the meromorphic continuation from these.

We begin with the second set of formulas, for the restriction to $w_{\text{long}}\widetilde{U}$.
 Letting
$\G$ instead stand for $w_{\text{long}}\widetilde{\G}_0(N) w_{\text{long}}$, the expression for
$\widetilde{E}_{\nu, \psi } \in \widetilde W_{\nu, \e }^{-\infty}$ in (\ref{eisen2}) may be rewritten as
\begin{equation}\label{eisen2rewr}
   \ell(w_{\text{long}})  \widetilde{E}_{\nu, \psi } \ \ = \ \ L(\nu+{\textstyle\f{n}{2}},\psi)  \, \sum_{\g \,
     \in \,
    \G/\G\cap \widetilde{P}}   \widetilde{\alpha}(w_{\text{long}}\,\g\, w_{\text{long}}) \,  \ell(\g) \,\d_{\tilde{e}}\,.
\end{equation}
The last column of a matrix is unchanged, up to sign, after right multiplication by an element of $\G\cap
\widetilde{P}$.  Moreover, every $n$-tuple of relatively prime integers occurs as the last column of some matrix in
$GL(n,\Z)$.  Its subgroup $\G$ is defined by the congruence that all entries except for the final one  in its last
column are divisible by $N$. Therefore, the
 cosets $\G/\G\cap \widetilde{P}$ are in bijective
correspondence with the set
\[
\{\text{vectors $v=(v_1,\ldots,v_n)\in \Z^n$ with $GCD(v)=1$ and $N|v_1,\ldots,v_{n-1}\}/\{\pm 1\}$}.
\]
Given $v\in \Z^n$ whose entries are relatively prime  and satisfy the above divisibility condition, we let $\g_v$
denote a coset representative in $\G/\G\cap \widetilde{P}$.

When (\ref{eisen2rewr}) is restricted to $\widetilde{U}$, some of the terms in the sum on the right hand side
vanish because the $\g$-translate of $\d_{\tilde{e}}$ does not lie in the big cell. The nonvanishing terms are
precisely those for which $\g\in \G\subset G$ projects into the big cell $\widetilde{U}\subset
\widetilde{Y}=G/\widetilde{P}$.  A matrix whose final column is the vector $v$ projects to the big cell
$\widetilde{U}$ if and only if its last entry is nonzero; in this situation, applied to $\g_v$, we have the
explicit matrix decomposition
\begin{equation}\label{projbigcell}
    \g_v \ \ = \ \ \ttwo{I}{u}{}{1} \ttwo{A}{}{\star}{v_n}\, ,
\end{equation}
where $u = \f{1}{v_n}(v_1,\ldots,v_{n-1})\in \R^{n-1}$ and $A$ is a matrix with determinant $\pm 1/v_{n}$.
Therefore the range of summation in (\ref{eisen2rewr}) is in bijection with
\begin{equation}\label{cosetrepsinsum}
\{v=(v_1,\ldots,v_n)\in\Z^n\,\ \text{with~}GCD(v)=1,\, N\mid v_1,\ldots,v_{n-1},\text{~and~}v_n>0\}.
\end{equation}
The decomposition (\ref{projbigcell}) allows us to compute the following action of $\g_v\i$ on the delta function
$\d_{\left(\f{v_1}{v_n},\ldots,\f{v_{n-1}}{v_n}\right)}$ on $ \widetilde U $:
\begin{equation}\label{deltaact1}
\aligned
\ell(\g_v\i)\,\d_{\left(\f{v_1}{v_n},\ldots,\f{v_{n-1}}{v_n}\right)}
\ \ &  = \ \ \ell\(\g_v\i \(\begin{smallmatrix}{
}_{\scriptstyle{1}} & & &
    {}_{v_1/v_{n}} \\
 & \ddots  & {}^{\textstyle{0}} & \vdots  \\
 { }^{\textstyle{0}}  &   & 1 & {v_{n-1}/v_n}  \\
 &  &   &   1
 \end{smallmatrix}\) \) \d_{\tilde{e}} \\
& = \ \ \ell\(\ttwo A{}{\star}{v_n} \i\)\d_{\tilde{e}} \qquad\qquad (\det{A} \,=\,\pm  1/v_n) \\
& = \ \  (\sgn v_n)^\e \,|v_n|^{\nu+n/2}\,\d_{\tilde{e}} \,.
 \endaligned
\end{equation}
In this last equation, the transformation rule (\ref{mira6explicated}) has provided a factor of $(\sgn v_n)^\e\\
|v_n|^{\nu-n/2}$, while the $\d$-function identity $\d_{\tilde{e}}(\f{Au}{v_n})=|v_n|^n\d_{\tilde{e}}(u)$ is
responsible for the rest of the exponent.  Using $\widetilde{\alpha}(w_{\text{long}} \g
w_{\text{long}})=\psi(v_n)$, the summand for $\g_v$ in (\ref{eisen2rewr}) can be written as
\begin{equation}\label{lgvaction}
  \widetilde{\alpha}(w_{\text{long}}\,\g_v\, w_{\text{long}}) \, \ell(\g_v) \, \d_{\tilde{e}} \ \ = \ \   \psi(v_n)\,(\sgn v_n)^\e  \, |v_n|^{-\nu-n/2}\,\d_{\left(\f{v_1}{v_n},\ldots,\f{v_{n-1}}{v_n}\right)}\,.
\end{equation}
Summing this expression over the coset representatives from (\ref{cosetrepsinsum}) gives,
 in terms of the coordinates $(u_1,\ldots,u_{n-1})$
on $\widetilde{U}$ in the second set of statements in the proposition,  an expression similar to the one claimed there for
$\ell(w_{\text{long}})\widetilde{E}_{\nu,\psi}$.  They differ only in that the latter has no condition on $GCD(v)$.   However, the first
set consists of scalar multiples, by positive integers relatively prime to $N$, of the second set, and
multiplication by the pre-factor  $L(\nu+\f n2 ,\psi)$   in (\ref{eisen2rewr}) -- unused until now --  accounts for
the discrepancy. (Note that terms for which $(v_n,N)>1$ vanish.)

At this point, we have established the $\d$-function formula  for the restriction of
$\ell(w_{\text{long}})\widetilde{E}_{\nu,\psi}$ to $\widetilde U$, and therefore also the one for the restriction of $\ell(w_{\text{long}})E_{\nu,\psi}$ to $U$, to which it is equivalent.  Had we instead considered the series $\widetilde{E}_{\nu,\psi}$ instead of its $w_{\text{long}}$-translate, the last column of $\g$ would have entries $(v_n,\ldots,v_1)$, the reverse of the situation we encountered above.  The identical reasoning produces the same formula, but with $v_j$ replaced by $v_{n+1-j}$ in the summand -- exactly the first claim of the proposition.

Next we turn to the assertions about the Fourier expansions, starting first with the
common expression  for the $w_{\text{long}}$ translates.  It is periodic in each $u_i$ with period $N$, so the coefficient $c_r$ is computed by the
integral
\begin{equation}
\label{crcalc1}
\begin{aligned}
&   \f{1}{N^{n-1}}   \int_{(  N   \Z\backslash \R)^{n-1}}{\sum}_{\stackrel{\scriptstyle v \, \in \, \Z^n\!,\ {\scriptstyle v_n
\, >  \, 0}}{N\mid v_1,\ldots,v_{n-1}\ }}\ \ \psi(v_n)\,   v_n^{\,-\nu-n/2} \ e\!\(  -\f{\sum_{i=1}^{n-1}r_i u_i}{N}  \) \times
\\
&\qquad\qquad\qquad \times \,\d_{v_1/v_n}(u_1)\cdots
\d_{v_{n-1}/v_n}(u_{n-1})\,du_1\cdots du_{n-1} \\
& \ \ \  =   \f{1}{N^{n-1}}   \  \sum_{v_n \, > \, 0}\ \ \sum_{\stackrel{\scriptstyle v_1,\ldots,v_{n-1} \,
\in \, \Z/  N   v_n \Z}{   N\mid v_1,\ldots,v_{n-1}}}  \psi(v_n) \,  v_n^{\,-\nu-n/2} \ e\!\(
-\f{\sum_{i=1}^{n-1}r_i v_i}{  N   v_n}\) \\
& \ \ \ =   \f{1}{N^{n-1}}   \  {\sum}_{d \, > \, 0} \ \ {\sum}_{v_1,\ldots,v_{n-1}
\, \in \, \Z/d \Z} \ \ \psi(d) \,  d^{-\nu-n/2} \ e\!\(\f{\sum_{i=1}^{n-1}r_i v_i}{d}\).
\end{aligned}
\end{equation}
  The sum over any fixed  $v_j$, for $1\le j \le n-1$, equals $d$ if $d|r_j$, and zero otherwise.
Therefore $c_r$ is given by the formula stated in the proposition.  The formula for $a_r$ is computed by the same procedure.  The hybrid formula for the restriction $E_{\nu,\psi}$ or $\widetilde{E}_{\nu,\psi}$ which involves a Fourier series in $u_1$, and $\d$-functions in the other variables, is proven by taking a Fourier integral only in the variable $u_1$, and leaving the other $u_j$ alone.

Finally we come to the analytic continuation, which is equivalent for each of the expressions involved.  We therefore consider the last formula in the statement of the proposition.   The coefficient $c_r$ equals a finite sum which is entire in $\nu$,
unless
 $r = (0,0,\ldots,0)$.  In this exceptional case\, $c_0 = N^{1-n}   L(\nu-n/2+1,\psi)$,
 which is entire   for all nontrivial characters $\psi$, and has a simple pole at  $\nu=n/2$ when $\psi$ is trivial.
This establishes  the asserted meromorphic continuation of the restriction of the Eisenstein series $E_{\nu,  \psi
}$ to  the open Schubert cell   $w_{\text{long}}U$.  Since $E_{\nu,\psi}$ is automorphic under $\G_0(N)$, and the
$\G_0(N)$-translates of $w_{\text{long}}U$ cover $Y=G/P$, the continuation is valid on all of $Y$.  Likewise, the
identical meromorphic continuation applies to $\widetilde{E}_{\nu,\psi}$ because of (\ref{twoeis}). \bx

We have now shown the analytic continuation of the mirabolic Eisenstein distributions. We next turn to
their functional equations.   The two degenerate principal series representations (\ref{mira6}) are
related by the {\it standard intertwining operator}
\begin{equation}
\label{mira7} I_\nu\ : \ W_{-\nu,\e}^\infty \ \longrightarrow\
\widetilde W_{\nu,\e}^\infty\,,
\end{equation}
defined in terms of the realization by $C^\infty$ functions by the integral
\begin{equation}
\label{mira8} \left( I_\nu f\right)(g)\ = \ \int_U f(g
\,w_{\text{long}}\, u)\,du\,;
\end{equation}
recall the definition of $w_{\text{long}}$ in (\ref{ab9}). It is well known that the integral converges
absolutely\footnote{For the sake of notational simplicity we are dropping the subscript $\e$ for
$I_{\nu}$\,,\,\ since the action of the intertwining operator affects only $\nu$, not $\e$. } for $\,\Re
\nu
>  n/2-1$, and we shall also see this directly. Two properties of $I_\nu$ are crucial for our purposes:
\begin{equation}
\label{mira9}
\begin{aligned}
&\!\!\!\text{a)\ \ $I_\nu$ has a meromorphic continuation to all
$\nu\in\C$\,, and}
\\
&\!\!\!\text{b)\ \ it extends continuously to a linear operator
$I_\nu : W_{-\nu,\e}^{-\infty} \, \rightarrow\, \widetilde
W_{\nu,\e}^{-\infty}$\,;}
\end{aligned}
\end{equation}
see \cite{knapp} for the former, and \cite{Casselman:1989} for the latter.

We now give an explicit formula for the action of $I_\nu$ in terms of the restriction of $C^\infty$ functions to
the open Schubert cells $\,U\subset G/P$, $\,\widetilde{U}\subset G/\widetilde{P}$, for $\nu$ in the range of
convergence -- i.e., for $\Re \nu > n/2 -1$.

\begin{prop}\label{expltwineprop}
Let $f\in W_{-\nu,\e}^\infty\,$, and regard $f$ as a function on $U\cong \R^{n-1}$ via its restriction to $U$ and
the identification
\begin{equation*}
\R^{n-1} \ni  \ x  \ \mapsto \  u(x)\ \ =_{\rm{def}} \ \ \left(
\begin{smallmatrix}1& x_{n-1} & \cdots & x_{1} \\
  & { }_{\scriptstyle{1}} & \ \  { }_{{ }_{\textstyle{0}}} &  \\
{ }^{\textstyle{\ 0}} & & \ \ \ddots &  \\
 &  &  & 1
 \end{smallmatrix}
\right) \  \in  \ U\,.
\end{equation*}
Similarly, regard $I_\nu f \in \widetilde{W}_{\nu,\e}^\infty$ as a function on $\widetilde U\cong \R^{n-1}$ via the
identification\begin{footnote}{The minus signs are necessary to make (\ref{ab9}) consistent with
(\ref{twoeis}).}\end{footnote}
\begin{equation*}
\R^{n-1} \ni  \ y  \ \mapsto \ \widetilde u(y)\ \ =_{\rm{def}} \ \
\left(
    \begin{smallmatrix}{ }_{\scriptstyle{1}} &  &  &
    {}_{-y_1} \\
 & \ddots \ \  & {}^{\textstyle{0}} & \vdots  \\
 { }^{\textstyle{0}}  &   & 1 & {-y_{n-1}}  \\
 &  &   &   1\ \
 \end{smallmatrix}
\right) \  \in  \ \widetilde U\,.
\end{equation*}
Then, for $\,{\rm Re}\, \nu > n/2 -1$, $(I_\nu f)\(\widetilde u(y)\) $ is given by the integral
\begin{equation*}
  \int_{z\in\R^{n-1}}
  f\(u(z)\)\left|
\textstyle \sum_{j=2}^{n-1}y_j\,z_{n+1-j} - y_1-z_1 \right|^{\nu-n/2} \sgn(
\textstyle \sum_{j=2}^{n-1}y_j\,z_{n+1-j} - y_1-z_1)^\e\,dz.
\end{equation*}
\end{prop}

\noindent {\bf Proof:}  By construction, the intertwining operator $I_\nu$ is invariant under left translation by
any $g\in G$. To establish the assertion of the proposition, it therefore suffices to establish the integral
expression for $y=0$, and then to check that it is compatible with translation from $\,\widetilde u(0)= e$ to
$\widetilde u(y)$.

First the compatibility with translation. On the one hand, $ (I_\nu f)\(\widetilde u(y)\) = \(\ell(\widetilde u(-y)
)   \(I_\nu f\)\)(e)= \(I_\nu \, \ell(\widetilde u(-y)   )  f\)(\widetilde u(0))$; on the other,
\begin{equation}\label{expltwine1}
\begin{aligned}
\int_{z\in\R^{n-1}} \ \( \ell(\widetilde u(-y)   )   f\)&(u(z)) \left|
\, z_1\,\right|^{\nu-n/2}\,\sgn(-z_1)^\e\,dz \ \ =
\\
&\ \ = \int_{z\in\R^{n-1}} \ f\( \widetilde u(y)      \cdot u(z)\)
\left| \, z_1\,\right|^{\nu-n/2}\,\sgn(-z_1)^\e\,dz\,.
\end{aligned}
\end{equation}
Since
\begin{equation}\label{expltwine2}
\widetilde u(y) \cdot u(z)\ = \  \left(
\begin{smallmatrix}1& z_{n-1} & z_{n-2} & \cdots & z_2 & \tilde z_1 \\
 & { }_{\scriptstyle{1}} & { }_{\scriptstyle{0}} & { }_{\scriptstyle{\cdots}} &  { }_{\scriptstyle{0}} &  { }_{\scriptstyle{0}} \\
 & { }_{{ }_{\textstyle{\ 0}}}&  \ddots  & & & \vdots \\
 & & & & 1 & 0 \\
 &  &  & & & 1
 \end{smallmatrix} \right)
 \left(
\begin{smallmatrix}1& 0 & 0 & \cdots & 0 & 0\\
 & { }_{\scriptstyle{1}} & { }_{\scriptstyle{0}} & { }_{\scriptstyle{\cdots}} &  { }_{\scriptstyle{0}}  &  { }_{\scriptstyle{-y_2}}  \\
 & { }_{{ }_{\textstyle{\ 0}}}& \ \ \ddots \ \ & \ddots & & \vdots \\
 & & & & 1 & -y_{n-1} \\
 &  &  & & & 1
 \end{smallmatrix} \right)\,,
 \end{equation}
with $\tilde z_1= z_1 - y_1 + \sum_{2\leq j \leq n-1}\,z_j y_{n+1-j}$, the transformation law
(\ref{mira6explicated}) implies that the integral (\ref{expltwine1}) coincides with the integral in the
proposition.

At this point, it suffices to treat the case $y=0$. According to the definition of the intertwining operator,
\begin{equation}
\label{expltwine3}
\begin{aligned}
 & \!\!\!\!\!\!\! \!\!\!  (I_\nu f)\left(\widetilde u(0) \right) \ = \ \int_{z\in\R^{n-1}}
\, f(w_{\text{long}}\,u(z))\,dz
\\
&= \ \ \int_{z\in \R^{n-1}}\, f(\textstyle u(\f {1}{z_1},
\f{-z_{n-1}}{z_1},\dots, \f{-z_3}{z_1}, \f{-z_2}{z_1}))\, \left|\,
z_1  \,\right|^{-\nu - n/2}\,\sgn(-z_1)^\e\,dz
\\
&= \ \ \int_{z\in \R^{n-1}}\, f(u(z))\, \left|\, z_1
\,\right|^{\nu - n/2}\,\sgn(-z_1)^\e\,dz \,;
\end{aligned}
\end{equation}
at the second step, we have used the transformation law (\ref{mira6explicated}) and the matrix identity
\begin{equation}
\label{expltwine4} w_{\text{long}}\,u(z)\ = \  \left(
\begin{smallmatrix} 1& -z_2/z_1 & \cdots & -z_{n-1}/z_1 & 1/ z_1 \\
 & { }_{\scriptstyle{1}}  &  { }_{\scriptstyle{0}}&  &  { }_{\scriptstyle{0}} \\
 & { }_{{ }_{\textstyle{\ 0}}}&  \ddots  &\ddots & \vdots \\
 & & &  1 & 0 \\
 &  &  & & { }_{\scriptstyle{1}}
 \end{smallmatrix} \right)
 \left(
\begin{smallmatrix}
-1/z_1& 0 &  \cdots & 0 & 0 \\
{ }_{\scriptstyle{0}} & { }_{\scriptstyle{0}} & { }_{\scriptstyle{\cdots}}&  { }_{\scriptstyle{1}} &  { }_{\scriptstyle{0}}    \\
{ }^{\scriptstyle{\vdots}} &  & \!\!\! { }^{\scriptstyle{\iddots}}  & & { }^{\scriptstyle{\vdots}}  \\
0 & \!\! 1 & 0 & \cdots & 0 \\
1 & z_{n-1}  & z_{n-2}  & \cdots & z_1
 \end{smallmatrix} \right)\,,
 \end{equation}
and at the third step, the change of variables
\begin{equation}
\label{expltwine5} (z_1, z_2, \dots, z_{n-1})\ \mapsto\
(\textstyle \f {1}{z_1}, \f{-z_{n-1}}{z_1},\dots, \f{-z_3}{z_1},
\f{-z_2}{z_1})\,.
\end{equation}
The identity (\ref{expltwine3}) completes the proof of the proposition. \bx

The identity (\ref{expltwine4}) and the transformation law (\ref{mira6explicated}) directly imply a simple
estimate: along the line $\{x_2 = x_3 = \dots = x_{n-1}=0\}$, any $f\in W_{-\nu,\e}^\infty$ satisfies the bound
$|f(u(x))| = O(\|x\|^{-\Re \nu-n/2})$ as $\|x\|\to \infty\,$; the implied constant depends on a bound for
$\ell(w_{\text{long}})f$ on a neighborhood of the origin. We consider $SO(n-1)$ as a subgroup of $GL(n)$ by
embedding it into the  bottom right corner. Then $SO(n-1)$ acts transitively, by conjugation, on the set of lines
in $\,\R^{n-1}\cong U$.  By compactness, the translates $\ell(w_{\text{long}}m)f$, for $m\in SO(n-1)$, are
uniformly bounded on bounded subsets of $\,\R^{n-1}\cong U$. Since $f\in W_{-\nu,\e}^\infty$ is invariant under
right translation by elements of $SO(n-1)$, the estimate we gave holds not on just a single line, but globally on
$U$:
\begin{equation}
\label{expltwine6} f\in W_{-\nu,\e}^\infty\ \ \Longrightarrow \ \
\|f(u(x))\| \, = \, O(\|x\|^{-\Re\nu-n/2})\, \ \text{as}\,\
\|x\|\to \infty  \,.
\end{equation}
This bound and its derivation are valid for all $\nu\in\C$. When $\,\Re\nu > n/2-1$, it implies the convergence of
the integral (\ref{expltwine3}), both near the origin and at infinity. Since $I_\nu$ is $G$-invariant, we have
established that the integral (\ref{mira8}) does converge for $\Re\nu > n/2-1$ and any $g\in G$, as was mentioned
earlier.

In complete  analogy to $I_\nu : W^\infty_{-\nu,\e} \to \widetilde W_{\nu,\e}^\infty$ in (\rangeref{mira7}{mira8}),
one can define   the operator $\widetilde I_\nu : \widetilde W^\infty_{-\nu,\e} \to W_{\nu,\e}^\infty$; this
involves integrating over $\widetilde U$ instead of $U\!$. Then $I_\nu$, $\widetilde I_\nu$   are dual to
each other, in the sense that
\begin{equation}
\label{expltwine7}
\begin{aligned}
\int_{\widetilde U} I_\nu f_1(\widetilde u)\, \widetilde {f_2}
(\widetilde u)\,d\widetilde u \ &= \ \int_U f_1(u)\, \widetilde
I_\nu \widetilde f_2(u)\,du \,,
\\
&\text{for all}\ \  f_1 \in W_{-\nu,\e}^\infty\, \ \text{and}\,\
\widetilde {f_2}\in \widetilde W_{-\nu,\e}^\infty\,;
\end{aligned}
\end{equation}
the integrals on the two sides implement the natural $G$-equivariant pairings between $\widetilde
W_{\nu,\e}^\infty$ and $\widetilde W_{-\nu,\e}^\infty$, respectively $W_{-\nu,\e}^\infty$ and $W_{\nu,\e}^\infty$.
For $\Re \nu > n/2-1$, i.e., when the integrals defining $I_\nu$ and $\widetilde{I_\nu}$ converge, the  identity
follows from the explicit formula for $I_\nu$ in proposition \ref{expltwineprop} and the analogous formula for
$\widetilde I_\nu$. Meromorphic continuation implies the identity for other values of $\nu$.

Since $I_\nu$ extends continuously to $I_\nu : W_{-\nu,\e}^{-\infty} \to \widetilde W_{\nu,\e}^{-\infty}$, the
identity (\ref{expltwine7}) implies a concrete description of the effect of $I_\nu$ on distribution vectors,
\begin{equation}
\label{expltwine8}
\begin{aligned}
\int_{\widetilde U} I_\nu \tau (\widetilde u)\, \widetilde {f}
(\widetilde u)\,d\widetilde u \ &=\ \int_U \tau(u)\, \widetilde
I_\nu \widetilde f(u)\,du \,,
\\
&\text{for all}\ \ \tau \in W_{-\nu,\e}^{-\infty},\ \widetilde {f}\in
\widetilde W_{-\nu,\e}^\infty\,.
\end{aligned}
\end{equation}
Unlike in (\ref{expltwine7}), the integrals in this identity have merely symbolic meaning: the pairings $\widetilde
W_{\nu,\e}^{-\infty} \times \widetilde W_{-\nu,\e}^\infty \to \C$ and $W_{-\nu,\e}^{-\infty} \times
W_{\nu,\e}^\infty \to \C$ involve ``integration" over $\widetilde Y=G/\widetilde P$ and $Y=G/P$, not only over the
dense open cells $\widetilde U\subset \widetilde Y$, $U\subset Y$. The integrals as written do extend naturally to
$\widetilde Y$ and $Y$.

Let $E_{1,n} \in \mathfrak g\mathfrak l(n,\R)$ denote the matrix
with the entry 1 in the $(1,n)$-slot, and zero entries otherwise.
If $f\in W_{-\nu,\e}^\infty$ and $\,\Re \nu > 1 - n/2$, the estimate
(\ref{expltwine6}) shows that the integrals
\begin{equation}
\label{expltwine9} J_{\nu} f (g)\ =_{\text{def}}\ \int_\R \, f\(g
\exp(t\,E_{1,n})\)\, dt \qquad (\,f\in W^\infty_{-\nu}\,,\ g \in
G\,)
\end{equation}
converge. For other values of $\nu$, $\nu \notin 1 - n/2 -
\Z_{\geq 0}$, the integrals still make sense by meromorphic
continuation (the unspecified integer in $\Z_{\ge 0}$ in fact has the same parity as $\e$ at any singularity). This can be seen by translating the point
$\,\lim_{t\to\infty} \exp(tE_{1,n}) P \in Y$ to the origin.

\begin{lem}\label{expltwine10}
Suppose $I_\nu : W_{-\nu,\e}^\infty \to \widetilde W_{\nu,\e}^\infty$
has no pole at $\nu$, $W_{-\nu,\e}^\infty$ and $\widetilde
W_{\nu,\e}^\infty$ are irreducible, and $\nu \notin 1 -n/2 -
\Z_{\geq 0}\,$. Then for any $f\in W_{-\nu,\e}^\infty\,$, the
integrals $J_{\nu} f (u)$ vanish for all $u\in U\,$ if and only if
$\,I_{\nu} f \in \widetilde W^\infty_\nu$, viewed as $C^\infty$
section of the line bundle $\widetilde \L_{\nu-\rho_{\text{mir}},\e}
\to\widetilde Y$, vanishes on the entire complement of $\widetilde
U$ in $\widetilde Y$.
\end{lem}
Both representations  are generically irreducible, and $I_\nu$
depends meromorphically on $\nu$, so the hypotheses are satisfied
outside a discrete set of values of the parameter $\nu$. The
automorphism {\rm (\ref{ab9})} preserves the one parameter group
$t\mapsto \exp(t\,E_{1,n})$. Since this automorphism switches the
roles of $I_\nu$ and $\widetilde I_{\nu}$,  $W_{\nu,\e}^\infty$ and $\widetilde
W^\infty_{\nu,\e}$, etc., the lemma applies analogously to $\widetilde
I_\nu$.

The explicit formula for $I_\nu f$ -- for $f\in C^\infty_c(U)$, so
that convergence is not an issue -- shows that $I_\nu$ cannot
vanish. Because of the other hypotheses of the lemma, $I_\nu$ must
then be one-to-one and have dense image. But the image is
necessarily closed \cite{Casselman:1989}, hence in the situation
of the lemma,\begin{equation} \label{expltwine11} I_\nu :
W_{-\nu,\e}^\infty \ \longrightarrow\ \widetilde W_{\nu,\e}^\infty\ \
\text{is a topological isomorphism.}
\end{equation}

\noindent{\bf Proof} of Lemma \ref{expltwine10}: The $J_\nu f(u)$
depend meromorphically on $\nu$, provided $f\in W^\infty_{-\nu,\e}$
varies meromorphically with $\nu$. Evaluation of $I_\nu f$ at any
particular point is also a meromorphic function of $\nu$. Thus,
without loss of generality, we may suppose
\begin{equation}
\label{expltwine12} \Re \nu \ \gg\ 0\,.
\end{equation}
We shall relate $I_\nu$ and $J_\nu$ to the $GL(n-1)$-analogue of
$I_\nu$. This requires a temporary change in notation: in this
proof we write $W_{n,\nu}^\infty\,$, $I_{n,\nu}\,$, etc., to
signify the dependence on $n$ (we omit the subscript $\e$ since it is fixed and does not play an essential role). We define
\begin{equation}
\label{expltwine13}
\begin{aligned}
&R_{n,\nu} : \widetilde W_{n,\nu}^\infty\ \longrightarrow\
\widetilde W_{n-1,\,\nu-1/2}^\infty\,,
\\
& \ \ (R_{n,\nu} \widetilde f )(g_1)\, = \, \left|\det
g_1\right|^{\f{n/2-\nu}{n(n-1)}} \, \widetilde f \( \!
\(\begin{smallmatrix}
0 & & & \\
\vdots & & { }^{\textstyle{1_{(n-1)\times(n-1)}}} &     \\
0 & & & \\
1 & 0  & \cdots & 0
\end{smallmatrix}\) \!
\(\begin{smallmatrix}
1 & 0 & \cdots & 0 \\
{ }_{\scriptstyle{0}} &    &     &  \\
\vdots & & { }^{\textstyle{g_1}} &  \\
0 &  & &
\end{smallmatrix}\) \! \);
\end{aligned}
\end{equation}
the fractional power of $\,|\det g_1|\,$ is necessary to relate
the transformation law (\ref{mira6}) for $\widetilde f\in
\widetilde W_{n,\nu}^\infty$ to that for $R_{n,\nu}\widetilde f\in
\widetilde W_{n-1,\,\nu-1/2}^\infty\,$. The first matrix factor in
the argument of  $\widetilde f$ makes this restriction operator
$GL(n-1)$-invariant relative to the tautological action on
$\,\widetilde W_{n-1,\,\nu-1/2}^\infty\,$ and the action on
$\,\widetilde W_{n,\nu}^\infty\,$ via the embedding
$GL(n-1)\hookrightarrow GL(n)$ into the top left corner. This top
left copy of $GL(n-1)$ acts transitively on the complement of
$\widetilde U$ in $\widetilde Y$, hence
\begin{equation}
\label{expltwine14} R_{n,\nu} \widetilde f\, \equiv\, 0 \ \ \
\Longleftrightarrow\ \ \ \widetilde f\, \ \text{vanishes on the
complement of $\,\widetilde U$ in $\,\widetilde Y$.}
\end{equation}
Next we define
\begin{equation}
\label{expltwine15}
\begin{aligned}
&A_{n,\nu} : W_{n,\nu}^\infty\ \longrightarrow\
W_{n-1,\,\nu+1/2}^\infty\,,
\\
&\qquad\qquad \(A_{n,\nu} f\)\! (g_1)\, = \, \left|\det
g_1\right|^{\f{n/2+\nu}{n(n-1)}} \,  J_{n,\nu} f \!
\(\begin{smallmatrix}
 & & & 0 \\
 & { }^{\textstyle{g_1}} &  & \vdots   \\
 & & & 0 \\
0 & \cdots & 0 &1
\end{smallmatrix}\).
\end{aligned}
\end{equation}
In this case, the power of $\,|\det g_1|\,$ reflects not only the
discrepancy between the transformation laws (\ref{mira6}) for $n$
and $n-1$, but also the commutation of the appropriate factor
across $\exp(t E_{1,n})$ in the defining relation
(\ref{expltwine9}) for $J_\nu$. It is clear from the definition
that $A_{n,\nu}$ relates the tautological action of $GL(n-1)$ on
$W_{n-1,\,\nu+1/2}^\infty$ to that on $W_{\nu,\e}^\infty$ via the
embedding $GL(n-1)\hookrightarrow GL(n)$ into the top left corner.
We claim:
\begin{equation}
\label{expltwine16} A_{n,\nu}  f\, \equiv\, 0 \ \ \
\Longleftrightarrow\ \  J_{n,\nu} f(u)\,=\,0\, \ \text{for all
$\,u\in U$.}
\end{equation}
Indeed, since $U$ is dense in $G/P$, $f$ vanishes identically if
and only if $f$ vanishes on $U$. We use the analogous assertion
about $A_\nu f$, coupled with the following observation: let $U_1$
denote the intersection of $U$ with the image of
$GL(n-1)\hookrightarrow GL(n)$; then $U_1 \cdot \{\exp(tE_{1,n})\}
=U$.

The intertwining operators $I_{n,\nu}$\,, $I_{n-1,\,\nu-1/2}$\,
and the operators we have just defined constitute the four edges
of a commutative diagram,
\begin{equation}
\label{expltwine17}
\begin{CD}
W_{n,-\nu}^\infty \ \ \ @>{\ \ \ \ \ I_{n,\nu} \ \ \ \ \ }>>  \ \
\ \widetilde W_{n,\nu}^\infty
\\
@V A_{n,-\nu}  VV   @VV  R_{n,\nu} V
\\
W_{n-1,\,-\nu+1/2}^\infty \ \ \  @>{\ \ I_{n-1,\,\nu-1/2}\ \ }>>
\ \ \ \widetilde W_{n-1,\,\nu-1/2}^\infty \,.
\end{CD}
\end{equation}
The commutativity is a consequence of two matrix identities. The first,
\begin{equation}
\label{expltwine18}
\begin{aligned}
&\(\begin{smallmatrix}
0 &  \cdots & 0 & 1\\
   &    &  1   &  0 \\
\vdots & \iddots & & \vdots \\
1 &  & &  0
\end{smallmatrix}\)
\(\begin{smallmatrix}
1 & x_{n-1} & \cdots & x_1 \\
0 & 1    & \cdots &0 \\
\vdots & & \ddots  & \\
0 & \cdots &  & 1
\end{smallmatrix}\) \ =
\\
&\qquad\qquad\qquad = \ \(\begin{smallmatrix}
1 & 0 & \cdots & 0 \\
0 &    &     &  1 \\
\vdots & & \iddots &  \\
0 & 1 & &
\end{smallmatrix}\)
\(\begin{smallmatrix}
1 & 0 & \cdots & \cdots & 0 \\
0 & 1   &   x_{n-1}  & \cdots & x_2 \\
\vdots & &\ \ \ddots  & & \\
0 & \cdots &\!\! \cdots & & \!\! 1
\end{smallmatrix}\)
\(\begin{smallmatrix}
0 & \cdots & 0 & 1 \\
1 & 0   \cdots & 0 & x_1 \\
   & \ddots  &  & \\
0 & \cdots & 1 &  0
\end{smallmatrix}\),
\end{aligned}
\end{equation}
implies a factorization of $I_{n,\nu}$ as the composition of $I_{n-1,\,\nu-1/2}$
with a certain intermediate operator, which involves an integration over the one
parameter group $\{\exp(tE_{2,n})\}$ instead of $\{\exp(tE_{1,n})\}$, as in the
case of $J_\nu$. The second,
\begin{equation}
\label{expltwine19}
\begin{aligned}
&\(\begin{smallmatrix}
0 & & & \\
\vdots & & { }^{\textstyle{1_{(n-1)\times(n-1)}}} &     \\
0 & & & \\
1 & 0  & \cdots & 0
\end{smallmatrix}\) \!
\(\begin{smallmatrix}
1 & 0 & \cdots & 0 \\
{ }_{\scriptstyle{0}} &    &     &  \\
\vdots & & { }^{\textstyle{g_1}} &  \\
0 &  & &
\end{smallmatrix}\)\(\begin{smallmatrix}
0 & \cdots & 0 & 1 \\
1 & 0   \cdots & 0 & x_1 \\
   & \ddots  &  & \\
0 & \cdots & 1 &  0
\end{smallmatrix}\)\ =
\\
&\qquad\qquad\qquad\qquad\qquad\qquad =\ \(\begin{smallmatrix}
 & & & 0 \\
 & { }^{\textstyle{g_1}} &  & \vdots   \\
 & & & 0 \\
0 & \cdots & 0 &1
\end{smallmatrix}\)
\(\begin{smallmatrix}
1 & 0 & \cdots & 0 &x_1 \\
0 & 1    &0 & \cdots & \\
\vdots & &\ \  \ddots  & &  \\
0 &  & \cdots &  & 1
\end{smallmatrix}\),
\end{aligned}
\end{equation}
relates this intermediate operator to $J_{n,\nu}$.

Under the hypotheses of the lemma $I_{n,\nu}$ is an isomorphism -- recall (\ref{expltwine11}).
One can show that under the same hypotheses $I_{n-1,\,\nu-1/2}$ is
also an isomorphism. Alternatively one can use the meromorphic
dependence on $\nu$ to disregard the discrete set on which
$I_{n-1,\,\nu-1/2}$ might fail to be an isomorphism. In any case, when
both $I_{n-1,\,\nu-1/2}$ and $I_{n-1,\,\nu-1/2}$ are isomorphisms,
(\ref{expltwine14}), (\ref{expltwine16}), and the commutativity of
the diagram (\ref{expltwine17}) imply the assertion of the lemma. \bx

The functional equation of the mirabolic Eisenstein series relates $E_{-\nu,\psi}$ to $\widetilde E_{\nu,\psi\i}$
via the inter\-twining operator $I_\nu : W_{-\nu,\e}^{-\infty} \to \widetilde W_{\nu,\e}^{-\infty}$. For the
statement, we   follow the notational convention
\begin{equation}
\label{eisen3} G_\d(s)  =  \! \int_\R \! e(x)
\left(\sg(x)\right)^\d |x|^{s-1}\,dx  =  \begin{cases}
2(2\pi)^{-s}\, \G(s)\, \cos\textstyle\frac{\pi s}{2} &\text{if}\ \d=0\\
2(2\pi)^{-s}\, \G(s)\, \sin\textstyle\frac{\pi s}{2} &\text{if}\
\d=1
\end{cases}
\end{equation}
\cite{inforder}, which we shall also use later in this paper. Note that the integral converges, conditionally only,
  for $0 < \re s <1$, but the expression on the right provides a meromorphic continuation to the entire
$s$-plane. The two cases on the right hand side of (\ref{eisen3}) can be written uniformly using $\G$-function
identities as
\begin{equation}\label{eisen3b}
    G_\d(s) \ \ = \ \ i^\d \,\f{\G_\R(s+\d)}{\G_\R(1-s+\d)} \ \ , \ \ \ \ \text{with} \ \,
    \G_\R(s) \,=\,\pi^{-s/2}\G(\smallf s2)\ \ \text{and} \ \ \d\,\in\,\{0,1\}.
\end{equation}
We also need some notation pertaining to the finite harmonic analysis of Dirichlet characters.  Let
 $\tau_\psi=\widehat{\psi}(1) = \sum_{b\imod N}\psi(b)e(\smallf{b}{N})$ denote the Gauss sum for $\psi$, a Dirichlet character of modulus $N$ (cf.~(\ref{fft})).  We let $\widehat{(\Z/N\Z)^*}$ denote the group of characters of $\Z/N\Z^*$ and  $\phi(N)$, the Euler $\phi$-function,  its order.

\begin{prop}[Functional Equation]\label{eisfe}
\begin{multline*}
I_\nu  E_{-\nu,\psi} \ \ = \\  \  (-1)^\e  N^{2\nu-\f{\nu}{n}-\f12} G_\e(\textstyle{\nu-\f n2}+1)\,\f{1}{\phi(N)}\sum_{\srel{a\imod
N}{\xi\in \widehat{(\Z/N\Z)^*}}} \widehat{\psi}(a)\xi(a)\i
\,\ell(w_{long})\,\ell\ttwo{N}{}{}{I_{n-1}} \,\widetilde E_{\nu,\xi} \,.
\end{multline*}
Consequently, if $\psi$ is a primitive Dirichlet character of modulus $N$, then
\begin{equation*}
I_\nu  E_{-\nu,\psi} \ =  \ (-1)^\e \tau_\psi N^{2\nu-\f{\nu}{n}-\f12} G_\e(\textstyle{\nu-\f
n2}+1)\,\ell(w_{long})\,\ell\ttwo{N}{}{}{I_{n-1}}\,\widetilde E_{\nu, \psi\i} \,.
\end{equation*}  In particular
\begin{equation*}
    I_\nu\, E_{-\nu,{\mathbbm 1}} \ \ = \ \ G_0(\nu-\smallf n2+1)\, \widetilde E_{\nu,{\mathbbm 1}}\,,
\end{equation*}
where ${\mathbbm 1}$ is the trivial Dirichlet character of conductor $N=1$.
\end{prop}

\begin{proof}
 Since both sides of the equation depend
meromorphically on $\nu$, we may assume that the hypotheses of
lemma \ref{expltwine10} hold, both at $\nu$ and $-\nu$. We shall
also require
\begin{equation}\label{eisfe1}
\ \Re \nu \ \gg \ n/2\,,
\end{equation}
so that the integral defining $\,\widetilde I_\nu$ converges.
Because of (\ref{expltwine8}), the proposition is equivalent to
the equality
\begin{multline}
\label{eisfe2} \f{1}{(-1)^\e N^{2\nu-\f{\nu}{n}-\f 12} G_\e(\textstyle{\nu-\f n2}+1)}\int_{U} E_{-\nu,\psi} (u)\, \widetilde{I_\nu}\widetilde
f(u)\,du \  \ = \\  \f{1}{\phi(N) }\sum_{\srel{a\imod
N}{\xi\in \widehat{(\Z/N\Z)^*}}} \widehat{\psi}(a)\xi(a)\i \displaystyle
\int_{\widetilde U}\ell\ttwo{I_{n-1}}{}{}{N}\,\ell(w_{long})\,\widetilde E_{\nu,\xi} (\widetilde u)\, \widetilde
f(\widetilde u)\,d\widetilde u \,,
\end{multline}
for all $\widetilde f \in \widetilde W_{-\nu}^\infty\,$. Both
$E_{-\nu,\psi}$ and $\widetilde E_{\nu,\xi}$ are  invariant under congruence subgroups of $GL(n,\Z)$, and
$\widetilde I_\nu : W_{-\nu,\e}^{\infty} \simeq \widetilde
W_{\nu,\e}^{\infty}$ by (\ref{expltwine11}). It therefore suffices to
establish this equality when $\widetilde I_\nu \widetilde f$
has~-- necessarily compact -- support in the open cell $U \subset
Y$,
\begin{equation}
\label{eisfe3} \text{supp}\,\bigl(\widetilde I_\nu\widetilde
f\,\bigr) \,\ \text{is compact in}\,\ U\,.
\end{equation}
We shall make one other assumption, namely
\begin{equation}
\label{eisfe4} \int_\R \widetilde I_\nu \widetilde f
(u(x_1,\,x_2,\,\dots\,\,x_{n-1}))\,dx_1\, = \, 0\,, \ \ \text{for
all}\,\ x_2,\,\dots\,,x_{n-1}\in\R\,.
\end{equation}
Indeed, if (\ref{eisfe2}) were to hold subject to the condition
(\ref{eisfe4}), the restriction to $\,\widetilde U$ of the
difference between $I_\nu E_{-\nu,\psi}$ and the formula we have asserted it is equal to
could be expressed as a Fourier series
\begin{equation}
\label{eisfe5} \sum_{r_2,\,\dots,\,r_{n-1} \in \Z} \,
a_{r_2,\,\dots,\,r_{n-1} } \,e(r_2 y_2 + \dots +
r_{n-1}y_{n-1})\,,
\end{equation}
without dependence on $y_1$. But no such expression can be the
restriction to $\widetilde U$ of a   distribution
vector invariant under a congruence subgroup $\G$: any generic $\g\in\G$ will transform the expression
(\ref{eisfe5}) to a distribution that does depend non-trivially on
$y_1$. This justifies the additional hypothesis (\ref{eisfe4}).

In effect, the integrals  (\ref{eisfe4}) coincide with the
integrals $J_{-\nu}\bigr( \widetilde I_\nu\widetilde f\bigr)(u)$,
as in (\ref{expltwine9}), for $u\in U$. Consequently lemma
\ref{expltwine10} implies the vanishing of $I_{-\nu}\!\circ\!
\widetilde I_\nu f$ on the complement of $\widetilde U$. But our
hypotheses ensure that $I_{-\nu}\!\circ\! \widetilde I_\nu\,$ is a
multiple of the identity, so
\begin{equation}
\label{eisfe6} \widetilde f\ \ \text{vanishes on the complement of
$\widetilde U$ in $\widetilde Y$}.
\end{equation}
Having compact support in $U$, $\widetilde I_\nu \widetilde f$
surely vanishes on the complement of $U$ in $Y$. Thus, applying
the lemma in reverse, we find
\begin{equation}
\label{eisfe7} \int_\R \widetilde f (\widetilde
u(y_1,\,y_2,\,\dots\,\,y_{n-1}))\,dy_1\, = \, 0\,, \ \ \text{for
all}\,\ y_2,\,\dots\,,y_{n-1}\in\R\,.
\end{equation}
We shall also need the estimate
\begin{equation}
\label{eisfe8} \left| P\bigl(\textstyle\f{\partial\ }{\partial
y_1},\,\dots,\,\f{\partial\ \ \ }{\partial
y_{n-1}}\bigr)\widetilde f(\widetilde u(y)) \right| \,= \,
O(\|y\|^{-\Re\nu-n/2})\ \ \text{as $\,\|y\|\to \infty$}\,,
\end{equation}
for all constant coefficient differential operators
$P\bigl(\textstyle\f{\partial\ }{\partial
y_1},\,\dots,\,\f{\partial\ \ \ }{\partial y_{n-1}}\bigr)$. It follows from (\ref{expltwine6}),
combined with the fact that the elements of the Lie algebra
$\widetilde {\mathfrak u}$ of $\widetilde U$ act on $\widetilde
W_{-\nu}^\infty$ by constant coefficient vector fields on
$\widetilde U\cong \R^{n-1}$.  In view of (\ref{eisfe1}),  (\ref{eisfe8})
implies the decay of $\widetilde f(\widetilde u(y))$ and all its derivatives.

We compute the integral on the right hand side of (\ref{eisfe2}) using the last restriction formula in proposition \ref{eisonu}:
\begin{equation}
\label{eisfe9}
\begin{aligned}
&\int_{\widetilde U} \ell\ttwo{I_{n-1}}{}{}{N}\ell(w_{\text{long}})\widetilde E_{\nu,\xi} (\widetilde u)\, \widetilde
f(\widetilde u)\,d\widetilde u \ \  =
\\
& = \  \ \f{N^{(1/2-\nu/n)(n-1)}}{N^{n-1}}\, \sum_{\srel{r\in \Z^{n-1}}{\srel{r_1\neq
0}{d|GCD(r)} } }  \xi(d)\,
d^{-\nu+n/2-1} \int_{\R^{n-1}} \widetilde f(\widetilde u(y))\,
e(r\cdot y)\, dy\,;
\end{aligned}
\end{equation}
here we have used  the fact that $\ttwo{I_{n-1}}{}{}{N} \i \widetilde u(y)=\widetilde u(Ny)\ttwo{I_{n-1}}{}{}{N}\i$, and the transformation law (\ref{mira6}) to pull out the power of $N$ in the numerator.  The terms corresponding to $r_1=0$
have been dropped because of (\ref{eisfe7}).
The sum in (\ref{eisfe9}) is absolutely convergent because  of the derivative bound (\ref{eisfe8}).

Let us now consider the finite sum over $a$ and $\xi$ to its left in (\ref{eisfe2}).  By orthogonality of characters
\begin{equation}\label{eisfe2p1}
    \f{1}{\phi(N)}\sum_{\srel{a\imod
N}{\xi\in \widehat{(\Z/N\Z)^*}}} \widehat{\psi}(a)\,\xi(a)\i\,\xi(d) \ \ = \ \  \left\{
                                                                                \begin{array}{ll}
                                                                                  0\,, & (d,N)>1 \\
                                                                                  \widehat\psi(d)\, , & (d,N)=1\,.
                                                                                \end{array}
                                                                              \right.
\end{equation}
Therefore the right hand side of (\ref{eisfe2}) is equal to
\begin{equation}\label{eisfe2p3}
    N^{(1-n)(1/2+\nu/n)}\,\sum_{\srel{r\in \Z^{n-1}}{\srel{r_1\neq 0 }{d|GCD(r)}}} \widehat\psi(d)\,d^{-\nu+n/2-1}
 \int_{\R^{n-1}} \widetilde f(\widetilde u(y))\,
e(r\cdot y)\, dy\,.
\end{equation}

The compact support of $\widetilde I_\nu\widetilde f$ and
(\ref{eisfe4}) imply the analogous expression for the integral on
the other side of (\ref{eisfe2}), but using the hybrid formula for the restriction of $E_{-\nu,\psi}$ to $U$ in proposition~\ref{eisonu}:
\begin{equation}
\label{eisfe10}
\begin{aligned}
&\int_{U}E_{-\nu,\psi} (u)\, \widetilde I_\nu \widetilde f(u)\,d u \ \  = \ \ N^{\nu-n/2}
\sum_{\srel{v\in \Z^{n}}{\srel{v_1>0}{v_n\neq 0}}}\widehat\psi(v_n)\,
v_1^{\nu-n/2+1}  \ \times
\\
& \ \ \ \ \ \ \ \ \ \ \ \ \ \ \ \times \
\int_{\R^{n-1}} \widetilde I_\nu \widetilde f( u(x))\,
e(v_1v_n x_1)\d_{v_{n-1}/v_1}(x_2)\cdots \d_{v_2/v_1}(x_{n-1}) dx\,.
\end{aligned}
\end{equation}
It is important to note that this sum converges absolutely.  Indeed,
\begin{equation}
\label{eisfe12}
\begin{aligned}
&\sum_{v_2,\,\dots,\,v_{n-1}\in\Z}\  \bigl| \int_\R
\phi( x_1,\textstyle \f{v_{n-1}}{v_1},\,\dots,\,\f{v_2}{v_1})\,  \displaystyle e(v_1v_nx_1)\,dx_1  \bigr|\ \leq
\\
&\qquad\ \ \leq\ C\,v_1^{n-2}\sup_{x_2,\,\dots,\,x_{n-1}\in\R}\ \bigl|  \int_\R \phi(x_1,x_2,\dots,x_{n-1})\, e(v_1 v_n x_1)\,dx_1 \bigr|  \,,
\end{aligned}
\end{equation} for any $\phi \in C_c^\infty(U)$ such as $\phi=\widetilde I_\nu \widetilde f$,
with $C$ depending only on the diameter of the support of $\phi\,$; the supremum on the right decays
faster than any negative power of $|v_1v_n|$.

In view of (\ref{eisfe2p3}) and (\ref{eisfe10}),  a notation change reduces
(\ref{eisfe2}) to the following assertion: under the hypotheses (\ref{eisfe1}) and (\rangeref{eisfe3}{eisfe4}),
\begin{equation}
\label{eisfe13}
\begin{aligned}
&(-1)^\e G_\e(\textstyle{\nu-\f n2}+1)\displaystyle  \sum_{\srel{r\in \Z^{n-1}}{\srel{r_1\neq 0 }{d>0}}} \widehat\psi(d)\, d^{\,n/2-\nu-1} \!\!
\int_{\R^{n-1}} \!\! \widetilde f(\widetilde u(y))\, e\bigl(\,\textstyle \sum_j d \,r_j \,y_j\bigr) dy
\\
& =\ \  \sum_{\srel{d>0}{k\neq 0}}\, \widehat{\psi}(d)    \,k^{\nu- n/2+1} \sum_{r_2, \dots, r_{n-1}\in \Z} \, \int_\R\! \widetilde I_\nu \widetilde f (x_1, \textstyle \f{r_2}{k},\dots , \f{r_{n-1}}{k})  e(dkx_1)\, dx_1\,.
\end{aligned}
\end{equation}
The explicit formula for $I_\nu$ in proposition
\ref{expltwineprop} -- or more accurately, the analogous formula
for $\widetilde I_\nu$ -- implies
\begin{equation}
\label{eisfe14}
\begin{aligned}
&  \int_\R \widetilde I_\nu \widetilde f (x_1, \textstyle \f{r_2}{k},\,\dots , \, \f{r_{n-1}}{k})\,   e(dkx_1)\, dx_1\ =
\\
& = \ \  \int_{\R} \int_{\R^{n-1}}
\widetilde f(\widetilde u(z)) \, e(dkx_1)\,  | {\textstyle \sum}_{j\geq 2} \smallf{r_j z_{n+1-j}}{k} -z_1-x_1
|^{\nu -n/2}\ \times \\
&   \qquad\qquad\qquad\qquad \times \  \sgn({\textstyle\sum}_{j\geq 2} \smallf{r_j z_{n+1-j}}{k} -z_1-x_1)^\e  \,dz\,dx_1
\\
& = \ \  \int_{\R} \int_{\R^{n-1}}
\widetilde f(\widetilde u(z)) \, e(dkx_1+d {\textstyle\sum}_{j\ge 2}r_jz_{n+1-j}-dkz_1) \ \times \\
&   \qquad\qquad\qquad\qquad \times \  | -x_1
|^{\nu -n/2}\, \sgn( -x_1)^\e  \,dz\,dx_1
\\
&= \ \ \int_{\R}|x_1|^{\nu-n/2}\sgn(-x_1)^\e e(dkx_1) \ \times \\
&\qquad\qquad\qquad\qquad \times \   \int_{\R^{n-1}}
\widetilde f(\widetilde u(z))\,e\bigl(-dkz_1 +
d \textstyle\sum_{j\geq 2} r_j \displaystyle
z_{j}\bigr) \,dz\,.
\end{aligned}
\end{equation}
The change of variables $x_1 \mapsto x_1 - z_1 + d^{-1}\sum r_j z_{n+1-j}$ at the second step depends on
interchanging the order of the two integrals. The $z$-integral is an ordinary, convergent integral,
whereas the $x_1$-integral is that of a distribution against a $C^\infty$ function. It can be turned into an
ordinary, convergent integral by repeated integration by parts near $x_1=\infty$ to bring down the real
part of the exponent $\nu-n/2$. Away from infinity the $x_1$-integral already is an ordinary
convergent integral since $\Re \nu \gg 0$; the two phenomena must be separated by a suitable cutoff
function. Our paper \cite{inforder} describes these techniques in detail. They apply equally
to the evaluation of the integral
\begin{equation}
\label{eisfe15}
\int_{\R}\ | x_1 |^{\nu -n/2}\,\sgn(-x_1)^\e\,e(dkx_1)\, dx_1\ = \ (-1)^\e\, |dk|^{n/2-\nu-1}\,G_\e(\nu-\textstyle \f n 2 + 1)\,,
\end{equation}
reducing it to (\ref{eisen3}) in the convergent range.  Identifying $k$ with $r_1$ and summing over $d>0$ and $r\in \Z^{n-1}$, $r_1\neq 0$, gives the identity (\ref{eisfe13}), and hence completes the proof.
\end{proof}

The parameter $\nu$ is natural from the representation theoretic point of view. In applications to functional
equations, we set
\begin{equation}
\label{eisen5} \nu \ = \ n\,s - \rho_{\text{mir}}   \ = \ n(s-1/2)\,,
\end{equation}
which has the effect of translating the symmetry $\nu \mapsto -\nu$ into $s \mapsto 1-s$.

\section{Pairing of Distributions}
\label{pairingsec}

In this section we discuss some  pairings of automorphic distributions that were constructed in
\cite{pairingpaper}, and how the analytic continuation and functional equations of Eisenstein distributions carry
over to these pairings.  In some cases the pairings can be computed as a product of shifts of the functions $G_\d$
defined in (\ref{eisen3}), times certain $L$-functions.  This gives a new construction of these $L$-functions, and
a new method to directly study their analytic properties.  In particular the results here are used crucially in our
forthcoming paper   \cite{extsqpaper}  to give new results about the analytic  continuation that were not available
by the two existing methods, the Rankin-Selberg and Langlands-Shahidi methods.

  We begin with a discussion of the distributional pairings in \cite{pairingpaper}, though not in the same
degree of generality as in that paper.   We consider the semidirect product $G\cdot U$ of a real
linear group $G$ with a unipotent group $U$. We suppose that $G \cdot U$ acts on flag varieties or generalized flag
varieties $Y_j$ of real linear groups $G_{j}$, $1 \leq j \leq r$, in each case either by an inclusion $G\cdot U
\hookrightarrow G_{j}$, or via $G\hookrightarrow G_{j}$ composed with the quotient map $G\cdot U \to G$. Then
$G\cdot U$ acts on the product $Y_1 \times \dots \times Y_r$. We suppose further that
\begin{equation}
\label{pair5o}
\begin{gathered}
G\cdot U\ \ \text{has an open orbit}\  \ \O \subset Y_1 \times \dots \times Y_r\,, \text{and at points of $\O$}
\\
\text{the isotropy   subgroup of $G \cdot U$ coincides with $Z_G$=\,\,center of $G$, }
\end{gathered}
\end{equation}
so that   $\O \simeq (G\cdot U)/Z_G$\,, and that
\begin{equation}
\label{pair5+1o}
\text{the conjugation action of $G$ on $U$ preserves Haar measure on $U$.}
\end{equation}
We let $\G \subset G$, $\G_U \subset U$, $\G_j \subset G_{j}$ denote arithmetically defined subgroups such that
$\G\cdot \G_U \hookrightarrow \G_1 \times \cdots \times \G_r$.

Our theorem    also   involves automorphic distributions $\tau_j \in C^{-\infty}(Y_j,\L_j)^{\G_j}$,
in other words, $\G_j$-invariant distribution sections of $G_{j}$-equivariant $C^\infty$ line bundles $\L_j \to
Y_j$, $1\leq j\leq r$. The exterior tensor product
\begin{equation}
\label{pair6o} \L_1 \boxtimes \dots \boxtimes \L_r \
\longrightarrow \ Y_1 \times \dots \times Y_r
\end{equation}
restricts to a $G\cdot U$-equivariant line bundle over   $\mathcal O \simeq (G\cdot U)/Z_G$. If
\begin{equation}
\label{pair7o}
\begin{gathered}
\text{the isotropy group $Z_G$   acts trivially on the fiber}
\\
\text{of}\,\ \L_1 \boxtimes \dots \boxtimes \L_r\, \ \text{at
points of $\O$},
\end{gathered}
\end{equation}
as we shall assume from now on, the restriction of the line bundle (\ref{pair6o}) to the open orbit $\O$ is
canonically trivial. We can then regard
\begin{equation}
\label{pair8o} \tau\ = \ \text{restriction of}\,\ \tau_1 \boxtimes
\dots \boxtimes \tau_r\,\ \text{to}\,\ \O
\end{equation}
as a scalar valued distribution on   $(G\cdot U)/Z_G$ -- a $\,\G\cdot \G_U$-invariant distribution, since
the $\,\tau_j$ are $\G_j$-invariant:
\begin{equation}
\label{pair9o} \tau \in C^{-\infty}\left((\G\cdot \G_U)\backslash
(G\cdot U)/Z_G\right).
\end{equation}
As the final ingredient, we fix a character
\begin{equation}
\label{pair10o}
\begin{gathered}
\chi \, : U \, \rightarrow \{\,z\in \C^* \mid \ |z|=1\,\}\,\
\text{such that}\,\ \chi (g u g^{-1}) = \chi(u)
\\
\text{for all $g\in G\,,\ u\in U$,\, and}\,\ \chi (\g) = 1\,\
\text{for all $\g\in\G_U$\,}.
\end{gathered}
\end{equation}
Since $\G_U\backslash U$ is compact,
\begin{equation}
\label{pair11o} \left\{\, g \ \mapsto \ \int_{\G_U\backslash U}
\chi(u)\,\tau(ug)\,du\,\right\}\ \in \ C^{-\infty}(\G\backslash   G/Z_G)
\end{equation}
is a well defined distribution on $G/Z_G$ -- a $\G$-invariant scalar valued distribution because of
(\rangeref{pair9o}{pair10o}). Finally, we require that
\begin{equation}
\label{cuspidal3o} \text{at least one of the $\,\tau_i$ is
cuspidal.}
\end{equation}

\begin{thm}{\cite[Theorem~2.29]{pairingpaper}.}\label{thmpairingo}
Under the hypotheses just stated, for every test function $\,\phi\in C_c^{\infty}(G)$, the function
\[
g\ \mapsto \ F_{\tau,\chi,\phi}(g)\ = \ \int_{h\in G}\
\int_{\G_U\backslash U} \chi(u)\,\tau(ugh)\,\phi(h)\,du\,dh
\]
is a well defined $C^{\infty}$ function on $G/Z_G$, invariant on the left under $\G$. This function is integrable
over $\G\backslash G/Z_G$, and the resulting integral
\[
P(\tau_1,\dots, \tau_r) \ \ = \ \ \int_{\G\backslash G/Z_G}\
\int_{h\in G}\ \int_{\G_U\backslash U}
\chi(u)\,\tau(ugh)\,\phi(h)\,du\,dh\,dg
\]
does not depend on the choice of $\phi$, provided $\phi$ is normalized by the condition $\int_G \phi(g)\,dg = 1$.
The $r$-linear map $(\tau_1,\dots,\tau_r) \mapsto F_{\tau,\chi,\phi} \in L^1(\G\backslash G/Z_G)$ is continuous,
relative to the strong distribution topology, in each of its arguments, and relative to the $L^1$ norm on the
image. If any one of the $\tau_j$ depends holomorphically on a complex parameter $s$, then so does
$P(\tau_1,\dots,\tau_r)$.
\end{thm}

  At first glance, the hypothesis (\ref{pair5o}) does not seem to include the hypothesis (2.4b) in
\cite{pairingpaper}. However, since $Z_G$ acts trivially on the orbit $\mathcal O$, the hypothesis (2.4b) does hold
if we replace $G$ by its derived group. Thus, instead of integrating over $\G \backslash G/Z_G$, we could integrate
over $(\G \cap [G,G])\backslash [G,G]/Z_G$. The hypotheses (\rangeref{pair5o}{pair5+1o}) are therefore sufficient
to apply the results of \cite{pairingpaper}.

We shall now describe two interesting cases of this pairing that both involve a similar setup of flag varieties and
the mirabolic Eisenstein series as a factor.  Because we shall work with more than one group and flag variety, we
use subscripts: $G_k$ will denote $GL(k,\R)$ and $X_k = G_k/B_k$  its flag variety; cf.~(\rangeref{ps2}{ps3}). The
Eisenstein distributions $E_{\nu,\psi}$ from (\ref{eisen1}) are $\G_1(N)$-invariant sections of the
line bundle $\L_{\nu-\rho_{\text{mir}},\e}$ over the generalized flag variety $Y_n\cong \R\mathbb P^{n-1}$.
  In addition to these series and representations $W_{\nu,\e}$ and $\widetilde W_{\nu,\e}$, we also consider their  products  with the character $\sgn(\det)^\eta$, $\eta\,\in\,\Z/2\Z$ (see the remark above (\ref{mira5})).
Our two particular pairings depend  crucially on the following geometric fact:\begin{equation} \label{pair7}
\begin{gathered}
G_n\ \ \text{acts on}\  X_n\times X_n \times Y_n\ \text{with a dense open orbit; the action on}
\\
\text{this open orbit is free modulo the center, which acts trivially.}
\end{gathered}
\end{equation}
Indeed, the diagonal action of $G_n$ on $X_n\times X_n$ has a dense open orbit. At any point in the open orbit, the
isotropy subgroup consists of the intersection of two opposite Borel subgroups~-- equivalently, a $G_n$-conjugate
of the diagonal subgroup. That group has a dense open orbit in $Y_n$, and only $Z_n=$ center of $G_n$ acts
trivially.

In the first example, which represents the Rankin-Selberg $L$-function for automorphic distributions $\tau_1$,
$\tau_2$ on $GL(n,\R)$,   the integer $r=3$, $U=\{e\}$, $Y_1=Y_2=X_n$, and $Y_3=\RP^{n-1}$. We require both
$\,\tau_1$ and $\,\tau_2$ to be cuspidal,   but impose no such condition on $\,\tau_3$, which is taken to
be the mirabolic Eisenstein distribution.

The second example, which represents the exterior square $L$-function of a cuspidal automorphic distribution $\tau$
on $GL(2n,\R)$,  involves a nontrivial unipotent group, has $r=2$, and only a single cusp form $\tau_1=\tau$
($\tau_2$ is the mirabolic Eisenstein distribution).
 The
decomposition $\,\R^{2n} = \R^n \oplus \R^n$ induces embeddings
\begin{equation}
\label{pair1} G_n \times G_n \ \hookrightarrow \ G_{2n}\,,\qquad
X_n \times X_n \ \hookrightarrow \ X_{2n}\,.
\end{equation}
The translates of $X_n\times X_n$ under the abelian subgroup
\begin{equation}
\label{pair2} U  \ = \ \left\{ \left. \begin{pmatrix}
\textstyle{I_{n}}  & \textstyle{A} \\ \textstyle{0_{n}}  &
\textstyle{I_{n}}  \end{pmatrix}\ \right| \ A\in M_{n\times
n}(\R)\, \right\} \ \ \subset\  \ G_{2n}
\end{equation}
sweep out an open subset of $X_{2n}$; moreover the various $U$-translates are disjoint, so that
\begin{equation}
\label{pair3} U  \times X_n \times X_n \ \hookrightarrow \
X_{2n}\,.
\end{equation}
Let $\tau \in C^{-\infty}(X_{2n},\L_{\l-\rho,\d})^{\G}$ be a cuspidal automorphic distribution as in (\ref{ps11}),
and $du$ be the Haar measure on $U$ identified with the standard Lebesgue measure on $M_{n\times n}(\R)$. The group
of integral matrices $U(\Z)$ lies in the kernel of the character
\begin{equation}
\label{pair4} \theta : U \ \longrightarrow \ \C^*\,,\qquad
\theta\(\begin{smallmatrix} I_n & A \\ 0_n & I_n \end{smallmatrix}
\) = e(\tr A)\,,
\end{equation}
and because $\G\cap U (\Z)$ has finite index in $U(\Z)$,  the integral
\begin{equation}
\label{pair5}
\gathered S_\theta \tau    =_{\text{def}} \qquad
\qquad\qquad\qquad\qquad\qquad\qquad\qquad\qquad\qquad\qquad\qquad\qquad\ \ \ \ \ \ \
\\  \ \f{1}{\operatorname{covol}(\G\cap U(\Z))}
\int_{\G \cap U(\Z)\backslash U} \theta(u)\, \ell(u)\tau\,du \ \
\in \ \ \ C^{-\infty}(X_{2n},\L_{\l-\rho,\d})
\endgathered
\end{equation}
is well defined, even if $\G$ is replaced by a finite index subgroup. It restricts to a distribution section of
$\L_{\l-\rho,\d}$ over the image of the open embedding (\ref{pair3}). As such, it is smooth in the first variable,
since $\,\ell(u)S_\theta \tau = \theta(u)\i S_\theta\tau\,$ for $u\in U\,$. We can therefore evaluate this
distribution section at $e\in U$, and define
\begin{equation}
\label{pair6} \left. S\tau\ =  \ S_\theta {\tau} \right|_{X_n\times
X_n}  \ \in \ \ C^{-\infty}(X_n\times
X_n,\L_{\l-\rho,\d}|_{X_n\times X_n})^{\G_n}\,.
\end{equation}
Here $\G_n$ is a congruence subgroup of $G_n(\Z)$ whose diagonal embedding into $G_n\times G_n \subset G_{2n}$
leaves $\tau$ invariant under the left action, and preserves $\G\cap U(\Z)$ by conjugation. The superscript
signifies invariance under the diagonal action of $\G_n$  on $X_n\times X_n$. This invariance is a consequence of
the fact that conjugation by the diagonal embedding of any $\g\in \G_n$ also preserves  the character $\theta$ as
well as   $U$, without changing the measure.

We restrict the product of the $G_n$-equivariant line bundles $\L_{\l-\rho,\d}|_{X_n\times X_n}$ and
$\L_{\nu-\rho_{\text{mir}},\e} \to Y_n$ to the open orbit and pull it back to $G_n/Z_n$   ($Z_n=Z_{G_n}$=\,center of
$G_n$\,)\,,   resulting in a $G_n$-equivariant line bundle $\L\to G_n/ Z_n\,$;\, $S\tau\! \cdot\!
E_{\nu,\psi}$ is then a $\G'$-invariant distribution section of $\L$ for
\begin{equation}\label{Gammaprime}
\G' \ \ =  \ \ \G_n\, \cap \, \G_1(N)\,.
\end{equation}
The center $Z_n$ acts on the fibers of $\L$ by the restriction to $Z_n$ of   the character
$\chi_{\l-\rho,\d} \cdot \chi_{\nu-\rho_{\text{mir}},\e}\cdot \sgn(\det)^\eta$, where $\eta\in\Z/2\Z$; recall (\ref{ps5}) and
(\ref{mira4}), and note that $\chi_{\l-\rho,\d}$ takes values on $Z_n$ via its diagonal embedding into
${Z_{2n}}\subset G_{2n}$. We shall assume that $Z_n$ lies in the kernel of $\chi_{\l-\rho,\d} \cdot
\chi_{\nu-\rho_{\text{mir}},\e}\cdot \sgn(\det)^\eta$~-- equivalently,
\begin{equation}
\label{pair8} \l_1 + \l_2 + \dots + \l_{2n}\ = \ 0\,,\qquad \d_1 +
\d_2 + \dots + \d_{2n}\ \equiv \ \e + n\,\eta \pmod 2\,.
\end{equation}
  The first of these conditions involves no   essential loss of generality, since twisting an
automorphic representation by a central character does not affect the automorphy.  The character
$\chi_{\nu-\rho_{\text{mir}},0}$ takes the value $1$ on $Z_n$ regardless of the choice of $\nu$, hence (\ref{pair8}) makes
$\L\to G_n/Z_n$ a $G_n$-equivariantly trivial line bundle. In this situation, $S\tau\! \cdot \! E_{\nu,\psi}$
becomes a $\G'$-invariant scalar valued distribution on $G_n/Z_n$,
\begin{equation}
\label{pair9} S\tau\! \cdot\! E_{\nu,\psi} \in
C^{-\infty}(G_n/Z_n)^{\G'}\,.
\end{equation}
Theorem~\ref{thmpairingo} applies to this specific setting and states
\begin{cor}[\cite{pairingpaper}]
\label{thmpairingcor} Under the hypotheses just stated, for every test function $\,\phi\in C_c^{\infty}(G_n)$
\[
P(\tau,E_{\nu,\psi}) \ \ = \ \   \int_{\G'\backslash G_n/Z_n}
\int_{h\in G_n}\!\!\! \(S\tau \!\cdot \!
E_{\nu,\psi}\)\!(gh)\,\phi(h)\,dh\,dg
\]
does not depend on the choice of $\phi$, provided $\phi$ is normalized by the condition $\int_{G_n} \phi(g)\,dg =
1$. The function $\nu\mapsto P(\tau,E_{\nu,\psi})$ is holomorphic for $\nu \in \C-\{n/2\}$, with at most a  simple
pole at $\nu =  n/2$.
\end{cor}

To make (\ref{pair9}) concrete, we identify $X_{2n} \cong G_{2n}/B_{2n}$, $Y_n \cong G_n/P_n$ as before. We regard
$\tau$ and $E_{\nu,\psi}$ as scalar distributions on $G_{2n}$ and $G_n$ respectively, with $\tau$ left invariant
under $\G \subset G_{2n}(\Z)$, transforming according to $\chi_{\l-\rho,\d}$ on the right under $B_{2n}$, and
$E_{\nu,\psi}$ left invariant under $\G_1(N)\subset G_n(\Z)$, transforming according to
$\chi_{\nu-\rho_{\text{mir}},\e}$ on the right under $P_n$. The averaging process (\ref{pair5}) makes sense also on this
level. When we choose $f_1,\, f_2,\, f_3 \in G_n$ so that $(f_1B_n,f_2B_n,f_3P_n)$ lies in the open orbit, we
obtain an explicit description of $S\tau\! \cdot\! E_{\nu,\psi}\,$,
\begin{equation}
\label{pair10}
\gathered
S\tau\! \cdot\! E_{\nu,\psi} (g)\ \ =\qquad
\qquad\qquad\qquad\qquad\qquad\qquad\qquad\qquad\qquad \qquad\qquad\ \
 \\    \f{1}{\operatorname{covol}(\G\cap U(\Z))}
\int_{\G\cap U(\Z)\backslash U} \theta(u)\
\bigl(\ell(u)\,\tau\bigr)\! \( \begin{smallmatrix} {gf_1} & {0_n}
\\  {0_n} & {gf_2} \end{smallmatrix}\) E_{\nu,\psi} (gf_3)\,du \,.
\endgathered
\end{equation}
We note that the $f_j$ are determined up to {\it simul\-taneous left translation\/} by some $f_0\in G_{n}$ and {\it
individual right translation\/} by factors in $B_n$, respectively $P_n$. Translating the $f_j$ by $f_0$ on the left
has the effect of translating $S\tau\!\cdot\!E_{\nu,\psi}$ by $f_0^{-1}$ on the right; it does {\it not\/} change
the value of $P(\tau,E_{\nu,\psi})$ because the ambiguity can be absorbed by $\phi$. Translating any one of the
$f_j$ on the right by an element of the respective isotropy group affects both $S\tau\!\cdot\!E_{\nu,\psi}$ and
$P(\tau,E_{\nu,\psi})$ by a multiplicative factor~-- a non-zero factor depending on $(\l,\d)$ in the case of $f_1$
or $f_2$, and the factor $\chi_{\nu-\rho_{\text{mir}},\e}(p^{-1})$ when $f_3$ is replaced by $f_3 p$, $p\in P_n$.

One can eliminate the potential dependence on $\nu$ in this factor by requiring $f_3\in U_n\,$; cf. (\ref{mira1}).
Specifically, in the following, we choose
\begin{equation}
\label{fixflag} f_1  \ =  \ I_n \,, \ \ \ f_2  \ = \
\(\begin{smallmatrix}  0 & \cdots & 0 & 1  \\  \vdots & & \iddots
&  \\ 0 & 1 & & \\ 1 & 0 & \cdots & 0 \end{smallmatrix}\) \,, \ \
\ \text{and}\ \ \ f_3 \ =\ \(\begin{smallmatrix}  1 &1 & \cdots &
1  \\  0 &  & &  \\ \vdots &  & \textstyle{I_{n-1}}  &  \\ 0  &  &
&  \end{smallmatrix}\) \,,\end{equation} which do determine a
point $(f_1B_n,f_2B_n,f_3P_n)\in X_n\times X_n\times Y_n$ lying in the open orbit. Note that $f_3\in U_n$ and $f_2
= w_{\text{long}}$, in the notation of (\ref{ab9}).

The pairing $P(\tau,E_{\nu,\psi})$ inherits a functional equation from that of $E_{\nu,\psi}$, which involves the contragredient
automorphic distribution $\widetilde \tau$ defined in (\rangeref{ab10}{ab12}).  The argument  we give below for it
works {\em mutatis mutandis} to provide an analogous statement  for the Rankin-Selberg pairing as well.
\medskip

\begin{prop}\label{penufe}
\begin{multline*} P(\tau,E_{-\nu,\psi}) \ \ = \\
(-1)^{\e+\d_{n+1}+\cdots+\d_{2n}}\, N^{2\nu-\f{\nu}{n}-\f12} \, \prod_{j=1}^n
G_{\d_{n+j}+\d_{n+1-j}+\eta}(\l_{n+j}+\l_{n+1-j}+\smallf{\nu}{n}+\smallf{1}{2})  \ \times \\ \times \ \smallf{1}{\phi(N)}\sum_{\srel{a\imod
N}{\xi\in \widehat{(\Z/N\Z)^*}}}    \widehat{\psi}(a)\,\xi(a)\i\,P\(\ell\(\ttwo{-w_{long}}{}{}{w_{long}} \(\begin{smallmatrix}
N & & & \\
 & I_{n-1} & & \\
 & & N & \\
 & & & I_{n-1}
\end{smallmatrix}
\) \)\widetilde \tau,E_{\nu,\xi}\).
\end{multline*}
\end{prop}

The  pairings on the right hand side are integrations over the quotient $\G^\ast\backslash G_n/Z_n$, where
\begin{equation}\label{Gast}
\G^\ast \ \ = \ \ w_{long}\ttwo{N}{}{}{I_{n-1}}\widetilde \G' \ttwo{N}{}{}{I_{n-1}}\i w_{long} \ \ \ \ \ \ (\,\widetilde\G'\,=\,\{\widetilde\g\,|\,\g\in\G'\}\,)
\end{equation}
 is the subgroup that $S\widetilde\tau\cdot E_{\nu,\xi}$ is naturally invariant under (cf.~(\ref{Gammaprime})).
In the special case that $\tau$ is invariant under $GL(2n,\Z)$, $N=1$, $\psi={\mathbbm 1}$ is the trivial Dirichlet
character, and $\e\equiv \eta\equiv 0\imod 2$, the relation simplifies to
\begin{equation}\label{penufefulllevel}
\gathered
    P(\tau,E_{\nu,{\mathbbm 1}})  \ \ = \qquad
     \qquad \qquad \qquad \qquad \qquad \qquad \qquad  \qquad \qquad  \qquad  \qquad \qquad\\  \ \ \ (-1)^{\d_1+\cdots+
\d_{n}}\, {\prod}_{j\,=\,1}^n
G_{\d_{n+j}+\d_{n+1-j}}(\l_{n+j}+\l_{n+1-j}- \textstyle{\f{\nu}n +
\f 12})\,P(\widetilde{\tau},E_{-\nu,{\mathbbm 1}})\,.
\endgathered
\end{equation}
A similar formula using the second displayed line in proposition~\ref{eisfe} of course also gives a simplified functional equation when $\psi$ is primitive, though we will not need to use this formula in what follows.

\noindent {\bf Proof:} In analogy to   $S\tau\! \cdot\! E_{\nu,\psi}$ in (\ref{pair10}), one can define a product $S\tau\!
\cdot\! \widetilde \rho$  of $S\tau$ and any distribution section $\widetilde \rho$ of $\L_{\nu-\rho_{\text{mir}},\e} \to
\widetilde Y$   as
\begin{equation}
\label{pair12}
\gathered
S\tau \cdot  \widetilde \rho(g)\ \ = \qquad
\qquad\qquad\qquad\qquad\qquad\qquad\qquad\qquad\qquad \qquad\qquad\ \  \ \ \ \ \  \\ \f{1}{\operatorname{covol}(\G\cap U(\Z))}
\int_{\G\cap U(\Z)\backslash U} \theta(u)\
\bigl(\ell(u)\,\tau\bigr)\! \( \begin{smallmatrix} {g
\widetilde{f_2}} & {0_n} \\  {0_n} & {g\widetilde{f_1}}
\end{smallmatrix}\) \widetilde \rho (g\widetilde{f_3})\,du \,.
\endgathered
\end{equation}
Here we have applied the outer automorphism (\ref{ab9}) to the base points $f_1B_n$, $f_2B_n$, $f_3P_n$, and also
switched the order of the two factors $X_n$. This choice of base points is in effect {\it only\/} when we multiply
$S\tau$, or $S\widetilde \tau$, by a section of $\widetilde\L_{\nu-\rho_{\text{mir}},\e} \to \widetilde Y$  such as $\widetilde
E_{\nu,\xi}$ or $I_\nu E_{-\nu,\psi}$\,, rather than by $E_{-\nu,\psi}$\,; it is used internally in this proof, but not elsewhere in the paper.

Though corollary~\ref{thmpairingcor} as stated does not apply to (\ref{pair12}) when $\widetilde\rho=\widetilde
E_{\nu,\xi}$ or $I_\nu E_{-\nu,\psi}$, its conclusions apply so long as $\G'$ is appropriately modified to take into
account the invariance group of $\widetilde\rho$.  This can be seen either as a consequence of the general
statement theorem~\ref{thmpairingo}, or alternatively  deduced directly from corollary~\ref{thmpairingcor} using
the outer automorphism (\ref{ab9}).  Let $\phi\in C_c^\infty(G_n)$ have $\int_{G_n}\phi(h)dh=1$. The proof of the
proposition involves computing the integral
\begin{equation}
\label{pair16}
\mathcal I \ \ = \ \  \displaystyle \int_{\G'\backslash G_n/Z_n} \int_{G_n}
\(S\tau\! \cdot\! I_\nu E_{-\nu,\psi}\)\! ( gh)\,
\phi(h)\,dh\,dg
\end{equation}
in two different ways.  The first involves inserting the formula for $I_\nu E_{-\nu,\psi}$ from
proposition~\ref{eisfe}, obtaining
\begin{equation}\label{pair16b}
\gathered
   \mathcal I \ \ =  \ \ (-1)^\e\,N^{2\nu-\f{\nu}{n}-\f12}\,G_\e(\nu-\smallf{n}{2}+1)\,\smallf{1}{\phi_{\text{Euler\!}}(N)} \  \times \qquad
        \qquad\qquad\qquad\qquad\qquad \ \ \ \\
\sum_{\srel{a\imod
N}{\xi\in \widehat{(\Z/N\Z)^*}}}    \widehat{\psi}(a)\,\xi(a)\i\,
    \int_{\G'\backslash G_n/Z_n} \int_{G_n}\!\!\!
\(S\tau\! \cdot\! \ell\(\ttwo{I_{n-1}}{}{}{N}w_{long}\)\widetilde E_{\nu,\xi}  \)
\! ( gh)\,
\phi(h)\,dh\,dg
\endgathered
\end{equation}
(we have denoted the Euler $\phi$-function as $\phi_{\text{Euler}}$ here in order to avoid confusing it with the smooth function $\phi$ in the integrand).
The integral can be written as
\begin{equation}\label{pair16c}
\gathered
    \f{1}{\operatorname{covol}(\G\cap U(\Z))}  \displaystyle \int_{  \G'\backslash G_n/Z_n} \int_{G_n}
 \int_{\G\cap U(\Z)\backslash U}    \theta(u) \ \times \qquad
 \qquad\qquad \qquad\qquad \\ \qquad \times \
\bigl(\ell(u)\,\tau\bigr)\! \( \begin{smallmatrix} {gh
\widetilde{f_2}} & {0_n} \\  {0_n} & {gh\widetilde{f_1}}
\end{smallmatrix}\)  \widetilde E_{\nu,\xi} \(w_{long}\ttwo{I_{n-1}}{}{}{N}\i gh\widetilde{f_3}\)\,du \,\phi(h)\, dh\,dg\,.
\endgathered
\end{equation}

We now change variables $g \mapsto \widetilde g$, $h\mapsto \widetilde h$, and then apply identities (\ref{ab10})
and (\ref{twoeis}), after which we must replace $\G'$ by $\widetilde{\G}'$:~the integral becomes
\begin{equation}\label{pair16d}
\gathered
    \f{1}{\operatorname{covol}(\G\cap U(\Z))}  \displaystyle \int_{\widetilde \G'\backslash G_n/Z_n} \int_{G_n}
 \int_{\G\cap U(\Z)\backslash U}    \theta(u)\ \times \qquad
        \qquad\qquad\ \   \\ \qquad
        \times \
\widetilde\tau\!\(\widetilde{u}\i  \!\(\begin{smallmatrix} {gh
f_1} & {0_n} \\  {0_n} & {gh f_2}
\end{smallmatrix}\) \)  E_{\nu,\xi} \(w_{long}\ttwo{N}{}{}{I_{n-1}} gh f_3\)\,du \,\phi(\widetilde h)\, dh\,dg\,.
\endgathered
\end{equation}
The above expression is unchanged if both instances of $\G$ are replaced by any finite index subgroup, in
particular the principal congruence subgroup $\G(m)=\{\g\in G_{2n}(\Z)|\g\equiv I_{2n}\imod m\}$ for some $m$ (and
hence any positive multiple of it).  The change of variables $u\mapsto \widetilde u\i=w_{long}u^t w_{long}$
preserves $\G(m)$, $U(\Z)$, $U$, the character $\theta$, and the Haar measure $du$; it allows us to rewrite
(\ref{pair16d}) as
\begin{equation}\label{pair16e}
\gathered
    \f{1}{\operatorname{covol}(\G(m)\cap U(\Z))}  \displaystyle \int_{\widetilde \G'\backslash G_n/Z_n} \int_{G_n}
 \int_{\G(m)\cap U(\Z)\backslash U}    \theta(u)\ \times \qquad\qquad \\ \qquad \times \
\widetilde\tau\!\(u \! \(\begin{smallmatrix} {gh
f_1} & {0_n} \\  {0_n} & {gh f_2}
\end{smallmatrix}\) \)  E_{\nu,\xi} \(w_{long}\ttwo{N}{}{}{I_{n-1}} gh f_3\)\,du \,\phi(\widetilde h)\, dh\,dg\,.
\endgathered
\end{equation}
We may freely replace $\phi(h)$ with $\phi(\widetilde h)$ because corollary~\ref{thmpairingcor}
 guarantees that this substitution of smoothing function does not affect the overall value.  Since \begin{equation}
\label{esfe2a}
\begin{aligned}
{ \widetilde\tau}\( {\ttwo{I_n}{A}{0_n}{I_n}} \ttwo{g_1}{0_n}{0_n}{g_2}\) \ &= \  {
\widetilde\tau}\({\ttwo{-I_n}{0_n}{0_n}{I_n}} {\ttwo{I_n}{A}{0_n}{I_n}}^{-1}
\ttwo{g_1}{0_n}{0_n}{g_2}{\ttwo{-I_n}{0_n}{0_n}{I_n}}\)
\\
&=\ (-1)^{{\d}_{n+1}+\cdots +{\d}_{2n}} \, {
 \widetilde\tau}\({\ttwo{-I_n}{0_n}{0_n}{I_n}}{\ttwo{I_n}A{0_n}{I_n}}^{-1}
\ttwo{g_1}{0_n}{0_n}{g_2}\),
\end{aligned}
\end{equation}
we may replace the $u$ in the argument of $\widetilde\tau$ by $u\i$, so long as we left translate it by
$\ttwo{-I_{n}}{}{}{I_n}$ and multiply the overall expression by $(-1)^{\d_{n+1}+\cdots+\d_{2n}}$.  Replacing $g$
by $g\mapsto \ttwo{N}{}{}{I_{n-1}}\i w_{long}g$ converts $\widetilde\G'$ into $\Gamma^\ast$, and nearly converts
(\ref{pair16e})  into
\begin{equation}\label{pair16ef}(-1)^{\d_{n+1}+\cdots+\d_{2n}}P\(\ell\(
\ttwo{-w_{long}}{}{}{w_{long}}
\(\begin{smallmatrix}
N & & & \\
 & I_{n-1} & & \\
 & & N & \\
 & & & I_{n-1}
\end{smallmatrix}
\)
\)\widetilde \tau,E_{\nu,\xi}\);\end{equation}
the only difference is that the $u$-integration is changed by the presence of these two matrices that left-translate $\widetilde\tau$.
The compensating change of variables in $u$ that undoes this conjugation  preserves the character $\theta$, but
alters $\G(m)$ because some nondiagonal entries are multiplied or divided by $N$.  Were $m$ replaced by $mN$ in
(\ref{pair16e}) this conjugate would still be a subgroup of $\G$, and hence its normalized $u$-integration would have the
same value.
 We conclude that
\begin{equation}\label{pair16f}
\gathered
    \mathcal I \ \ = \ \ (-1)^{\e+\d_{n+1}+\cdots+\d_{2n}} \,N^{2\nu-\f{\nu}{n}-\f 12}\,G_\e(\nu-\smallf{n}{2}+1)\, \smallf{1}{\phi(N)} \ \times \qquad
    \qquad \ \ \  \ \ \  \     \\    \
    \sum_{\srel{a\imod
N}{\xi\in \widehat{(\Z/N\Z)^*}}}    \widehat{\psi}(a)\,\xi(a)\i\, P\(\ell\(
\ttwo{-w_{long}}{}{}{w_{long}}
\(\begin{smallmatrix}
N & & & \\
 & I_{n-1} & & \\
 & & N & \\
 & & & I_{n-1}
\end{smallmatrix}
\)
\)\widetilde \tau,E_{\nu,\xi}\).
\endgathered
\end{equation}

The proof of the proposition now reduces to demonstrating that
\begin{equation}\label{pair16g}
    \mathcal I \ \ = \ \  \f{G_\e(\nu-\smallf n2+1)}{\prod_{j=1}^{n}G_{\d_{n+j}+\d_{n+1-j}+\eta}(\l_{n+j}+\l_{n+1-j}+\smallf{\nu}{n}+\smallf 12)}
    \ P(\tau,E_{-\nu,\psi})\,.
\end{equation}
By combining (\ref{pair12}) and (\ref{pair16}), $\mathcal I$ can be written as \begin{equation}\label{}
\gathered
    \mathcal I \ \ = \ \
   \f{1}{\operatorname{covol}(\G\cap U(\Z))}
    \int_{\G'\backslash G_n/Z_n}\int_{G_n}
\int_{\G\cap U(\Z)\backslash U} \theta(u)  \ \times \qquad\qquad
\\ \qquad \qquad \qquad \times \
\bigl(\ell(u)\,\tau\bigr)\! \( \begin{smallmatrix} {gh\widetilde f_2} & {0_n}
\\  {0_n} & {gh\widetilde f_1} \end{smallmatrix}\) I_\nu E_{-\nu,\psi} (gh\widetilde f_3)\,du     \,\phi(h)\,dh\,dg\,.
\endgathered
\end{equation}
Right translating $h$ by $w_{long}$ converts $h\widetilde f_1=h$ to $hw_{long}=hf_2$, and $h\widetilde
f_2=hw_{long}$ to $h=hf_1$.  It also changes $\phi(h)$ to $\phi(hw_{long})$; however, this change can be undone by
replacing $\phi(g)$ with $\phi(gw_{long})$, as both functions have the same total integral over $G_n$.  Hence
$\mathcal I$ can be expressed as
\begin{equation}\label{pair16i}
\gathered
        \mathcal I \ \ = \ \
   \f{1}{\operatorname{covol}(\G\cap U(\Z))}
    \int_{\G'\backslash G_n/Z_n}\int_{G_n}
\int_{\G\cap U(\Z)\backslash U} \theta(u)  \ \times
\qquad\qquad
\\ \qquad \qquad \qquad \times \
\bigl(\ell(u)\,\tau\bigr)\! \( \begin{smallmatrix} {ghf_1} & {0_n}
\\  {0_n} & {ghf_2} \end{smallmatrix}\) I_\nu E_{-\nu,\psi} (ghw\widetilde f_3)\,du     \,\phi(h)\,dh\,dg\,.
\endgathered
\end{equation}

We shall now use the definition (\ref{mira8}) of the intertwining operator $I_\nu$. Since this involves an integral
over the non-compact manifold $U_n$, it might seem that the formula cannot be applied to the distribution
$E_{-\nu,\psi}$. However, the self-adjointness property (\ref{expltwine7}) justifies the calculations we are about
to present. In effect, the calculations with $E_{-\nu,\psi}$ reflect legitimate operations on the dual side. This is
completely analogous to applying the calculus of differential operators to distributions as if they were functions.
The duality depends on interpreting $\phi$ as a $C^\infty$ section of a line bundle over $X_n\times X_n\times Y_n$,
the mirror image of viewing the distribution section $S\tau\! \cdot \!E_{-\nu,\psi}$ as a   scalar
distribution\begin{footnote}{Strictly speaking, we should work with a smoothing function $\phi\in C^\infty_c(
G_n/Z_n)$ instead of $\phi\in C^\infty_c(G_n)$, but this makes little difference for the rest of the argument.}
\end{footnote} on $G_n$. In effect, we interpret the $h$-integration as the pairing of a distribution section of one
line bundle against a smooth section of the dual line bundle, tensored with the line bundle of differential forms
of top degree, by integration over the compact manifold $X_n\times X_n\times Y_n$. In a slightly different setting,
this process is carried out in the proof of lemma 3.9 in \cite{pairingpaper}. What matters is that $G_n$ acts on
$X_n\times X_n\times Y_n$ with an open orbit.   In any case, applying the definition (\ref{mira8}) of
$I_\nu$, the notation $u(x)$ in proposition \ref{expltwineprop}, and the definition (\ref{ab9}) of the automorphism
$g\mapsto \widetilde g$, we find
\begin{equation}
\label{pair19}
\begin{aligned}
&  \!\!\!\!\!\!\!\!\!  I_\nu E_{-\nu,\psi} (gh w_{\text{long}}\widetilde{f_3}) \\
&=
\int_{\R^{n-1}}  E_{-\nu,\psi} (gh ({f_3}^{-1})^t u(x))\,dx
\\
&=  \int_{\R^{n-1}}\!\!  E_{-\nu,\psi} \textstyle \bigl(g\,h
\,u\!\(\f{\displaystyle{x}}{1-\sum x_j }\)\bigr)\left|1 - \textstyle\sum_j
x_j\right|^{-\nu-n/2} \sgn(1 -\textstyle \sum_j
x_j)^\e dx
\\
&=  \int_{\R^{n-1}}\!\!  E_{-\nu,\psi}  \bigl(g\,h\, u(x)\bigr)\left|1
+ \textstyle\sum_j x_j\right|^{\nu-n/2}\sgn(1
+ \textstyle\sum_j x_j)^\e\,dx\,;
\end{aligned}
\end{equation}
the equality at the second step follows from the transformation law (\ref{mira6explicated}) and the matrix identity
\begin{equation}\label{esfe7}
\begin{aligned}
&\left( \begin{smallmatrix} \ 1 & 0  & \cdots & 0 \\
{ }_{\scriptstyle{-1}} & { }_{\scriptstyle{1}} &  &  \\
 \vdots &  & \ \ \ddots &  \\
 -1 &    &  & 1
\end{smallmatrix} \right)
\left( \begin{smallmatrix} \ 1 & x_{n-1} & \cdots  & x_1 \\
{ }_{\scriptstyle{0}} & { }_{\scriptstyle{1}} &   &  \\
 \vdots & & \ \ \ddots &  \\
 0 &    &  & 1
 \end{smallmatrix}
\right)    \ \ =
\\
&\qquad\qquad\qquad =\ \   \left( \begin{smallmatrix} \ 1 & \f{x_{n-1}}{1-\sum x_j} & \cdots & \f{x_{1}}{1-\sum x_j}  \\
{ }_{\scriptstyle{0}} & { }_{\scriptstyle{1}} & &  \\
 \vdots &  & \ \ \ddots &  \\
 0 &    &  & 1
\end{smallmatrix} \right)
\left( \begin{smallmatrix} \f{1}{1-\sum x_j} & \  0 & \cdots & 0 \\
{ }_{\scriptstyle{-1}} &  &  &  \\
\vdots&{\ }^{\textstyle *} & &  \\
-1 &  &  &
\end{smallmatrix} \right)\,,
\end{aligned}
\end{equation}
and the third step in (\ref{pair19}) from the change of coordinates $x_j\mapsto x_j (1+\sum_j x_j)^{-1}$. To ensure
convergence of the integral -- or rather, of the corresponding integral on the dual side -- we suppose $\re \nu >
n/2-1$.

We now combine (\ref{pair16i}) with (\ref{pair19}). The resulting expression for $\mathcal I$ involves four
integrals: the integrals over $\R^{n-1}$ and $(\Gamma\cap U(\Z))\backslash U$ on the inside -- in either order, since they are
independent -- then the $h$-integral, and finally the integral over $G_n(\Z)\backslash G_n/Z_n$ on the outside. We
claim that we can interchange the order of integration, to put the integration over $\R^{n-1}$ on the
outside\begin{footnote}{The integration over $U(\Z)\backslash U$ must remain on the inside; it is necessary to make
sense of $S\widetilde\tau$ as a distribution section over $X_n\times X_n$.}\end{footnote}: we can use partitions of
unity to make the integrands for all the integrals have compact support.   Then, using the definition of
operations on distributions using the duality between distributions and smooth functions, the expression is converted into one for which
Fubini's theorem applies. In terms of our specific choice of flags (\ref{fixflag}), this means
\begin{equation}\label{pair20b}
\gathered
        \mathcal I \ \ = \ \
   \f{1}{\operatorname{covol}(\G\cap U(\Z))}          \int_{\R^{n-1}} \int_{\G'\backslash G_n/Z_n}
\int_{G_n} \int_{\G\cap U(\Z)\backslash U}
\theta(u)\ \times \ \
\\
 \qquad \qquad \times\  \bigl(\ell(u)\,\tau\bigr)\! \(
\begin{smallmatrix} {g h} & {0_n} \\  {0_n} & {g
h{w_{\text{long}}}} \end{smallmatrix}\) E_{-\nu,\psi}  \bigl(g\,h\,
u(x)\bigr)\left|1 + \textstyle\sum_j
x_j\right|^{\nu-n/2} \\ \times \  \sgn(1 + \textstyle\sum_j
x_j)^{\e}
\phi(h)\,du\,dh\,dg\,dx\,.
\endgathered
\end{equation}
Neglecting a set of measure zero, we may integrate over $(\R^*)^{n-1}$ instead of $\R^{n-1}$.  For $x\in
(\R^*)^{n-1}$, $u(x)$ is conjugate to $f_3$ under the diagonal Cartan subgroup of $G_n$,
\begin{equation}
\label{pair21} u(x)\ = \ a_x^{-1}f_3\,a_x\,,\ \ \ \text{with}\ \
a_x \ \ = \ \ \(
\begin{smallmatrix} 1 & & & & \\  & x_{n-1} & & {\textstyle{0}} & \\ & & \ddots & & \\ & {\textstyle{0}} & & x_2 & \\ & & & & x_1 \end{smallmatrix}
\)  \,.
\end{equation}
We now change variables to replace $h$ by $ha_x$.  The identity
\begin{equation}
\label{pair22} E_{-\nu,\psi}(g\,h\,a_x\,u(x)) \ = \ \left| \,
{\prod}_{j=1}^{n-1}\, x_j\,\right|^{-\f \nu n -\f 12} \!\!\!
\sgn\({\prod}_{j=1}^{n-1}\, x_j\)^\eta
\!\!\!
E_{-\nu,\psi}(g\,h\,f_3)\,.
\end{equation}
follows from the transformation law (\ref{mira4}) because the representation $W_{\nu,\d}$ has been tensored by
$\sgn(\det(\cdot))^\eta$ (see the comments between (\ref{pair6}) and (\ref{pair8})). Similarly the identity
\begin{equation}
\label{pair23} w_{\text{long}}^{-1}\, a_x \,w_{\text{long}}\ \ = \
\ \(
\begin{smallmatrix} x_1^{ } & & & & \\  & x_2 & & {\textstyle{0}} & \\ & & \ddots & & \\ & {\textstyle{0}} & & x_{n-1} & \\ & & & & 1 \end{smallmatrix}
\)
\end{equation}
and the transformation law (\ref{ps9})  imply
\begin{equation}
\label{pair24}
\begin{aligned}
&\!\!\bigl(\ell(u)\,\tau\bigr)\! \( \begin{smallmatrix}
{g h a_x} & {0_n} \\  {0_n} & {g h a_x {w_{\text{long}}}}
\end{smallmatrix}\) \ =
\\
&\ \ \ = \ \(\prod_{j=1}^{n-1} |x_j|^{-\l_{n+j}-\l_{n+1-j}} (\sg
x_j )^{\d_{n+j}+\d_{n+1-j}}\)
\bigl(\ell(u)\,\tau\bigr)\! \( \begin{smallmatrix} {g h}
& {0_n} \\  {0_n} & {g h{w_{\text{long}}}} \end{smallmatrix}\) .
\end{aligned}
\end{equation}
Therefore these characters of the $x_j$ may be moved to the outermost integral in (\ref{pair20b}).
The only remaining instance of $x$ in the inner three integrations is in the argument of the test function, $\phi(ha_x)$.  By the same reasoning as before,    $h\mapsto \phi(ha_x)$ has
total integral one, just like $\phi$.  Since these inner three integrations define the pairing $P(\tau,E_{-\nu,\psi})$, they depend only on this total integral, and hence their value is unchanged if $a_x$ is removed from the argument of $\phi$.
 The $x$-integral in
(\ref{pair20b}) splits off to give
\begin{equation}
\label{pair25b}
\mathcal I \ \ = \ \
\mathcal H \,\times \,P(\tau,E_{-\nu,\psi})\,,
\end{equation}
with
\begin{equation}
\label{pair26}
\begin{aligned}
&\mathcal H \ = \ \int_{\R^{n-1}}\, \left| \,1 +
{\sum}_{j=1}^{n-1} x_j\,\right|^{\nu-n/2}\,\sgn\(1 +
{\sum}_{j=1}^{n-1} x_j\)^\e \ \times
\\
&\qquad\qquad\ \ \ \times \(\,\prod_{j=1}^{n-1}
|x_j|^{-\l_{n+j}-\l_{n+1-j}-\nu/n-1/2} (\sg x_j
)^{\d_{n+j}+\d_{n+1-j}+\eta}\) dx\,.
\end{aligned}
\end{equation}
This integral can be explicitly evaluated: according to lemma \ref{betalikeint} below,
\begin{equation}
\label{pair27}
\gathered
\mathcal H\ = \ (-1)^{\d_2+\cdots + \d_{2n-1}+(n-1)\eta} \ \times \qquad\qquad\qquad\qquad\qquad
\qquad\qquad\qquad\qquad
\\   \times \ \f{G_\e(\nu \! - \! \f n 2 \! + \!
1)\,{\prod}_{j=1}^{n-1}\,G_{\d_{n+j}+\d_{n+1-j}+\eta}  (-\l_{n+j} \! -
\! \l_{n+1-j} \! - \!  \f \nu n \! + \! \f 12)  }{G_{\e+\d_2+\cdots
+\d_{2n-1}+(n-1)\eta}(\nu\! -\! \f n2 \! + \!1\! -\! \l_2 \! + \! \cdots \!
- \! \l_{2n-1}\! - \!\f {n-1}{n} \nu \! +\! \f {n-1}2)}\,.
\endgathered
\end{equation}
At this point, the hypothesis (\ref{pair8}) and the identity
\begin{equation}\label{Gdfe}
G_\d(s)\,
G_\d(1-s) \ \   = \ \  (-1)^\d
\end{equation}
 (which follows directly from (\ref{eisen3b})), establish (\ref{pair16g}) and hence the proposition. \bx

\begin{lem}
\label{betalikeint} For $t\in \R^n$, $t_n \neq 0$, the integral
\begin{multline*}
\int_{\R^{n-1}}  \left| \, t_n  \ \,-\, \ {\textstyle{\sum}_{j=1}^{n-1}} \,t_j\,\right|^{\b_0-1} \, \sgn\( t_n  \
\,-\, \ {\textstyle{\sum}_{j=1}^{n-1}} \,t_j\)^{\eta_0}\ \times \\ \times \  {\prod}_{j=1}^{n-1}\bigl(\,
|t_j|^{\b_j-1}\sgn(t_j)^{\eta_j}\bigr) dt_1\cdots
    dt_{n-1} \, ,
\end{multline*}
converges absolutely when the real parts of $1-\b_0-\b_1-\cdots -\b_{n-1}$   and of the $\b_j$ are all positive. As
a function of the $\b_j$ it extends meromorphically to all of $\C^n$, and equals
\begin{equation*}
\f{G_{\eta_0}(\b_0)G_{\eta_1}(\b_1)\cdots
    G_{\eta_{n-1}}(\b_{n-1})}{G_{\eta_0+\eta_1+\cdots+\eta_{n-1}}(\b_0+
    \b_1+\cdots+\b_{n-1})}\,|t_n|^{\b_0+\b_1+\cdots +\b_{n-1}-1}\,(\sgn t_n)^{\eta_0+\eta_1+\cdots+\eta_{n-1}}\,.
\end{equation*}
\end{lem}
\bigskip

\noindent{\bf Proof:} First we show that absolute convergence implies the formula we want to prove. We let $I(t_n)$
denote the value of the integral. Changing variables appropriately one finds
\begin{equation}
\label{betalikeint1} I(t_n) \ = \ |t_n|^{\b_0+\b_1+\cdots
+\b_{n-1}-1} \sgn(t_n)^{\eta_0+\eta_1+\cdots+\eta_{n-1}}I(1)\,.
\end{equation}
Recall the defining formula (\ref{eisen3}). Integration of the right hand side of the equality (\ref{betalikeint1})
against the function $e(t_n)$ results in the expression
\begin{equation}
\label{betalikeint2} G_{\eta_0+\eta_1+ \dots +\eta_{n-1}}( \textstyle
\sum_{j=0}^{n-1} \b_j)\,I(1)\,,
\end{equation}
whereas multiplication of the actual integral with $e(t_n)$, subsequent integration with respect to $t_n$,
interchanging the order of integration, and the change of variables $t_j \mapsto t_j$ for $1\leq j\leq n-1$, $t_n
\mapsto \sum t_j$, result in the integral
\begin{equation}
\label{betalikeint3} \int_{\R^{n}}  e(t_1+ \cdots +
t_n)\,|t_n|^{\b_0-1}\,\sgn(t_n)^{\eta_0}\,{\prod}_{j=1}^{n-1}\bigl(\,
|t_j|^{\b_j-1}\sgn(t_j)^{\eta_j}\bigr)\, dt_1\cdots
    dt_{n}\,.
\end{equation}
Strictly speaking these integrals converge only conditionally, in the range $\Re\b_j \in (0,1)$. They can be turned
into convergent integrals by a partition of unity argument and repeated integration by parts; for details see
\cite{inforder}. The integral (\ref{betalikeint3}) splits into a product of integrals of the type (\ref{eisen3}).
The explicit formula for this integral, equated to the expression (\ref{betalikeint2}), gives the formula we want
for $I(1)$, and hence for $I(t_n)$. Absolute convergence of $I(t_n)$ in the range $\,\Re \b_j>0$, $\,\Re(\sum
\b_j)<1\,$ can be established by induction on $n$. For $n=2$, the assertion follows from direct inspection. For the
induction step, one integrates out one variable first and uses the uses the induction hypothesis, coupled with the
explicit formula for the remaining integral in $n-2$ variables. \bx

\section{Adelization of Automorphic Distributions}\label{adelize}

The definition of automorphic distribution in \secref{autdistsec} used classical language, as it is better suited
for describing the necessary analysis of distributions on Lie groups.  However, modern automorphic forms  heavily
uses the language of adeles to simplify and organize calculations, especially for general congruence subgroups
$\G$.  In this section, we extend the notions there to the adeles by illustrating two different methods.  In the
first, we use strong approximation to derive an adelization of cuspidal automorphic distributions, analogous to the
usual procedure of adelizing automorphic forms; in the second, we construct adelic Eisenstein distributions
directly.   Both constructions can be adapted to either case, and rely on the analysis in earlier sections at their
core:~it should be emphasized that the role of the adeles here is  nothing more significant than a bookkeeping
mechanism.  However, there are deeper generalizations of this adelization   which simultaneously take into account
embeddings of several components of an automorphic representation.  Such distributions are more complicated, and
are useful for extending our theory to nonarchimedean places and number fields. The section concludes with the
adelic analog of the pairing of the previous section.

For the sake of clarity, we have chosen to give an explicit, detailed discussion of this adelization for the linear algebraic group $GL(n)$ over $\Q$; this suffices for the application in \cite{extsqpaper}.  However, the method generalizes to adelic automorphic representations for arbitrary connected, reductive linear algebraic   groups defined over arbitrary number fields.  We will make comments about the general case after describing the specifics for $GL(n)$ over $\Q$.

We for the most part use standard notation:~$\A$ refers to the adeles of $\Q$,  and $\A_f$ denotes the finite
adeles, i.e., the restricted direct product of all $\Q_p$ with respect to $\Z_p$, $p<\infty$.  If $H$ denotes a
group defined over $\Z$ such as  $G=GL(n)$ or the unit upper triangular matrices $N$, we use the notation $H(R)$ to
represent its $R$-points for the rings
 $R=\Z,\Q,\Q_p,\R,\A,$ and $\A_f$.  The maximal compact subgroup $\prod_{p<\infty}G(\Z_p)$ of $G(\A_f)$ will be denoted by $K_f$.  We often stress membership in one of these groups with an appropriate subscript; for example, the general adele $g_\A \in G(\A)$ can be decomposed as the product $g_\A=g_\infty \times g_2 \times g_3 \times g_5 \times \cdots$, or more concisely as $g_\infty \times g_f$, where the finite part $g_f\in G(\A_f)$ is the remaining product over the primes.  The group $G(\Q)$ sits inside each $G(\Q_p)$, and so at the same time embeds diagonally into $G(\A)$.  In order to avoid confusion here we shall use $G_\Q$ to denote this diagonally-embedded image; likewise, we let $H_\Q\subset G_\Q$ denote the diagonally embedded image of the rational points of an algebraic subgroup $H\subset G$ defined over $\Z$.  Thus strong approximation, for example, asserts that $G(\A)=G_\Q G(\R)K_f$.

Suppose now  that $\pi=\otimes_{p\le\infty}\pi_p$ is an irreducible,  cuspidal adelic automorphic representation of
$G(\A)$, with representation space $U\subset L^2_\omega(G_\Q\backslash G(\A))$ under the right action of $G(\A)$.
Here $\omega$ denotes a   character of the center $Z(\A)$, which we may assume is a finite order character after
twisting $\pi$ by a character of the determinant. Each function $\phi_\A \in U$ restricts to a function $\phi_\R$
on $G(\R)\subset G(\A)$.  Since the representation $\pi$ acts continuously, $\phi_\A$ is stabilized by a congruence
subgroup $K$ of $K_f$.  At the same time it is invariant on the left under $G_\Q$; since the $K_f$ factor commutes
across the $G(\R)$ factor, we conclude that  $\phi_\R$ is left-invariant under a congruence subgroup $\G$ of
$G(\Z)$. The same holds true (with different $K$ and $\G$) if we restrict $\phi_\A$ to a different section of
$G(\R)$ inside $G(\A)$, for example one of the form $G(\R)\times\{g_f\}$:~this is simply the restriction to $G(\R)$
of $\pi(g_f)\phi_\A$.  By strong approximation and the left invariance of $\phi_\A$ under $G_\Q$, this is
tantamount to left translating $\phi_\R=\phi_\A|_{G(\R)}$ by a rational, real matrix whose inverse approximates
$g_f$.   Thus adelic automorphic forms are functions from $G(\A_f)$ to smooth automorphic forms on $G(\R)$.  We
shall use this vantage point as a template for adelizing automorphic distributions.

We now assume, as we may, that $\phi_\A$ corresponds to a nonzero pure tensor for $\pi=\otimes_{p\le\infty}\pi_p$
that is furthermore a smooth vector for $\pi_\infty$.  Right translation by $G(\R)$ commutes with the above
correspondence, so $\phi_\R$ sits inside a classical automorphic representation equivalent to $\pi_\infty$. It is
therefore the image of an embedding of the form (\ref{autodist1}).  By connecting these two constructions, an
automorphic distribution $\tau$ now defines an embedding $J$ of $(\pi_\infty,V_\infty)$ into a subspace $U_\infty$
of  $U$:~the closure of the subspace spanned by right $G(\R)$-translates of $\phi_\A$.

 Again as in \secref{autdistsec},    $\tau$  is a distribution vector for $\pi_\infty'$, and hence may be viewed as a distribution on $G(\R)$ once a principal series embedding  $\pi_\infty'\hookrightarrow V_{\l,\d}$ has been chosen (cf.~(\ref{autodist6})).  In what follows we fix such an embedding.
The above procedure of course  associates a distribution in $C^{-\infty}(G(\R))$ to any right translate of
$\phi_\A$ by $g_f \in G(\A_f)$, a distribution which  is left invariant under a discrete group that depends on
$g_f$.
    Assembling these together, we form a map from $G(\A_f)$ to $C^{-\infty}(G(\R))$ which we call an ``adelic automorphic distribution'' for the  automorphic representation $\pi$.
More concretely, $\tau_\A(g_\A)=\tau_\A(g_\infty\times g_f)$ is defined to be the automorphic distribution in the
variable $g_\infty$ which describes the embedding of $(\pi_\infty,V_\infty)$ into the space $\{$restrictions of
functions in $\pi(g_f)U_\infty$ to $G(\R)\}$.

The fixed principal series embedding for $\pi_\infty'$ naturally exhibits $\pi_\infty$ as the quotient of the dual
principal series $V_{-\l,\d}$.  In particular, we may regard the pairing between $\tau(g_\infty\times g_f)$ and
smooth vectors $v(g_\infty)$ in $V_\infty$ as integration in $g_\infty$ over a flag variety.  We shall use the
following notation generalizing (\ref{autodist4}):
\begin{equation}\label{adelautodist4}
\aligned
    J(v)(h_\infty\times h_f) \ \ & = \ \ \langle\, \tau_\A( g_\infty\times h_f)\,,\,\pi_\infty(h_\infty)\,v(g_\infty) \,\rangle \\ & = \ \  \langle \,\tau_\A(h_\infty  g_\infty\times h_f)\,,\,v(g_\infty) \,\rangle  \,,
\endaligned
\end{equation}
where $g_\infty$ is again the variable of integration in the pairing.

By convention $\tau_\A$ behaves like a function under diffeomorphisms  and is dual to smooth, compactly supported
measures in the $g_\infty$ variable. Right translation of $\tau_\A$ by $G(\A_f)$ corresponds to right translation
of functions in $U$.  The group $G(\A)$ also acts on $\tau_\A$  by left translation,
 \begin{equation}\label{piactsontauleft}
    \(\ell(h_\A)\tau_\A\)(g_\A) \ = \ \tau_\A(h_\A\i g_\A)\,.
 \end{equation}
  This action on $\tau_\A$, restricted to $G(\R)$, is consistent with (\ref{autodist4}) and (\ref{adelautodist4}), but note however that its restriction to $G(\A_f)$ acts on the left (as opposed to on the right, as it does for functions in $U$).  Because the purpose of (\ref{piactsontauleft}) is  merely notational, this discrepancy will be harmless.  Conjugates of the  congruence subgroup $K\subset K_f$ that stabilizes $\phi_\A$ also stabilize  $\tau_\A$:
 \begin{equation}\label{taufixleft}
    (\ell(k)\tau_\A)(g_\infty \times g_f) \ \ = \ \ \tau_\A(g_\infty \times g_f) \ \ \ \ \text{for each~}k\,\in\,g_fKg_f\i\,.
 \end{equation}
 We claim that $G_\Q$ acts trivially on $\tau_\A$ under $\ell$, i.e.,
 \begin{equation}\label{adelautdistpunch2}
    \tau_\A(\g g_\A) \  =  \ \tau_\A(g_\A) \  \ \ \ \text{for each~}\g\in G_\Q\,.
 \end{equation}
 Indeed, writing  $\g$ as  $\g_\infty\times \g_f$,
 this amounts to checking that
 \begin{equation}\label{checktauadelaut1}
    \langle \tau_\A( \g_\infty g_\infty\times \g_f g_f),v(g_\infty) \rangle \ \ = \ \ \langle \tau_\A(g_\infty\times g_f),v(g_\infty) \rangle\, ,
 \end{equation}
 or equivalently,
 \begin{equation}\label{checktauadelaut2}
    J(v)(\g_\infty\times \g_f g_f) \ \ = \ \ J(v)(g_f)
 \end{equation}
 for arbitrary $g_f\in \A_f$ and  smooth vectors $v\in V_\infty$.   The left hand side, $J(v)(\g g_f)$, equals the right hand side because the function $J(v)\in U$ is automorphic under $G_\Q$.

Let us now briefly indicate how this adelization works for a general connected, reductive linear algebraic group defined over a number field $F$ and its adele ring $\A=\A_F$ (we refer to \cite{boreljacquet} as a general reference for the definition, and  facts quoted below).  Let $\phi_\A$ again denote a smooth vector for an automorphic representation $\pi=\otimes_v \pi_v$ of $G(\A)$, where $v$ runs over all places of $F$.  The function $\phi_\A$ on $G(\A)$ is left invariant under the diagonally embedded $G_F$, and is right invariant under a congruence subgroup $K$ of $K_f$, the product of maximal compact subgroups of $G(F_v)$ over all nonarchimedean places $v$ of $F$. Though strong approximation fails in this setting (even for $G=GL(n)$ when the class number of $F$ is greater than 1), the restriction of $\phi_\A$ to $G_\infty$, the product of $G(F_v)$ over all archimedean places $v$, is  left invariant under
\begin{equation}\label{numberfieldchange1}
    \G \ \ = \ \ \{ \, \g \,\in\,G(F) \ | \ \g_f \in K  \,  \}\,,
\end{equation}
regarded as a subgroup of $G_\infty$.  Since $\G$ is an arithmetic subgroup of $G_\infty$, the quotient $Z_\infty \G\backslash G_\infty$ has finite volume, where $Z$ is the maximal $F$-split torus of the center of $G$, and $Z_\infty$ denotes the product of $Z(F_v)$ over all archimedean places $v$.  Automorphic representations are assumed to transform according to a character of the adelic points $Z(\A)$ of $Z$.  Thus, as before, the restriction of a vector in the adelic automorphic representation gives rise to   a classical automorphic representation of the real group $G_\infty$, and hence an automorphic distribution on $G_\infty$ (this uses the fact that the Casselman-Wallach embedding theorem holds for arbitrary real reductive groups).  Right translation by $G(\A_f)$ then allows us to construct an adelic automorphic distribution $\tau_\A$ following the same procedure as before.

 We now return to some features of the earlier discussion about $G=GL(n)$ over $F=\Q$, starting with a description of
   the adelic version of the Whittaker distribution $w_{\l,\d}$ from (\ref{btransform}).  Let $\psi_{+}$
denote  the standard choice of additive character on $\Q\backslash \A$:~the unique such character whose archimedean
component maps $x \mapsto e^{2\pi i x}$.  (What we say below needs to be modified slightly if a different
nontrivial character of $\Q\backslash \A$ is chosen instead.)  There is a standard group homomorphism $c$ defined
on the group of unipotent upper triangular matrices $N$, given by summing the entries just above the diagonal:
\begin{equation}\label{ccharacter}
    c \ : \ (n_{ij}) \ \ \mapsto
    \ \ n_{1,2} \,+\, n_{2,3} \,+ \, \cdots
    \,.
\end{equation}
The composition $\psi_{+}\circ c$ is a nondegenerate  character of $N_\Q\backslash N(\A)$, and is used to define
global Whittaker integrals on the automorphic representation $\pi$:
\begin{equation}\label{classwhit}
    W_{\phi_\A}(g) \ \   = \ \ \int_{N_\Q\backslash N(\A)} \phi_\A(ng) \, \psi_{+}(c(n))\i\,dn\ , \ \ \ \, \phi_\A \,\in\,U \,.
\end{equation}
Here, as usual, $dn$ denotes Haar measure on $N(\A)$, normalized to give the quotient $N_\Q\backslash N(\A)$ volume
equal to 1. Likewise, we define an analogous adelic Whittaker integral for $\pi$ using $\tau_\A$:
\begin{equation}\label{whittonta1}
    w(g) \ \ = \ \ \int_{N_\Q\backslash N(\A)} \tau_\A(ng) \, \psi_{+}(c(n))\i\,dn\, ,
\end{equation}
or more succinctly
\begin{equation}\label{whittonta2}
    w \ \ = \ \ \int_{ N(\A)/N_\Q} \ell(n)  \tau_\A \  \psi_{+}(c(n))\,dn\, .
\end{equation}
Like $\tau_\A$, $w(g)=w(g_\infty\times g_f)$ should be thought of as a function of $g_f\in G(\A_f)$ with values in
$C^{-\infty}(G(\R))$.  Indeed, for any fixed $g_f\in G(\A_f)$,  (\ref{taufixleft}) shows that $\tau_\A$ is
stabilized by a finite index subgroup of $K_f\cap N(\A_f)$; strong approximation then shows this integration is
therefore actually over a finite cover of the compact quotient $N(\Z)\backslash N(\R)$.  Hence it reduces to
(\ref{taucoeffbform}) and gives a valid distribution in the $g_\infty$ variable.  If $v$ is a smooth vector for
$V_\infty$ and $\phi_\A=J(v)$, then it is easily seen that the distribution $w$ embeds $v$ to (\ref{classwhit}).
This is because the pairing between distributions and vectors here involves integration on the right, whereas the
above integrations take place on the left.

When $\phi_\A$ is a pure tensor for $\pi=\otimes_{p\leq \infty}\pi_p$, the integral (\ref{classwhit}) factors into
a product of local Whittaker functions:
\begin{equation}\label{factorizepure1}
    W_{\phi_\A}(g_\infty\times g_f) \ \ = \ \ W_\infty(g_\infty)\,W_f(g_f) \, \ , \ \ \ W_f(g_f) \ \ = \ \ \prod_{p< \infty}W_p(g_p)  \, .
\end{equation}
 Here the $W_p$ lie in the Whittaker model ${\mathcal W}_p$ for $\pi_p$, and are constrained to be the standard spherical Whittaker function (i.e., $W_p|_{G(\Z_p)}\equiv 1$) for almost all primes $p$.  Importantly, by varying the pure tensor $\phi_\A$, the $W_p$ can be chosen arbitrarily in ${\mathcal W}_p$  for any given finite set of primes.
Were we to instead start with  such a modified choice of $\phi_\A\in U$ and construct $\tau_\A$ from it as above,
its adelic Whittaker integral (\ref{whittonta1}) would have a similar factorization:
\begin{equation}\label{factorizepure2}
      w(g_\infty\times g_f) \ \ = \ \ w_\infty(g_\infty)\,W_f(g_f) .
\end{equation}
The distribution $w_\infty \in C^{-\infty}(G(\R))$ coincides with a nonzero multiple of  the distribution
$w_{\l,\d}$ from (\ref{btransform}), where $(\l,\d)$ are the principal series parameters for the Casselman
embedding of $\pi_\infty'$.  The paper \cite{chm} provides a rather complete study of the connection between the
archimedean Whittaker distributions $w_\infty$ and Whittaker functions for general Lie groups.  The remaining
product over primes is itself naturally related to the coefficient in (\ref{canonexttauabel}).

We have therefore shown the following fact, which is useful in constructing adelic automorphic distributions with
prescribed behavior at finite places.
\begin{prop}\label{localchoice}
    Let $\pi=\otimes \pi_p$ be a cuspidal automorphic representation of $GL(n)/\Q$, and $S$ any finite set of primes.  For each $p\in S$ chose a function $W_p$ in the Whittaker model for $\pi_p$, and set $W_p$ equal to the standard spherical vector for each prime $p\notin S$.  Then there exists a pure tensor $\phi_\A$ for $\pi$ whose corresponding adelic automorphic distribution  $\tau_\A$ satisfies (\ref{factorizepure2}).
\end{prop}

A famous theorem independently proven by Piatetski-Shapiro and Shalika \cite{shalmulone,pscorvallis} states that a
smooth vector $\phi_\A \in U$ can be reconstructed as the sum of left translates of its global Whittaker function
(\ref{classwhit}) by coset representatives $\mathcal C$ for $N^{(n-1)}_\Q\backslash GL(n-1)_\Q$, where
$N^{(n-1)}=\{(n-1)\times (n-1)$~unit upper triangular matrices$\}$.  The analogous formula
\begin{equation}\label{tauareconstruct}
    \tau_\A(g) \ \ = \ \ \sum_{\g\,\in\,\mathcal C}w\(\ttwo{\g}{}{}1 g\)
\end{equation}
holds for $\tau_\A$, as a consequence of the above relationships between embeddings of smooth vectors $v \in
V_\infty$.  It can also be proven using Fourier analysis on the nilpotent group $N(\A)$, following along the lines
of the original argument in \cite{shalmulone,pscorvallis}.  In particular, integrating (\ref{tauareconstruct}) over
$N'_\Q\backslash N'(\A)$, where $N'=[N,N]$ is the derived subgroup of $N$, gives the following formula for the
adelization of $\tau_{\text{abelian}}$:
\begin{equation}\label{tauabeladel}
    \tau_{\text{abelian}}(g) \ \ = \ \ \sum_{k\,\in\,\Q^{n-1}}w(D(k)g)\,,
\end{equation}
where $D(k)\in G_\Q$ is the matrix defined just after (\ref{btransform}).  It is evident that $W_f(D(k)g)$ from
(\ref{factorizepure2}) must equal the ratio multiplying $w_{\l,\d}(D(k)g)=w_\infty(D(k)g)$ in
(\ref{canonexttauabel}).  This observation also demonstrates that the normalized coefficients $a_k$ are independent
of the chosen Casselman embedding.

Next we turn to the adelic version of the mirabolic Eisenstein series distributions that were defined and
analytically continued in \secref{eisenstein}.  Though these can be constructed as a special case of the adelic
automorphic distributions just described, it is more useful to construct them directly, and then verify that they
match the earlier construction. Jacquet and  Shalika  studied adelic mirabolic Eisenstein series as part of their
integral representations of the Rankin-Selberg $L$-functions on $GL(n)\times GL(n)$ \cite{jseuler}  and the
exterior square $L$-functions on $GL(2n)$ \cite{jsextsq}.  As we commented earlier, it is also possible to derive
the results here from theirs, using sophisticated machinery of Casselman and Wallach.

 Our adelic construction involves modifying the archimedean data in the Jacquet-Shalika construction in order to mimic the $\d$-function  that is averaged in (\ref{eisen1}), but leaving the nonarchimedean data intact. We begin by recalling the Schwartz-Bruhat space of $\Q_p^{\, n}$, which is the usual Schwartz space in $n$ real variables when $p=\infty$, and is the space of locally constant, compactly supported functions when $p<\infty$.  The latter are precisely the finite linear combinations of characteristic functions of sets of the form $v+p^N\Z_p^{\, n}$, where $v\in \Q_p^{\, n}$ and $N\in \Z$; for $v$ fixed one need only consider $N$ large, because of overlap among  these sets.  The global adelic Schwartz-Bruhat space consists of all finite linear combinations of functions which are global products $\Phi(g)=\prod_{p\le \infty}\Phi_p(g_p)$ of Schwartz-Bruhat functions $\Phi_p$ on $\Q_p^{\,n}$, in which all but a finite number of functions $\Phi_p$ are constrained to be the ``standard unramified choice'' of the characteristic function of $\Z_p^{\,n}$.

 The adelic Eisenstein series distributions are designed to have central character $\omega\i$, the inverse of the central character of $\tau_\A$; this is done in anticipation of the pairing between these objects at the end of the section.  Strong approximation for $\A^*$ equates the double cosets $\Q^*\backslash \A^*/\R^*_{>0}$ to the inverse limit of all $(\Z/N\Z)^*$, $N\in {\mathbb N}$.
 Therefore any Dirichlet character $\psi$, in particular the one in (\ref{eisen1}),
has an adelization to a global character  $\psi_\A=\prod_{p\le \infty}\psi_p$  of $\A^*$
that is trivial on $\Q^*$.\footnote{Please note this identification between Dirichlet and global characters is inverse to the one used by Jacquet-Shalika.}
 We assume for the rest of the paper that
 \begin{equation}\label{globalnthpower}
\psi_\A\ \  = \  \  \chi^n\,\omega\i\,,
\end{equation}
where $\chi$ is also a finite order character of $\Q^*\backslash \A^*$ of parity $\eta\in \Z/2\Z$, consistent with
(\ref{pair8}).

Set $P'$ equal to the $(n-1,1)$ standard parabolic subgroup of $G$, so that $P'=w_{\text{long}}Pw_{\text{long}}$
(cf.~(\ref{mira1})). Jacquet-Shalika form their Eisenstein series as averages of the function
\begin{equation}\label{jacdoes1}
    I(g,s) \ \ = \ \ \chi(\det g)\i\, |\det g|^s \,\int_{\A^*}\Phi(e_n tg)\,|t|^{ns}\,\psi_\A(t)\i\,d^*t\,,
\end{equation}
where $e_n=(0,0,\ldots,0,1)$ is the $n$-dimensional elementary basis row vector. Our construction of the Eisenstein
distribution differs in that we modify the archimedean component $\Phi_\infty$ of each summand of $\Phi$ to be the
$\d$-function of a nonzero point in $\R^n$. To emphasize this distinction, we sometimes refer to Jacquet-Shalika's choice as $\Phi_{\text{JS},\infty}$ and ours as $\Phi_{\text{D},\infty}$.
When  $\Phi(g)=\prod_{p\le \infty}\Phi_p(g_p)$ is a pure tensor, the
integral (\ref{jacdoes1})  splits as a product of local integrals over $\Q_p$, $p\le \infty$, so that
$I(g,s)=I_\infty(g_\infty,s) I_f(g_f,s)$,  $I_f(g_f,s)$ being the product over all $p<\infty$.  The computation of
$I_f(g_f,s)$ is unchanged from the setting of Jacquet-Shalika, but the archimedean integral
\begin{equation}\label{jacdoes1b}
    I_\infty(g_\infty,s) \ \ = \ \ \sgn(\det g_\infty)^\eta\,|\det g_\infty|^s\,\int_{\R^*}
    \Phi_\infty(e_ntg_\infty)\,|t|^{ns}\,\sgn(t)^{\e}\,d^*t
\end{equation}
differs in that it defines a distribution on $G$ instead of a smooth function when $\Phi_\infty=\Phi_{\text{D},\infty}$.
  The local integrals obey the transformation law
  \begin{equation}\label{ipnewtrans}
  I_{p}\(\ttwo{B}{\star}{}{a}g_p\) \ \ = \ \  \psi_p(a)\, \chi_p(a)\i \, |a|^{-(n-1)s}\, \chi_p(\det{B})\i \,|\det{B}|^s\, I_{p}(g_p)\,,
\end{equation}
 as can be seen by the change of variables $t\mapsto t/a$ in the integral.

We shall now describe how the respective local integrals $I_{\text{JS},\infty}$ and $I_{\text{D},\infty}$ are related by right smoothing.  If $\phi$ is any smooth, compactly supported function on $G(\R)$, we may choose
\begin{equation}\label{rightconv1}
 \Phi_{\text{JS},\infty}(v) \ \ = \ \    \int_{G(\R)}\Phi_{\text{D},\infty}(vh)\,\phi(h)\,\sgn(\det h)^\eta\,|\det h|^s\,dh \ \, , \ \ \ \ v \, \in \, \R^n\,,
\end{equation}
since the integral defines a smooth function of compact support in $v$. The  respective local integrals (\ref{jacdoes1b}) of  $\Phi_{\text{JS},\infty}$ and $\Phi_{\text{D},\infty}$ are related by
\begin{equation}\label{rightconv2}
    I_{\text{JS},\infty}(g_\infty,s) \ \ = \ \ \int_{G(\R)}I_{\text{D},\infty}(g_\infty h,s)\,\phi(h)\,dh\,.
\end{equation}
It follows that right convolution of our distributional $I(g,s)$ over $G(\R)$ results in an instance of Jacquet-Shalika's (\ref{jacdoes1}).

We now consider the computation of $I(g,s)$  for a particular type of pure tensor $\Phi$, namely when $\Phi_\infty$
is the $\d$-function supported at $e_1=(1,0,\ldots,0)$  and $\Phi_p$ is the characteristic function of
$e_n+p^{N_p}\Z_p^{\, n}$, where $N=\prod p^{N_p}$ is the factorization of a positive integer $N$. Then
$I_\infty(g,s)$ is  supported on $P'(\R)w_{\text{long}}=w_{\text{long}}P(\R)$ by construction, and is in fact a
constant multiple of the distribution
\begin{equation}
\label{jacdoes2-1}
\d_\infty \ \in \ W_{\nu,\e}^{-\infty}\,\otimes\,\sgn(\det)^\eta
\end{equation}
  defined in (\ref{mira10}), with  $\nu=n(s-1/2)$ (cf.~(\ref{eisen5})). The local integral  for $p<\infty$
is computed as
\begin{equation}
\label{jacdoes2}
I_p(g_p,s) \ \ = \ \   \chi_p(\det g_p)\i|\det g_p|^s \int_{\srel{t\,\in\,\Q_p^*}{tv\,\in\,e_n+p^{N_p}\Z_p^{\,n}}} |t|^{ns}\,\psi_p(t)\i\,d^*t\,,
\end{equation}
where $v$ is the bottom row of $g$. The transformation law (\ref{ipnewtrans}) reduces the computation to $g_p\in
GL(n,\Z_p)$, a set of coset representatives for the subgroup of upper triangular matrices, so that in particular we
may assume $v\in \Z_p^{\,n}$, $p\ndiv v$.  In the case that $N_p=0$, the set in the second constraint is simply
$\Z_p^{\, n}$, and the integration is over $0<|t|\le 1$.  The integral is then a  Tate integral for $L(s,\psi)$:~it
represents $(1-\psi(p)p^{-ns})\i$ if $\psi_\A$ is unramified at $p$, and zero otherwise. If $N_p\ge 1$, the second
constraint reads $tv_j\equiv 0 \imod{p^{N_p}}$ for $j<n$, while $tv_n\equiv 1 \imod{p^{N_p}}$.  This forces $v_n\in
\Z_p-p\Z_p$, and the range of integration to   $t\in v_n\i+p^{N_p}\Z_p$.  The integral vanishes if the ramification
degree of $\psi_p$ exceeds $N_p$, and equals a constant times $\psi(v_n)$ otherwise.

In particular, if $\g \in GL(n,\Z)$, then   $I(\g g_\infty,s)$ is  the product of a constant,
$\d_\infty(g_\infty)$, $L(ns,\phi)$, and the characteristic function of $\G_0(N)$.  The sum of this over all cosets
for $P'(\Z)\backslash G(\Z)$ is precisely the Eisenstein series in  (\ref{eisen1}), up to a constant multiple. The
coset space $P'(\Z)\backslash G(\Z)$ is in bijective correspondence with $P'(\Q)\backslash G(\Q)$ via the inclusion
map $G(\Z)\hookrightarrow G(\Q)$, because of the fact that every invertible rational matrix can be decomposed as an
upper triangular rational matrix times an invertible integral one.   We conclude that with this particular choice
of local data,
 \begin{equation}\label{adeliceis}
    E_\A(g_\A,s) \ \ = \  \ \sum_{\g\in P'_\Q\backslash G_\Q} I(\g g_\A,s)
\end{equation}
is a constant multiple of (\ref{eisen1}) when $g_\A\in G(\R)$, and in particular converges in the strong
distributional topology for $\Re{s}>1$.  Strong approximation reduces the  evaluation of the general $g_\A\in
G(\A)$ to this case, so the sum makes sense in general for $\Re{s}>1$ and defines an adelic automorphic
distribution:~ a map from $G(\A_f)$ to automorphic distributions in $C^{-\infty}(G(\R))$.  Because of (\ref{rightconv2}), the right smoothing of $E_\A(g_\A,s)$ over $G(\R)$ results in a smooth Eisenstein series on $G(\Q)\backslash G(\A)$ considered by Jacquet-Shalika.  Thus $E_\A$ is also an automorphic distribution in the earlier sense of a distribution which embeds into smooth automorphic forms.

The general choice of local data involves broader choices in two respects:~$\Phi_\infty$ may be a $\d$-function
supported at another nonzero point, and $\Phi_p$ may be the characteristic function of $v+p^N\Z_p^{\,n}$, $N$
large. Right translating $E_\A(g_\A,s)$  by some $h\in GL(n,\A)$ has the effect of replacing $\Phi(v)$ by
$\Phi(vh)$.  Since $GL(n)$ acts with two orbits on $n$-dimensional vectors, this means the general $\d$-function
for $\Phi_\infty$ can be reduced to the case above, and that the characteristic functions  for $\Phi_p$ can be
reduced to the situation that $v=0$ or $v=e_n$.  Since $e_n+\Z_p^{\, n}=\Z_p^{\, n}$, the sets $e_n+p^N\Z_p^{\, n}$
for $N\ge 0$ we considered above indeed cover all possibilities.  Thus the analytic properties of the general
instance of (\ref{adeliceis}) for linear combinations of such pure tensors $\Phi$ reduce to those we have just
considered.  In particular they have a meromorphic continuation to $s\in \C-\{1\}$, with at most a simple pole at
$s=1$ that occurs only when $\psi$ is trivial.

Finally, we conclude by writing the general form of the automorphic pairing in terms of adeles, generalizing
(\ref{thmpairingo}).  We need to slightly adapt the notation there to the adelic setting.  Let $U$ denote  the
algebraic group
  \begin{equation}\label{newU}
    U  \ = \ \left\{ \left. \begin{pmatrix}
\textstyle{I_{n}}  & \textstyle{A} \\ \textstyle{0_{n}}  &
\textstyle{I_{n}}  \end{pmatrix}\ \right| \ A\in M_{n\times
n}\, \right\} \ \ \subset\  \ GL(2n)
\end{equation}
 whose real points were previously denoted by $U$  in (\ref{pair2}).
The character $\theta$ from  (\ref{pair4}) has a natural adelic extension,
  \begin{equation}\label{pair4a}
   \theta : U(\A)\ \longrightarrow \ \C^*\,,\qquad
\theta\(\begin{smallmatrix} I_n & A \\ 0_n & I_n \end{smallmatrix}
\) = \psi_{+}(\tr A)\,,
\end{equation}
 where $\psi_{+}$ is the additive character defined just above (\ref{ccharacter}).  Let $du$ denote the Haar measure on $U(\A)$ which gives the quotient $U_\Q\backslash U(\A)$ volume 1.

With $f_1$, $f_2$, and $f_3$ still standing for flag representatives in $G(\R)$ and $\psi\in C_c^\infty(G(\R))$
having total integral 1, the general adelic pairing is defined as
\begin{equation}\label{adelicpairing}
\gathered
    P(\tau_\A,E_\A(s)) \ \ =  \qquad\qquad\qquad\qquad\qquad\qquad\qquad
    \qquad\qquad\qquad\qquad\qquad  \\
    = \ \int_{Z(\A)G_\Q\backslash G(\A)}
    \int_{G(\R)}
    \left[ \int_{U_\Q\backslash U(\A)}
    \tau_\A\(u\!\ttwo{ghf_1}{}{}{ghf_2}\)\overline{\theta(u)}du   \right] \, \times \\ \qquad\qquad\qquad\qquad \qquad\qquad\ \ \ \ \times \
    E_\A(ghf_3,s)\,\psi(h)\, dh\,dg\,.
\endgathered
\end{equation}
Several comments are in order to explain why the above makes sense.
 Firstly, for the same reason as in  (\ref{whittonta1}), the bracketed   inner integration is  over a finite cover of the compact quotient $U(\Z)\backslash U(\R)$, and so defines a map from $G(\A_f)$ to distributions in $G(\R)$ that  corresponds to (\ref{pair5}).  This map is left invariant under the diagonal rational subgroup $G_\Q$   because of (\ref{adelautdistpunch2}),  and because conjugation through $u$ changes neither $\theta(u)$ nor the measure $du$.  It is also invariant under $Z(\A)$ because
     (\ref{pair8}) ensures that the central characters of $\tau_\A$ and $E_\A$ are inverse to each other.
  The invariance under both $G_\Q$  and   $Z(\A)$ is not affected by the second integration, which only involves $h$ on the right. The second integration simultaneously smooths both the bracketed expression and $E_\nu(ghf_3)$ over $G(\R)$:~it gives a map from $G(\A_f)$ to smooth automorphic functions on $G(\R)$.    According to corollary~\ref{thmpairingcor} these restrictions to $G(\R)$ are each integrable over their fundamental domain.   Because of (\ref{taufixleft}) and strong approximation, the last integration  takes place on a finite  cover of $Z(\R)G(\Z)\backslash G(\R)$ -- again by the same reasoning used for the bracketed inner integration in (\ref{adelicpairing}),  and  for (\ref{whittonta1}) before it.  Corollary~\ref{thmpairingcor} shows that the last integral is independent of the choice of $\psi$, assuming its normalization  $\int_{G(\R)}\psi(g)dg=1$.

The above pairing inherits the meromorphic continuation to $s\in \C-\{1\}$ that its classical counterpart possesses
(corollary~\ref{thmpairingcor}), as well as a functional equation from (\ref{penufe}):
\begin{equation}\label{adelicpenufe}
\gathered
    P(\tau_\A,E_\A(1-s)) \ \ =  \qquad\qquad\qquad\qquad\qquad\qquad\qquad\qquad\qquad\qquad\qquad   \\
    \qquad  N^{2ns-s-n}\,\prod_{j\,=\,1}^n G_{\d_{n+j}+\d_{n+1-j}+\eta}(s+\l_{n+j}+\l_{n+1-j})\,P(\tau'_\A,E'_\A(s))\,,
\endgathered
\end{equation}
where $\tau'_\A$ and $E'_\A$ correspond to the translated contragredient cusp form $\widetilde\tau$ and sum of the
remaining Eisenstein data, respectively, from the right hand side in proposition~\ref{penufe}.  This formula
simplifies when both $\pi_p$ and the Eisenstein data $\Phi_p$ are unramified at all $p<\infty$ (which put us in the
situation that $N=1$).  If $\Phi_\infty$ is the delta function at $e_1\in \R^n$, then
$E_\A'(s)=(-1)^{\d_1+\cdots+\d_n}E_\A(s)$ and $\tau'_\A=\widetilde\tau_\A$ (cf.~(\ref{penufefulllevel})).

\appendix

\section{Archimedean components of automorphic representations on $GL(n,\R)$}\label{appendix}

 Recall from \secref{autdistsec} that we study automorphic distributions in terms of the embedding
(\ref{autodist6}) of $\pi_\infty'$ into principal series representations $V_{\l,\d}$. These embeddings are not unique. For full principal series representations, the parameters $(\l,\d)$ are determined only up to simultaneous permutation of the $\l_j$ and $\d_j$. In general, there is a smaller choice of embedding parameters. On the other hand, the Gamma factors predicted by Langlands also depend on the nature of the archimedean component of the automorphic representation in question. We use this connection between multiple embeddings and Gamma factors to exclude unwanted poles of $L$-functions.

In this appendix we collect the relevant results about embeddings into principal series and Langlands Gamma factors. All of these are well known to experts, but do not appear in the literature -- at least not in convenient form.

\subsection{The Generic unitary dual of $GL(n,\R)$ and embeddings into the principal series.}
\label{langclasssection}

The possible real representations of $GL(n,\R)$ that can occur as the archimedean component $\pi_\infty$ of a
cuspidal automorphic representation $\pi$ are extremely limited by a number of local and global constraints.  The
latter are extremely subtle, and hence a complete classification seems hopeless at present.  In this subsection we
will instead describe the representations that satisfy perhaps the most well known local constraints for
$\pi_\infty$, namely those that are unitary and generic (i.e., have a Whittaker model).

The unitary dual for $GL(n,\R)$ was first described by Vogan \cite{vogan}, and later by Tadi{\'c} \cite{tadic}
using different methods.  Tadi{\'c} describes the unitary dual as certain parabolically induced representations
from an explicit set $\mathcal B$ of representations of $GL(n',\R)$, $n'\le n$. He also proves that permuting the
order of the induction data yields the same irreducible representation of $GL(n,\R)$.  His set $\mathcal B$ is
defined in terms of not only induced representations of   square integrable (modulo the center)
representations   of $GL(1,\R)$ and $GL(2,\R)$, but also certain irreducible quotients. These quotients,
however, are not ``large'' in the sense of \cite{voganlarge}, and hence neither are any representations induced
from them. It is a result of Casselman, Zuckerman, and Kostant (see \cite{kostant}) that all generic
representations of $GL(n,\R)$ are large, and conversely that all large representations are generic.

Hence Tadi{\'c}'s list gives a description of the generic unitary dual, once these quotients are removed from
$\mathcal B$.  We now summarize this description, after making further simplifications using transitivity of
induction. Let $n=n_1+\cdots+n_r$ be a partition of $n$, and let $P\subset G=GL(n,\R)$ be the standard parabolic
subgroup   of block upper triangular matrices corresponding to this partition.  The Levi subgroup $M$ of $P$ is
isomorphic to $GL(n_1,\R)\times \cdots \times GL(n_r,\R)$.  Let $\sigma_i$ denote   an irreducible,
  square integrable   (modulo the center)   representation of $GL(n_i,\R)$. This forces
$n_i$ to  equal  1 or 2, and $\sigma_i$ to be one of the following possibilities:
\begin{enumerate}
    \item If $n=1$, $\sigma_i$ is either the trivial representation of $GL(1,\R)\simeq \R^*$, or else the sign
        character $\sgn(x)$.
    \item If $n=2$, $\sigma_i$ is a discrete series representation $D_k$ (indexed to correspond to holomorphic
        forms of weight   $k$\,,\,\ $k\geq 2$\,).
\end{enumerate}
  These representations are self dual.   For each $1\le i \le r$ and $s_i\in\C$, the twist
 $\sigma_i[s_i]=\sigma_i\otimes |\det(\cdot)|^{s_i}$
defines a representation of $GL(n_i,\R)$. The tensor product of these twists defines a representation of $M$ which
extends to $P$ by allowing the unipotent radical of $P$ to act trivially.
 Let $I(P;\sigma_1[s_1],\ldots,\sigma_r[s_r])$ denote the
representation of $G$ parabolically induced from this representation of $P$,  where the induction is normalized to
carry unitary representations to unitary representations.  In order to be consistent with the conventions of
\cite{langlandsrealgroups}, the group action in this induced representation operates on the right, on functions
which transform under $P$ on the left.

We now give the constraints on the parameters $s_i$ that govern precisely when
$I(P;\sigma_1[s_1],\ldots,\sigma_r[s_r])$ is irreducible, generic, and unitary according to the results of
Casselman, Kostant, Tadi{\'c}, Vogan, and Zuckerman mentioned above.  We assume that this representation is
normalized to have a unitary central character, as we of course may by tensoring with a character of the
determinant.

\begin{itemize}

    \item {\bf Unitarity constraint:}  the multisets $\{\sigma_i[s_i]\}$ and $\{\sigma_i[-\overline{s}_i]\}$
        must be equal, i.e., these lists are equal up to permutation (recall the  $\sigma_i$ are self dual).
        This is because the representation dual to $I(P;\sigma_1[s_1],\ldots,\sigma_r[s_r])$ is
        $I(P;\sigma_1[-s_1],\ldots,\sigma_r[-s_r])$ .
    \item {\bf Unitary dual estimate:} $|\Re{s_i}|<1/2$.  In the case of the principal series, this is commonly called the  ``trivial bound''.
     \item {\bf Permutation of order:} for any permutation $\tau\in S_r$, the induced representations
         $I(P;\sigma_1[s_1],\ldots,\sigma_r[s_r])$ and
         $I(P^\tau;\sigma_{\tau(1)}[s_{\tau(1)}],\ldots,\sigma_{\tau(r)}[s_{\tau(r)}])$ are equal, where
        $P^\tau$ is the standard parabolic whose Levi component is $GL(n_{\tau(1)},\R)\times\cdots\times
        GL(n_{\tau(r)},\R)$.
\end{itemize}

The principal series   representations   $V_{\l,\d}$ in (\ref{ps8}) are  induced representations,
but induced from a lower triangular Borel subgroup (\ref{ps2}).   Our convention is well-suited for studying
automorphic distributions, but induction from an upper triangular Borel subgroup is the more common convention in
the literature on Langlands' classification of representations of real reductive groups \cite{langlandsrealgroups}
(e.g., his prediction of $\G$-factors for automorphic $L$-functions).   Using the  Weyl group element
$w_{\text{long}}$ from (\ref{ab9}) and the inverse map between the two, it is straightforward to show that
$V_{\l,\d}$ is equivalent to   $I(B_+;\sgn^{\d_n}[\l_n],\ldots,\sgn^{\d_1}[\l_1])$, where $B_+$
is the upper triangular Borel subgroup of $GL(n,\R)$. More generally, induction on the right from a lower
triangular parabolic involves reversing the order of the inducing data, though the order is irrelevant for the
representations in Tadi{\'c}'s classification of the unitary dual anyhow.

Embeddings into principal series  are of course  tautological for $n=1$, where all irreducible representations are
one dimensional. When $n=2$, the discrete series   representation   $D_k$ is a subrepresentation of
the principal series   representation   $V_{\l,\d}$ with para\-meters $\l=(-\f{k-1}{2},\f{k-1}{2})$
and $\d=(k,0)$. This embedding is not unique:~actually $D_k\otimes \sgn\simeq D_k$, so $\d=(k+1,1)$ is an equally
valid parameter. An irreducible principal series   representation   $V_{(\l_1,\l_2),(\d_1,\d_2)}$
embeds not only into itself, but also into $V_{(\l_2,\l_1),(\d_2,\d_1)}$. However, $D_k$ is not a
subrepresentation, but instead a quotient, of the representation $V_{(\f{k-1}{2},-\f{k-1}{2}),(0,k)}$. If
$\rho_1\hookrightarrow \rho_2$, then $\rho_1[s]\hookrightarrow \rho_2[s]$.  The twist  $V_{\l,\d}[s]$ is the
principal series   representation   $V_{\l+(s,s,\ldots,s),\d}$, so $D_k[s]$ embeds both into
$V_{(s-\f{k-1}{2},s+\f{k-1}{2}),(k,0)}$ and also $V_{(s-\f{k-1}{2},s+\f{k-1}{2}),(k+1,1)}$.  The description above shows that these are the only types of unitary generic representations of $GL(2,\R)$.

Next we move to $GL(n,\R)$ and consider a unitary, generic representation $\pi_\infty=I(P;\sigma_1[s_1], \ldots,
\sigma_r[s_r])$ as above.  Embeddings for $\pi_\infty'=I(P;\sigma_1[-s_1], \ldots,
\sigma_r[-s_r])$ may be deduced from the previous paragraph, using the principle of transitivity of induction as
follows. Let $k_i$ denote the weight of the discrete series in block $i$ (provided $n_i=2$, of course). Now define
vectors $\l\in \C^n$ and $\d\in (\Z/2\Z)^n$ in the following manner. If the integer $1\le j \le n$ is contained in
the $i$-th block $n_i$ of the partition $n=(n_1,\ldots,n_r)$, set $\l_j$ to be
\begin{equation}\label{setlitobe}
   \l_j \ \ = \ \  \left\{%
\begin{array}{cl}
    -s_i\,, & n_i=1\,; \\
 -  s_i -\f{k_i-1}{2}\,, & n_i=2 \text{~~and~~}j=n_1+\ldots+n_{i-1}+1\,; \\
   - s_i +\f{k_i-1}{2}\,, & n_i=2 \text{~~and~~}j=n_1+\ldots+n_{i-1}+2\,. \\
\end{array}%
\right.
\end{equation}
Similarly, set
\begin{equation}\label{setdtobe}
    \d_j \ \ \equiv \ \ \left\{%
\begin{array}{cl}
     \e\,,  & n_i=1\text{~~and~~}\sigma_i=\sgn(\cdot)^\e\,; \\
    k_i \,, &  n_i=2 \text{~~and~~}j=n_1+\ldots+n_{i-1}+1\,; \\
   0 \, , & n_i=2 \text{~~and~~}j=n_1+\ldots+n_{i-1}+2\,.
\end{array}%
\right.
\end{equation}
One may alternatively replace $k_i$ and $0$ in the last two cases by $k_i+1$ and 1, respectively.
 In other words, $\l$ and $\d$ are formed by
concatenating the corresponding vectors  which describe the embedding parameters for the $\sigma_i[-s_i]$, $1\le
i\le r$.
 By transitivity of induction,
$\pi_\infty'=I(P;\sigma_1[-s_1],\ldots,\sigma_r[-s_r])$ is a subrepresentation of $V_{\l,\d}$.

\subsection{Langlands'
$\G$-factors}\label{langdescript}

The $\G$-factors which accompany an automorphic $L$-function $L(s,\pi,\rho)$ in its functional equation are
conjectured to always be products, with shifts, of the functions
\begin{equation}\label{gammarc}
    \G_\R(s) \, = \, \pi^{-s/2}\,\G(s/2) \ \ \ \ \text{and} \ \ \ \
     \G_\C(s) \, = \, 2\,(2\pi)^{-s}\,\G(s) \, = \, \G_\R(s)\,\G_\R(s+1)\,.
\end{equation}
Langlands  \cite{langlandsrealgroups} gives a procedure to compute this archimedean factor $L_\infty(s,\pi,\rho)$
in terms of his description of $\pi_\infty$ as a subquotient of an induced representation, along with a calculation
involving the $L$-group representation $\rho$ and the Weil group. When dealing with the group $GL(n)$, however, it
is much more convenient to avoid the Weil group, and instead describe these $\G$-factors in terms of the (freely
permuted) induction data.  We give a description of this for some notable examples, following the description in
\cite{knapparch}.

It is  convenient to use Langlands' {\em isobaric} notation \cite{langlandsmarchen} for induced representations
\begin{equation}\label{isobardefin}
  \pi_\infty \ \ = \ \  I(P;\sigma_1[s_1],\ldots,\sigma_r[s_r]) \ \ = \ \ \sigma_1[s_1]
   \, \isobar \,\cdots \,\isobar \,\sigma_r[s_r]\, ,
\end{equation}
in which the operation  $\isobar$ on the right hand side should be thought of as a formal, abelian addition. Recall
that the classification in section~\ref{langclasssection} shows that every generic unitary representation of $GL(n)$
is an isobaric sum of the form (\ref{isobardefin}), independent of the order. We use these formal sums here only as
a bookkeeping device used to define $\G$-factors; they do not always correspond to irreducible, archimedean
components of cuspidal automorphic representations.  This formal addition satisfies the following two properties.
First, two isobaric sums $\Pi_1$, $\Pi_2$ may themselves be concatenated into a longer isobaric sum
$\Pi_1\isobar\Pi_2$. Second, an isobaric sum can be twisted by the rule $(\Pi_1\isobar\Pi_2)[s]=\Pi_1[s]\isobar
\Pi_2[s]$.

We shall explain how to define $L(s,\Pi)$ for such a formal sum $\Pi$ of twists of the $\sigma_i$, and how $\rho$
transforms $\Pi$ into another such formal sum $\rho(\Pi)$ for some examples of representations $\rho$ of
$GL(n,\C)\,=\,^L GL(n)^0$. Then $L_\infty(s,\pi,\rho)$ is defined as $L(s,\rho(\Pi))$, where $\Pi$ is an isobaric sum
for $\pi_\infty$. We start with the definition of $L(s,\Pi)$ when $\Pi$ is one of the basic building blocks
$\sigma_i$, the self-dual, square integrable representations from \secref{langclasssection}:
\begin{equation}\label{lsigma1}
    L(s,triv) \ \ = \ \ \G_\R(s) \  \ , \ \ \ \ \ \ L(s,\sgn)
    \ \ = \ \ \G_\R(s+1)\ ,
\end{equation}
\begin{equation}\label{lsigma2}
   \text{and} \ \ \ \ \ \  \ \ \ L(s,D_k) \ \ = \ \ \G_\C(s+\textstyle{\f{k-1}{2}})\,.
\end{equation}
Next are rules for isobaric sums and twists:
\begin{equation}\label{rules}
    L(s,\Pi[s']) \, = \, L(s+s',\Pi) \ \ \  \ \ \text{and} \ \ \ \
    L(s,\Pi_1\,\isobar \,\Pi_2) \, = \, L(s,\Pi_1)\,L(s,\Pi_2)\,.
\end{equation}
Therefore $L(s,\Pi)$, for a general isobaric sum $\Pi= \sigma_1[s_1]
    \isobar \cdots \isobar \sigma_r[s_r]$, is given by
\begin{equation}\label{gammind}
\aligned
    L(s,\Pi)  \ \   = \ \ \prod_{i=1}^r L(s+s_i,\sigma_i)\,,
\endaligned
\end{equation}
and is explicitly determined by the definitions (\ref{lsigma1}-\ref{lsigma2}).

Let now $\Pi=\Pi_1\,\isobar\,\Pi_2\,\isobar\cdots\isobar\,\Pi_r$ be an isobaric representation of $GL(n,\R)$, and
$\Pi'=\Pi_1'\,\isobar\,\Pi_2'\,\isobar\cdots\isobar\,\Pi_{r'}'$ be an isobaric representation of $GL(m,\R)$.   The
isobaric sum for the Rankin-Selberg tensor product representation $\Pi\times\Pi'$ of $GL(nm,\R)$ is given by
\begin{equation}\label{tensorisobar}
   \Pi \, \times \, \Pi' \ \ = \ \  \isobar_{j=1}^r\isobar_{k=1}^{r'}
  \( \Pi_j \,\times\,\Pi_k'\)\,,
\end{equation}
where now the meaning of $\Pi_j \times \Pi_k'$ must be explained.
 It is in general {\em not} the usual tensor product of two
representations (more on this below). One has the relations
\begin{equation}\label{tib2}
    \Pi[s]\,\times\,\Pi'[s'] \ \ = \ \
    \(\Pi\,\times\,\Pi'\)[s+s']\,
\end{equation}
and
\begin{equation}\label{tib3}
    \Pi\,\times\,\Pi' \ \ = \ \ \Pi'\,\times\,\Pi\, ,
\end{equation}
which along with (\ref{tensorisobar}) may be regarded as formal rules for the calculation of tensor product on
isobaric representations.  They boil  the general calculation down to the examples  of  $\sigma\times\sigma'$,
where $\sigma,\sigma'\in \{triv,\sgn,D_k\, | \, k\ge 2\}$.  First, if $\sigma$ or $\sigma'$ is one of the
representations $triv$ or $\sgn$, then the Rankin-Selberg product corresponds to the usual tensor product. The only
other case is when $\sigma$ and $\sigma'$ are both discrete series representations of $GL(2,\R)$.  In this
situation one has $D_k\times D_\ell = D_{k+\ell-1}\isobar D_{|k-\ell|+1}$. In summary $\sigma\times\sigma'$ is
given by the following table:
    \begin{center}\begin{tabular}{|c|ccc|}
\hline
 $\sigma \ \diagdown \ \sigma'$ &  $triv$ &  $\sgn$ & $D_k$ \\
\hline
 $triv$ &  $triv$ &  $\sgn$ & $D_k$ \\
 $\sgn$ &  $\sgn $&  $triv $& $D_k$ \\
 $D_\ell$ &  $D_\ell$ & $ D_\ell$ & $D_{k+\ell-1}\isobar
  D_{|k-\ell|+1}$ \\
\hline
\end{tabular}\end{center}
If $k=\ell$ there is no representation $D_1$, yet we use the convention (\ref{lsigma2}) to write
$L(s,D_1)=\G_{\C}(s)$.  In light of (\ref{gammarc}), it is equivalent to regard $D_1$ as $triv\,\isobar\,\sgn$.

We now come to the exterior square representation $Ext^2$ that maps  $GL(n)\rightarrow GL(\f{n(n-1)}{2})$.  It
satisfies the following formal rules:
\begin{equation}\label{ext2ib}
    Ext^2 \(\isobar_{j=1}^r \Pi_j\) \ \ = \ \
     \(\isobar_{j=1}^r Ext^2 \Pi_j\) \,\isobar\,\(
    \isobar_{1\le j<k\le r}\, (\Pi_j\times\Pi_k)\)\,
\end{equation}
and
\begin{equation}\label{esib}
    Ext^2 \(\Pi[s]\) \ \ = \ \ \(Ext^2\Pi\)[2s]\,.
\end{equation}
Similarly to the above situation of tensor products, it is completely determined by the following table:
\begin{center}\begin{tabular}{|c|ccc|}
\hline
 $\sigma$ &  $triv$ &  $\sgn$ & $D_k$ \\
\hline
 $Ext^2\sigma$ &  $\emptyset$ &  $\emptyset$ & $\sgn^k$ \\
\hline
\end{tabular}\end{center}
The notation $\emptyset$ here indicates not to include a corresponding term in the formal sum; equivalently,
$L(s,\emptyset)=1$.

As a final example, consider the symmetric square representation $Sym^2$ that maps $GL(n)\rightarrow
GL(\f{n(n+1)}{2})$.  It satisfies both rules (\ref{ext2ib}) and (\ref{esib}), with the substitution of $Sym^2$ for
$Ext^2$, and is completely determined by the table
\begin{center}\begin{tabular}{|c|ccc|}
\hline
 $\sigma$ &  $triv$ &  $\sgn$ & $D_k$ \\
\hline
 $Sym^2\sigma$ &  $triv$ &  $triv$ & $D_{2k-1}\boxplus\sgn^{k+1}$ \\
\hline
\end{tabular}\end{center}

To illustrate, we will conclude by explicitly calculating $L_\infty(s,\pi,Ext^2\otimes\chi)$ when $\pi$ is a
cuspidal automorphic representation of $GL(n)$ over $\Q$, and $\chi$ is a Dirichlet character. We write
$\pi_\infty$ as the isobaric sum
\begin{equation}\label{piinftydesc}
    \Pi \ \ = \ \ \( \isobar_{i=1}^{r_1} \sgn^{\e_i}[s_i]  \)  \, \isobar \,  \(
    \isobar_{j=1}^{r_2}D_{k_j}[s_{r_1+j}]\),
\end{equation}
as this is its most general form according to the description in \secref{langclasssection}.  The rules
(\ref{ext2ib}-\ref{esib}) show that
\begin{equation}\label{5easyrules}
    Ext^2\Pi \ \ = \ \ \Pi_1  \, \isobar \, \Pi_2 \, \isobar \,
    \Pi_3 \, \isobar \, \Pi_4 \, \isobar \, \Pi_5 \ ,
\end{equation}
where
\begin{equation}\label{4easypieces}
\aligned
    \Pi_1 \ \  &  = \ \ \isobar_{i=1}^{r_1} (Ext^2
    \sgn^{\e_i})[2s_i] \ \ & = &  \ \ \isobar_{i=1}^{r_1}
    \,\emptyset[2s_i]\ = \ \emptyset
    \, ,
    \\
    \Pi_2  \ \ & = \ \ \isobar_{j=1}^{r_2} (Ext^2
    D_{k_j})[2s_{r_1+j}] \ \ & = &  \ \  \isobar_{j=1}^{r_2}\sgn^{k_j}[2s_{r_1+j}]  \, ,
    \\
    \Pi_3 \ \ & = \ \ \isobar_{\stackrel{\scriptstyle{i \le r_1}}{j \le r_2}}
    (\sgn^{\e_i}\times D_{k_j})[s_i+s_{r_1+j}]
    \ \ & = & \ \  \isobar_{\stackrel{\scriptstyle{i \le r_1}}{j \le r_2}}
    D_{k_j}[s_i+s_{r_1+j}]   \, ,
    \\
    \Pi_4 \ \ & = \ \ \isobar_{1 \le i < k \le r_1}
    (\sgn^{\e_i}\times \sgn^{\e_k})[s_i+s_k] \ \ & = &
     \ \ \isobar_{1 \le i < k \le r_1}
    \sgn^{\e_i+\e_k}[s_i+s_k]  \,,
\endaligned
\end{equation}
and
\begin{equation}\label{5thrule}
\aligned
     \Pi_5 \ \ & = \ \ \isobar_{1 \le j < \ell \le r_2}
    (D_{k_j}\times D_{k_\ell})[s_{r_1+j}+s_{r_1+\ell}]  \ \
     \qquad\qquad\qquad
    \qquad \ \ \ \ \qquad  \\ & = \ \
\isobar_{1 \le j < \ell \le r_2}\,  \(
D_{k_j+k_\ell-1\,}[s_{r_1+j}+s_{r_1+\ell}]  \, \isobar \,
D_{|k_j-k_\ell|+1\,}[s_{r_1+j}+s_{r_1+\ell}]\)
    .
\endaligned
\end{equation}
If we choose $\e_{ik}$ and $\e_j' \in \{0,1\}$ to be congruent to $\e_i+\e_k$ and $k_j$ modulo 2, respectively,
then
\begin{equation}\label{fivelpieces}
\aligned L(s,\Pi_1) \ \  & = \ \ \  \,  \ \ 1 \, ,  \\
L(s,\Pi_2) \ \ & = \ \ \  \, \  \prod_{j=1}^{r_2} \G_\R(s+2s_{r_1+j}+ \e_j')\, ,  \\
L(s,\Pi_3)\ \ & = \ \  \  \ \, \prod_{\stackrel{\scriptstyle{i \le
r_1}}{j \le r_2}}
\G_\C(s+s_i+s_{r_1+j}+\textstyle{\f{k_j-1}{2}}) \, ,  \\
L(s,\Pi_4)\ \  & = \ \  \prod_{1 \le i < k \le
r_1}\G_\R(s+s_i+s_k+\e_{ik}) \, , \ \ \ \
\text{and}  \\
L(s,\Pi_5)\ \  & =\\ &\!\!\!\!\!\!  \prod_{1 \le j < \ell \le r_2}
\G_\C(s+s_{r_1+j}+s_{r_1+\ell}+\textstyle{\f{k_j+k_\ell-2}{2}})
\G_\C(s+s_{r_1+j}+s_{r_1+\ell}+\textstyle{\f{|k_j-k_\ell|}{2}})
.  \\
\endaligned
\end{equation}
Consequently, $L_\infty(s,\pi,Ext^2)=L(s,Ext^2\Pi)=L(s,\Pi_2)L(s,\Pi_3)L(s,\Pi_4)L(s,\Pi_5)$ is the product of
these factors.

The archimedean component $\chi_\infty$ of the character $\chi$ is $\sgn^\eta$, where $\eta$ is the parity
parameter of $\chi$ defined by $\chi(-1)=(-1)^\eta$.  The isobaric decomposition of
\begin{equation}\label{ext2tenschiisobar1}
Ext^2\pi_\infty\otimes\chi_\infty \ \ = \ \ \isobar_{j=1}^5\(\Pi_j\otimes\sgn^\eta\)
\end{equation}
 may be computed using the tensoring rules above.  These imply that $\Pi_1$, $\Pi_3$, and $\Pi_5$ are unchanged by tensoring with $\chi_\infty$, and that $\Pi_2$ and $\Pi_4$ change by adding $\eta$ to their exponents of $\sgn$.  The result is that
 \begin{equation}\label{ext2tenschiisobar2}
 L_\infty(s,\pi,Ext^2\otimes\chi) \ \ = \ \ L(s,\Pi_1)L(s,\Pi_2\otimes\sgn^\eta)L(s,\Pi_3)L(s,\Pi_4\otimes\sgn^\eta) L(s,\Pi_5)\,,
 \end{equation}
 where
 \begin{equation}\label{twotwistedpieces}
    \aligned
    L(s,\Pi_2\otimes\sgn^\eta) \ \ & = \ \ \  \, \  \prod_{j=1}^{r_2} \G_\R(s+2s_{r_1+j}+ \e_{j\eta}')\, ,  \\
    L(s,\Pi_4\otimes\sgn^\eta)\ \  & = \ \  \prod_{1 \le i < k \le
r_1}\G_\R(s+s_i+s_k+\e_{ik\eta}) \, , \\
    \endaligned
 \end{equation}
 and $\e_{j\eta}'$ and $\e_{ik\eta}\in\{0,1\}$ are congruent to $\e_j'+\eta\equiv k_j+\eta$
and $\e_{ik}+\eta\equiv \e_i+\e_k+\eta \pmod 2$, respectively.


\begin{bibsection}


\begin{biblist}

\bib{boreljacquet}{article}{
   author={Borel, A.},
   author={Jacquet, H.},
   title={Automorphic forms and automorphic representations},
   note={With a supplement ``On the notion of an automorphic
   representation''\ by R. P. Langlands},
   conference={
      title={Automorphic forms, representations and $L$-functions (Proc.
      Sympos. Pure Math., Oregon State Univ., Corvallis, Ore., 1977), Part
      1},
   },
   book={
      series={Proc. Sympos. Pure Math., XXXIII},
      publisher={Amer. Math. Soc.},
      place={Providence, R.I.},
   },
   date={1979},
   pages={189--207},
   review={\MR{546598 (81m:10055)}},
}


\bib{bumprssurvey}{article}{
    author={Bump, Daniel},
     title={The Rankin-Selberg method: a survey},
 booktitle={Number theory, trace formulas and discrete groups (Oslo, 1987)},
     pages={49\ndash 109},
 publisher={Academic Press},
     place={Boston, MA},
      date={1989},
}

\bib{bumpfriedberg}{article}{
    author={Bump, Daniel},
    author={Friedberg, Solomon},
     title={The exterior square automorphic $L$-functions on ${\rm GL}(n)$},
 booktitle={Festschrift in honor of I. I. Piatetski-Shapiro on the occasion
            of his sixtieth birthday, Part II (Ramat Aviv, 1989)},
    series={Israel Math. Conf. Proc.},
    volume={3},
     pages={47\ndash 65},
 publisher={Weizmann},
     place={Jerusalem},
      date={1990},
}

\bib{bumpginzburg}{article}{
    author={Bump, Daniel},
    author={Ginzburg, David},
     title={Symmetric square $L$-functions on ${\rm GL}(r)$},
   journal={Ann. of Math. (2)},
    volume={136},
      date={1992},
    number={1},
     pages={137\ndash 205},
      issn={0003-486X},
}




\bib{Casselman:1980}{article}{
     author={Casselman, W.},
      title={Jacquet modules for real reductive groups},
  booktitle={Proceedings of the International Congress of Mathematicians (Helsinki, 1978)},
      pages={557\ndash 563},
  publisher={Acad. Sci. Fennica},
      place={Helsinki},
       date={1980},
}


\bib{Casselman:1989}{article}{
     author={Casselman, W.},
      title={Canonical extensions of Harish-Chandra modules to representations of $G$},
    journal={Canad. J. Math.},
     volume={41},
       date={1989},
     number={3},
      pages={385\ndash 438},
}


\bib{chm}{article}{
    author={Casselman, William},
    author={Hecht, Henryk},
    author={Mili{\v{c}}i{\'c}, Dragan},
     title={Bruhat filtrations and Whittaker vectors for real groups},
 booktitle={The mathematical legacy of Harish-Chandra (Baltimore, MD, 1998)},
    series={Proc. Sympos. Pure Math.},
    volume={68},
     pages={151\ndash 190},
 publisher={Amer. Math. Soc.},
     place={Providence, RI},
      date={2000},
}



\bib{ginzburge6}{article}{
   author={Ginzburg, David},
   title={On standard $L$-functions for $E_6$ and $E_7$},
   journal={J. Reine Angew. Math.},
   volume={465},
   date={1995},
   pages={101--131},
   issn={0075-4102},
   review={\MR{1344132 (96m:11040)}},
   doi={10.1515/crll.1995.465.101},
} 	



\bib{green}{article}{
author={Green, Michael},
author={Miller, Stephen D.},
author={Russo, Jorge},
author={Vanhove, Pierre},
title={Eisenstein series for higher rank groups and string theory amplitudes},
note={\url{http://arxiv.org/abs/1004.0163}},year={2010}}


\bib{jseuler}{article}{
   author={Jacquet, H.},
   author={Shalika, J. A.},
   title={On Euler products and the classification of automorphic
   representations. I},
   journal={Amer. J. Math.},
   volume={103},
   date={1981},
   number={3},
   pages={499--558},
   issn={0002-9327},
   review={\MR{618323 (82m:10050a)}},
   doi={10.2307/2374103},
}


\bib{jsextsq}{article}{
    author={Jacquet, Herv{\'e}},
    author={Shalika, Joseph},
     title={Exterior square $L$-functions},
 booktitle={Automorphic forms, Shimura varieties, and $L$-functions, Vol.\
            II (Ann Arbor, MI, 1988)},
    series={Perspect. Math.},
    volume={11},
     pages={143\ndash 226},
 publisher={Academic Press},
     place={Boston, MA},
      date={1990},
    }


\bib{knapp}{book}{
    author={Knapp, Anthony W.},
     title={Representation theory of semisimple groups},
    series={Princeton Landmarks in Mathematics},
      note={An overview based on examples;
            Reprint of the 1986 original},
 publisher={Princeton University Press},
     place={Princeton, NJ},
      date={2001},
     pages={xx+773},
      isbn={0-691-09089-0},
}

\bib{knapparch}{article}{
   author={Knapp, Anthony W.},
   title={Local Langlands correspondence: the Archimedean case},
   conference={
      title={Motives},
      address={Seattle, WA},
      date={1991},
   },
   book={
      series={Proc. Sympos. Pure Math.},
      volume={55},
      publisher={Amer. Math. Soc.},
      place={Providence, RI},
   },
   date={1994},
   pages={393--410},
   review={\MR{1265560 (95d:11066)}},
}


\bib{kostant}{article}{
   author={Kostant, Bertram},
   title={On Whittaker vectors and representation theory},
   journal={Invent. Math.},
   volume={48},
   date={1978},
   number={2},
   pages={101--184},
   issn={0020-9910},
   review={\MR{507800 (80b:22020)}},
   doi={10.1007/BF01390249},
}




\bib{langlandsmarchen}{article}{
    author={Langlands, Robert P.},
     title={Automorphic representations, Shimura varieties, and motives. Ein
            M\"archen},
 booktitle={Automorphic forms, representations and $L$-functions (Proc.
            Sympos. Pure Math., Oregon State Univ., Corvallis, Ore., 1977),
            Part 2},
    series={Proc. Sympos. Pure Math., XXXIII},
     pages={205\ndash 246},
 publisher={Amer. Math. Soc.},
     place={Providence, R.I.},
      date={1979},
}


\bib{langlandsrealgroups}{article}{
    author={Langlands, Robert P.},
     title={On the classification of irreducible representations of real
            algebraic groups},
 booktitle={Representation theory and harmonic analysis on semisimple Lie
            groups},
    series={Math. Surveys Monogr.},
    volume={31},
     pages={101\ndash 170},
 publisher={Amer. Math. Soc.},
     place={Providence, RI},
      date={1989},
}


\bib{inforder}{article}{
        author={Miller, Stephen D.},
        author={Schmid, Wilfried},
        title={Distributions and analytic continuation of Dirichlet series},
    journal={J. Funct. Anal.},
        volume={214},
        date={2004},
        number={1},
        pages={155\ndash 220},
        issn={0022-1236},
 }

\bib{voronoi}{article}{
   author={Miller, Stephen D.},
   author={Schmid, Wilfried},
   title={Automorphic distributions, $L$-functions, and Voronoi summation
   for ${\rm GL}(3)$},
   journal={Ann. of Math. (2)},
   volume={164},
   date={2006},
   number={2},
   pages={423--488},
   issn={0003-486X},
   review={\MR{2247965}},
}


\bib{korea}{article}{
   author={Miller, Stephen D.},
   author={Schmid, Wilfried},
   title={The Rankin-Selberg method for automorphic distributions},
   conference={
      title={Representation theory and automorphic forms},
   },
   book={
      series={Progr. Math.},
      volume={255},
      publisher={Birkh\"auser Boston},
      place={Boston, MA},
   },
   date={2008},
   pages={111--150}
}

%

\bib{glnvoronoi}{article}{author={Miller, Stephen D.}, author={Schmid, Wilfried}, title={A general Voronoi
summation formula for $GL(n,\Z)$},book={title={Geometric analysis: Present and Future},series={Advanced Lectures in
Math series},publisher={International Press}},note={to appear}}



\bib{decay}{article}{author={Miller, Stephen D.}, author={Schmid, Wilfried},
        title={On the rapid decay of cuspidal automorphic forms}, note={preprint}, year={2010}}

\bib{pairingpaper}{article}{author={Miller, Stephen D.}, author={Schmid, Wilfried},
        title={Pairings of automorphic distributions}, note={preprint}, year={2010}}


\bib{extsqpaper}{article}{author={Miller, Stephen D.}, author={Schmid, Wilfried},
        title={The archimedean theory of the Exterior Square $L$-functions over $\Q$}, note={preprint}, year={2010}}


\bib{pioline}{article}{
   author={Obers, Niels A.},
   author={Pioline, Boris},
   title={Eisenstein series in string theory},
   note={Strings '99 (Potsdam)},
   journal={Classical Quantum Gravity},
   volume={17},
   date={2000},
   number={5},
   pages={1215--1224},
   issn={0264-9381},
   review={\MR{1764377 (2001g:81218)}},
   doi={10.1088/0264-9381/17/5/330},
}



\bib{pspatterson}{article}{
   author={Patterson, S. J.},
   author={Piatetski-Shapiro, I. I.},
   title={The symmetric-square $L$-function attached to a cuspidal
   automorphic representation of ${\rm GL}_3$},
   journal={Math. Ann.},
   volume={283},
   date={1989},
   number={4},
   pages={551--572},
   issn={0025-5831},
   review={\MR{990589 (90d:11070)}},
   doi={10.1007/BF01442854},
}



\bib{pscorvallis}{article}{
    author={Piatetski-Shapiro, I. I.},
     title={Multiplicity one theorems},
 booktitle={Automorphic forms, representations and $L$-functions (Proc.
            Sympos. Pure Math., Oregon State Univ., Corvallis, Ore., 1977),
            Part 1},
    series={Proc. Sympos. Pure Math., XXXIII},
     pages={209\ndash 212},
 publisher={Amer. Math. Soc.},
     place={Providence, R.I.},
      date={1979},
}


\bib{shalmulone}{article}{
    author={Shalika, J. A.},
     title={The multiplicity one theorem for ${\rm GL}\sb{n}$},
   journal={Ann. of Math. (2)},
    volume={100},
      date={1974},
     pages={171\ndash 193},
      issn={0003-486X},
}


\bib{tadic}{article}{
   author={Tadi{\'c}, Marko},
   title={$\widehat{\rm GL}(n,\C)$ and $\widehat{\rm GL}(n,\R)$},
   conference={
      title={Automorphic forms and $L$-functions II. Local aspects},
   },
   book={
      series={Contemp. Math.},
      volume={489},
      publisher={Amer. Math. Soc.},
      place={Providence, RI},
   },
   date={2009},
   pages={285--313},
   review={\MR{2537046}},
}


\bib{voganlarge}{article}{
   author={Vogan, David A., Jr.},
   title={Gel\cprime fand-Kirillov dimension for Harish-Chandra modules},
   journal={Invent. Math.},
   volume={48},
   date={1978},
   number={1},
   pages={75--98},
   issn={0020-9910},
   review={\MR{0506503 (58 \#22205)}}
}



\bib{vogan}{article}{
    author={Vogan, David A., Jr.},
     title={The unitary dual of ${\rm GL}(n)$ over an Archimedean field},
   journal={Invent. Math.},
    volume={83},
      date={1986},
    number={3},
     pages={449\ndash 505},
      issn={0020-9910},
}



\bib{Wallach:1983}{article}{
    author={Wallach, Nolan R.},
     title={Asymptotic expansions of generalized matrix entries of representations of real reductive groups},
 booktitle={Lie group representations, I (College Park, Md., 1982/1983)},
    series={Lecture Notes in Math.},
    volume={1024},
     pages={287\ndash 369},
 publisher={Springer},
     place={Berlin},
      date={1983},
}


\end{biblist}
\end{bibsection}
\end{document}